\documentclass[10pt]{amsart}
\usepackage{r3p5}

\title[A classification of complex rank $3$ vector bundles on $\CP^5$]{A classification of complex rank $3$ vector bundles on $\CP^5$}
\author{Morgan Opie}

\begin{document}
  \maketitle
\begin{abstract}
Given integers $a_1,a_2,a_3$, there is a complex rank $3$ topological bundle on $\CP^5$ with $i$-th Chern class equal to $a_i$ if and only if $a_1,a_2,a_3$ satisfy the Schwarzenberger condition. Provided that the Schwarzenberger condition is satisfied, we prove that the number of isomorphism classes of rank $3$ bundles $V$ on $\CP^5$ with $c_i(V)=a_i$ is equal to $3$ if $a_1$ and $a_2$ are both divisible by $3$ and equal to $1$ otherwise. 

This shows that Chern classes are incomplete invariants of topological rank $3$ bundles on $\CP^5$. To address this problem, we produce a universal class in the $\tmf$-cohomology of a Thom spectrum related to $BU{\hspace{-.6pt}}(3)$, where $\tmf$ denotes topological modular forms localized at $3$. From this class and orientation data, we construct a $\Z/3$-valued invariant of the bundles of interest and prove that our invariant separates distinct bundles with the same Chern classes.
\end{abstract}
\setcounter{tocdepth}{2}
\tableofcontents

\section{Introduction}\label{sec:intro}

Let $X$ be a finite-dimensional CW-complex. From the perspective of homotopy theory, a topological vector bundle of complex rank $r$ over a space $X$ is identified with a classifying map $X \to BU{\hspace{-.6pt}}(r)$. Topologically equivalent vector bundles over $X$ correspond to homotopy equivalent maps to $BU{\hspace{-.6pt}}(r)$.

The integral cohomology of $BU{\hspace{-.6pt}}(r)$ is generated by universal Chern classes $c_1,\ldots c_r,$ with $c_i \in H^{2i}(BU{\hspace{-.6pt}}(r);\mathbb Z).$ These give rise to important invariants of complex bundles: the Chern classes of a bundle, defined for $V\: X \to BU{\hspace{-.6pt}}(r)$ as the pullbacks $$c_i(V):=V^*(c_i)\in H^{2i}(X;\mathbb Z).$$
In the case $X = \CP^n$, Chern classes are complete invariants of the {\em stable} equivalence class of the bundle. Explicitly, this means that $V\: \CP^n \to BU{\hspace{-.6pt}}(r)$ and $W\: \CP^n \to BU{\hspace{-.6pt}}(r')$ have the same Chern classes if and only if there exist integers $n,m$ greater than zero such that the bundles $V \oplus \underline \C^n$ and $W \oplus \underline \C^m$ are topologically equivalent. Here, $\underline \C$ is the trivial rank $1$ bundle on $X$.

This leads to the following fundamental question:
\begin{q}\label{q1} Are Chern classes sufficient to determine the (unstable) topological class of a complex rank $r$ vector bundle on $\CP^n$, up to topological equivalence? If not, what invariants beyond Chern classes are needed to distinguish such bundles?
\end{q} 

Rank $1$ bundles on all spaces $\CP^n$ are determined by their first Chern class \cite{AE,KS}. Rank $\geq n$ bundles on $\CP^n$ are also determined by their Chern classes. For $r$ strictly between $1$ and $n$, there is no uniform answer (although some patterns have been found when restricting to bundles with all Chern classes zero, see \cite{Hu}).

In \cite{AR}, Atiyah and Rees answer Question~\ref{q1} for complex rank $2$ topological bundles on $\CP^3$ by producing a $\Z/2$-valued invariant $\alpha$, which can be viewed as a characteristic class in the generalized cohomology of a classifying space.
\begin{thm}[{\cite[Theorem 2.8 and 3.3]{AR}}]\label{thm:AR} Given $a_1,a_2\in \Z$ with $a_1a_2\equiv 0 \pmod 2$, the number of isomorphism classes of rank $2$ bundles on $\CP^3$ with $i$-th Chern class $a_i$ is:
\begin{itemize}
\item equal to $2$ if $a_1\equiv 0 \pmod 2$; and
\item equal to $1$ otherwise.
\end{itemize}
In the first case, a rank $2$ vector bundle on $\CP^3$ is determined by $c_1,c_2,$ and $\alpha$. \end{thm}
\begin{rmk} The condition $a_1a_2\equiv 0 \pmod 2$ is necessary and sufficient for two integers to be the Chern classes of a rank $2$ bundle on $\CP^3$.\end{rmk}

Rank $2$ bundles on $\CP^4$ are determined by their Chern classes \cite[Theorem 1]{Switzer2}, as are $3$ bundles on $\CP^4$. The latter was already known to experts but is a consequence of results in this paper (see Corollary~\ref{cor:rk3p4_final}). Rank $2$ bundles on $\CP^5$ are not determined by their Chern classes, and are enumerated by Switzer \cite[Theorem 4]{Switzer2}. Given that the number rank $2$ bundles on $\CP^5$ with fixed Chern classes is known, one might hope for an invariant that distinguishes non-isomorphic bundles with the same stable class. We provide some partial results about this in Subsection~\ref{subsec:rank_2}. 

We focus on rank $3$ bundles on $\CP^5$ because this case is structurally analogous to the case of rank $2$ bundles on $\CP^3$.
Atiyah and Rees' work shows that the classification of rank $2$ bundles on $\CP^3$ is a $2$-primary problem.  There are similarities between the $2$-primary homotopical structure of $BU{\hspace{-.6pt}}(2)$ and the $3$-primary homotopical structure of $BU{\hspace{-.6pt}}(3) $, which we show result in similarities between the classification of rank $2$ bundles on $\CP^3$ and of rank $3$ bundles on $\CP^5$. We will elaborate on this analogy in order to answer Question~\ref{q1} for rank $3$ bundles on $\CP^5$. 

We answer Question~\ref{q1} by defining a $\Z/3$-valued invariant $\rho$ of such bundles and proving the following:

\begin{thm}\label{thm:combined} Given $a_1,a_2,a_3 \in \mathbb Z$ satisfying the Schwarzenberger condition $S_5$ (see Lemma~\ref{lem:explicit_S5}), the number of  isomorphism classes of rank $3$ bundles on $\CP^5$ with $i$-th Chern class equal to $a_i$ is:
\begin{itemize}
\item equal to 3 if $a_1\equiv 0 \pmod 3$ and $a_2\equiv 0 \pmod 3$; and
\item equal to 1 otherwise.
\end{itemize}
In the first case, a rank $3$ bundle on $\CP^5$ is determined by $c_1,c_2,c_3$ and $\rho$.
\end{thm}

\begin{rmk}\label{rmk:S5_nec_suff} In Subsection~\ref{subsec:schwarzenberger}, we show that the Schwarzenberger condition $S_5$ is necessary and sufficient for three integers to be the Chern classes of a rank $3$ bundle on $\CP^5$. We also give $S_5$ explicitly in this section. \end{rmk}

A priori, there is no simple geometric relationship between topologically distinct bundles with the same Chern classes. However, in both the case of rank $2$ bundles on $\CP^3$ and the case of rank $3$ bundles on $\CP^5$, any two bundles with the same Chern classes differ by an explicit action defined as follows.
\begin{const}\label{const:action_Z3} Associated to an inclusion $D^{2n} \hookrightarrow \CP^{n}$ of a disk in the top cell of $\CP^{n}$, we define $$Q\:\CP^n \to S^{2n} \vee \CP^n$$ by collapsing the boundary of $D^{2n}$ to a point.
Given vector bundles
$V\: \CP^n \to BU{\hspace{-.6pt}}(n)$ and $\sigma\: S^{2n} \to BU{\hspace{-.6pt}}(r)$ we define
$$\sigma V:= (\sigma \vee V )\circ Q\: \CP^n \to BU{\hspace{-.6pt}}(r).$$
Diagrammatically:
\begin{center}
\begin{tikzcd}[row sep=1.5em, column sep=3.5em]
& S^{2n} \arrow[d,"\iota_1" left]\arrow[dr, "\sigma"] \\
\CP^n \arrow[r,"Q"] & S^{2n} \vee \CP^n \arrow[r, "\sigma \vee V" near start] & BU{\hspace{-.6pt}}(r)\\
& \CP^n \arrow[u, "\iota_2" left] \arrow[ur, "V" below] &
\end{tikzcd}
\end{center}
where $\iota_1$ and $\iota_2$ are the standard maps of the summands into the wedge. The association $$(\sigma, V) \mapsto \sigma V$$ defines an action of $\pi_{2n}BU{\hspace{-.6pt}}(r)$ on equivalence classes of rank $r$ vector bundles over $\CP^{n}$.
\end{const}
This action preserves Chern classes provided that $n>r$. Therefore, if nontrivial, this action gives topologically distinct bundles with the same Chern classes.

In the case that $r=2$ and $n=3$, the action of $\pi_6BU{\hspace{-.6pt}}(2)\simeq \Z/2$ on rank $2$ bundles on $\CP^3$ with fixed Chern data is transitive, and is free if and only if $c_1 \equiv 0 \pmod 2$. This aligns with the enumeration of bundles with fixed Chern data in Theorem~\ref{thm:AR} above. The theorem below shows that the role of Construction~\ref{const:action_Z3} in analyzing rank $3$ bundles on $\CP^5$ is analogous.
\begin{theorem}\label{classification}
Let $a_1,$ $a_2$, and $a_3$ be integers satisfying $S_5$. Let $\V_{a_1,a_2,a_3}$ be the set of topological isomorphism classes of rank $3$ bundles on $\CP^5$ with $i$-th Chern class equal to $a_i$. Then:
\begin{enumerate}
\item The action of $\pi_{10}BU{\hspace{-.6pt}}(3) \simeq \Z/3$ on rank $3$ bundles over $\CP^5$, as given in Construction~\ref{const:action_Z3}, induces a transitive action on $\V_{a_1,a_2,a_3}.$
\item If $a_1$ or $a_2$ is nonzero $\bmod$ $3$, then the action of $\pi_{10}BU{\hspace{-.6pt}}(3)$ on $\V_{a_1,a_2,a_3}$ is trivial.
\item If $a_1$ and $a_2$ are zero $\bmod$ $3$, then the action of $\pi_{10}BU{\hspace{-.6pt}}(3)$ on $\V_{a_1,a_2,a_3}$ is free.
\end{enumerate}
\end{theorem}
This refines the enumeration result in Theorem~\ref{thm:combined}.

Theorem~\ref{classification} says that, if $a_1,a_2,a_3$ satisfy $S_5$, $a_1\equiv 0 \pmod 3$, and $a_2\equiv 0\pmod 3$, then the set of complex rank $3$ topological vector bundles on $\CP^5$ with $i$-th Chern class $a_i$ is a torsor for $\Z/3$. The main goal of the rest of the paper is to trivialize this torsor via a bundle invariant. To explain our approach to defining such an invariant for rank $3$ bundles on $\CP^5$, we discuss the $\alpha$-invariant of rank $2$ bundles on $\CP^3$ in greater detail.

The Atiyah--Rees invariant $\alpha$ is initially defined for rank $2$ bundles with $c_1 = 0$. Such bundles are classified by maps to $BSU{\hspace{-.6pt}}(2)$, allowing an invariant to be defined via a universal class in the generalized cohomology of $BSU{\hspace{-.6pt}}(2)$ rather than $BU{\hspace{-.6pt}}(2)$. Atiyah and Rees give a class $\alpha\in KO^{4}(BSU{\hspace{-.6pt}}(2))$, where $KO$ denotes real $K$-theory. They define the $\alpha$-invariant of $V\: \CP^3 \to BSU{\hspace{-.6pt}}(2)$ as $$\alpha(V):=p_*V^*(\alpha)\in KO^{-2}(\text{point})\simeq \Z/2,$$
where $V^*$ is pullback with respect to $V$ and $p_*\: KO^*(\CP^3)\to \KO^{*-6}(\text{point})$ is the $KO$-theory pushforward for the spin manifold $\CP^3$.
They extend $\alpha$ to bundles with $c_1(V)\equiv 0 \pmod 2$ by $$\alpha(V):= \alpha
\left(V \otimes \mathcal O(-\frac{c_1(V)}{2})\right).$$

Alternatively, the Atiyah--Rees invariant can be rephrased as a {\em twisted characteristic class}. Recall that, given a virtual bundle $W$ over a space $X$, the Thom spectrum of $X$ with respect to $W$, written $\Th{X}{W}$, can be viewed as a twisted version of the suspension spectrum $\suspp X$. By a twisted characteristic class, we will mean a class in some generalized cohomology of a Thom spectrum over a classifying space. Let $\gamma_2$ denote the universal bundle on $BU(2)$. One can show that there is a class 
\[\tilde \alpha\in KO^*(\Th{BU{\hspace{-.6pt}}(2)}{\smallminus\gamma_2})\] which extends $\alpha$ in a precise sense. Given any rank $2$ vector bundle on $\CP^3$, the pullback of $\tilde \alpha$ gives a class
$$\tilde \alpha(V):=V^*{\tilde \alpha}\in KO^4(\Th{\CP^3}{\smallminus V}).$$
If $c_1(V)\equiv 0\pmod 2,$ $V$ is canonically $KO$-oriented, yielding a $KO$-Thom isomorphism $$KO^*(\Th{\CP^3}{\smallminus V})\simeq KO^*(\suspp \CP^3).$$ We can thus define
$$\alpha'(V)=p_*(\tilde\alpha(V))\in KO^{-2}(\text{point})\simeq \Z/2,$$ where as before $p_*$ is the $KO$-theory pushforward.
The invariant $\alpha'$ also distinguishes rank $2$ bundles on $\CP^3$ with $c_1\equiv 0 \pmod 2$ and agrees with the original $\alpha$-invariant when $c_1(V)=0$.

The insight here is that both $BSU{\hspace{-.6pt}}(2)$ and $\Th{BU{\hspace{-.6pt}}(2)}{\smallminus\gamma_2}$ stabilize $\pi_6BU{\hspace{-.6pt}}(2)$, in the following sense. While $\pi_6BU{\hspace{-.6pt}}(2)\simeq \Z/2,$ the stable homotopy group $\pi_6\left(\susp BU{\hspace{-.6pt}}(2)\right)$ is trivial, so bundles differing by an element in the unstable group $\pi_6BU{\hspace{-.6pt}}(2)$ cannot be distinguished by a characteristic class in the generalized cohomology of $BU{\hspace{-.6pt}}(2)$ itself. However, both $\pi_6\left( \susp BSU{\hspace{-.6pt}}(2)\right)$ and $\pi_6\Th{BU{\hspace{-.6pt}}(2)}{\smallminus\gamma_2}$ are nontrivial and are canonically isomorphic to $\pi_6BU{\hspace{-.6pt}}(2)$, permitting their $KO$-cohomology to supply the classes $\alpha$ and $\alpha'$, respectively.

Recalling Theorem~\ref{classification}, the relevant group for understanding rank $3$ bundles on $\CP^5$ is $\pi_{10}BU{\hspace{-.6pt}}(3)$. We also find that $\pi_{10}BU{\hspace{-.6pt}}(3)\simeq \Z/3$ is stably trivial. By the above discussion, we might attempt to classify rank $3$ bundles on $\CP^5$ by first stabilizing $\pi_{10}BU{\hspace{-.6pt}}(3)$ and then detecting it with some generalized cohomology theory.
Indeed, our strategy to define an invariant of rank $3$ bundles on $\CP^5$ is as follows:
\begin{itemize}
\item We identify a Thom spectrum related to $BU{\hspace{-.6pt}}(3)$ which stabilizes $\pi_{10}BU{\hspace{-.6pt}}(3)$ (Introduction to Section~\ref{sec:exist});
\item We define a twisted characteristic class in an appropriate generalized cohomology of this Thom spectrum, with certain key properties (Section~\ref{sec:exist}); and
\item We show that our twisted characteristic class can be resolved to an honest invariant $\rho$, via orientation data, and that this invariant distinguishes vector bundles with the same Chern data (Section~\ref{sec:untwisting}).
\end{itemize}

The main result of Section~\ref{sec:exist} can be stated as follows:
\begin{thm}\label{thm:imprecise_existence} Let $\BUc$ be the homotopy fiber of $c_1{\hspace{-.6pt}}\pmod 3\: BU{\hspace{-.6pt}}(3) \to K(\Z/3,2)$. Let $\tmf$ denote the $3$-localization of the spectrum of topological modular forms, obtained by inverting all primes other than $3$. There is a
class $$\tilde \rho \in \tmf^{-3}(\sk^{26}\BUct)$$ such that the pullback of $\tilde \rho$ with respect to the Thomificiation of a generator for $\pi_{10}\BUc$ induces an isomorphism
\begin{equation}\label{eq:isom_pi_10}\pi_{10}\BUc \simeq \pi_{13}\tmf.\end{equation}
\end{thm}
The class $\tilde \rho$ and isomorphism \eqref{eq:isom_pi_10} tell us that the cohomology theory $\tmf$ stably detects $\pi_{10}BU{\hspace{-.6pt}}(3)$ and therefore retains information about the bundles of interest. Under pullback, the class $\tilde\rho$ together with Thom isomorphisms determined by orientation data give rise to the invariant $\rho$ of Theorem~\ref{thm:combined}, as follows.

Theorem~\ref{thm:imprecise_existence} gives an association
\begin{equation}\label{eq:twisted_invariant_bad} V \mapsto \operatorname{Th}(V)^*(\tilde \rho) \in \tmf^{-3}(\Th{\CP^5}{\smallminus V}),\end{equation}
where $\Thom(V)$ denotes the Thomification of the classifying map $V\: \CP^5 \to BU{\hspace{-.6pt}}(3)$. Equation~\eqref{eq:twisted_invariant_bad} does not define an invariant of $V$ because the target depends on $V$ itself.
However, vector bundles with $c_1\equiv 0 \pmod 3$ and $c_2\equiv 0 \pmod 3$ are $\tmf$-orientable and therefore admit a $\tmf$-Thom isomorphisms $\tmf^*(\Th{\CP^5}{\smallminus V})\simeq \tmf^*(\suspp \CP^5)$. The problem is not quite solved: we need a consistent way of choosing Thom isomorphisms. This is the main project of Section~\ref{sec:untwisting}, which involves a detailed study of $\tmf$-orientations for the relevant bundles. The problem cannot be reduced to a known orientation problem (e.g. using the celebrated string orientation for topological modular forms \cite{AHR}).

\subsecl{Paper outline}{subsec:outline}

The proof of Theorem~\ref{classification} is the main project of Section~\ref{sec:count} and proceeds via
analyses of the set of homotopy classes of maps from $\CP^5$ to $BU{\hspace{-.6pt}}(3)$ completed at the primes $3$ and $2$. These arguments are carried out in Subsections~\ref{subsec:post_BU3} and \ref{subsec:2local}, respectively, and involve obstruction-theoretic arguments. Subsection~\ref{subsec:technical_claims_1} proves claims used in Subsection~\ref{subsec:post_BU3}. In Subsection~\ref{subsec:schwarzenberger} we compute the Schwarzenberger condition explicitly and show that it is necessary and sufficient for three integers to be the Chern classes of a rank $3$ bundle on $\CP^5$.

In Subsection~\ref{subsec:start_existence_proof}, we outline our method to produce the class $\tilde \rho$ in the $\tmf$-cohomology of $\sk^{26}\BUct$. The remaining Subsections~\ref{subsec:proof_exists}, \ref{subsec:cohomology_BUct}, and \ref{subsec:proof_unique} supply the details of the proof, which includes a uniqueness result. This concludes the proof of Theorem~\ref{thm:imprecise_existence}.

In Subsection~\ref{subsec:background_orient}, we review the theory of Thom isomorphisms and orientations and establish notation. In Subsection~\ref{subsec:choiceOrientation}, we study orientations of rank $3$ bundles on $\CP^5$ with $c_1\equiv 0 \pmod 3$ and $c_2\equiv 0\pmod 3$ and isolate a desirable set of $\tmf$-orientations. Using this set of orientations, we can produce a well-defined invariant: in Section~\ref{subsec:rho_works}, we combine orientation data with $\tilde \rho$ to define the invariant $\rho$ of complex rank $3$ topological bundles on $\CP^5$ with $c_1\equiv 0\pmod 3$ and $c_2\equiv 0\pmod 3$. We prove that $\rho$ separates topological equivalence classes of rank $3$ bundles on $\CP^5$ with the same Chern data. This completes the proof of Theorem~\ref{thm:combined}.

The remaining subsections offer examples and suggest future directions. In Subsection~\ref{subsec:sum_line_bundles}, we show that $\rho(L^{\oplus 3})=0$ for $L$ a line bundle with $c_1(L)\equiv0\pmod 3$. We also state an additivity result for $\rho$ on sums of line bundles.
In Subsection~\ref{subsec:rank_2}, we show that the methods discussed in this paper also produce a $3$-local invariant of rank $2$ bundles.

\subsecl{Acknowledgements}{subsec:acknowledgements}

First and foremost I want to thank my PhD advisor, Mike Hopkins, for suggesting this project and for his immense support throughout my PhD program. I am also immensely grateful to Haynes Miller for his mentorship during my time in graduate school; and to both Haynes and Elden Elmanto for serving on my dissertation committee and offering feedback on my thesis write-up. After my move to UCLA, Mike Hill's guidance and encouragement -- mathematical and practical -- were invaluable for me while improving and revising my thesis. Hood Chatham and Jeremy Hahn were both extremely generous in offering specific suggestions for methods and strategies used in this paper. This work benefited greatly from my conversations with Aravind Asok, Lukas Brantner, Yang Hu, Brian Shin, Dev Sinha, Alexander Smith, and Dylan Wilson. I am also grateful to the referee for detailed and constructive feedback.

While working on this project, the author was supported by the National Science Foundation under Award No.~2202914.

\subsecl{Conventions}{subsec:conventions}
\begin{itemize}
\item ``Vector bundle" will refer to a complex, topological vector bundle. As such, we use ``vector bundle" and ``map to $BU{\hspace{-.6pt}}(r)$" interchangeably. ``Rank" refers to complex rank.
\item Given spaces $X$ and $Y$, we write $[X,Y]$ for homotopy classes of maps from $X$ to $Y$.
\item $H^*$ will refer to ordinary cohomology with $\Z$ coefficients, except otherwise stated. We write $\HF{p}{*}$ for cohomology with $\mathbb F_p=\Z/p$ coefficients.
\item If $C$ is a space, then $\pi_*C$ will refer to {\em unstable} homotopy groups. If $X$ is a spectrum, $\pi_*X$ will refer to its stable homotopy groups. Thus, the stable homotopy groups of a space $C$ will be written as $\pi_*(\suspp C)$.
\item Given a space or spectrum $X$, we write $\tau_{n}X$ for its $n$-th Postnikov section.
\item Given virtual complex $W\: X\to BU$ on a topological space $X$, we write $\Th{X}{W}$ for the Thom spectrum associated to the stable homotopy sphere bundle 
\[X \xrightarrow{W} BU \xrightarrow{j} BGL_1\sphere,\] where $j$ is the complex $j$-homomorphism (for a discussion of Thom spectra defined in this way, see \cite[Section 5.1]{ABGHR}). 

\item For $W\: X \to BU(r)$ a finite-rank bundle, let $i_r\: BU(r)\to BU$ denote the map induced on classifying spaces by the inclusion $U(r) \to U$. We define $\Th{X}{W}$ to be the Thom spectrum of the composite $i_r \circ W\: X \to BU$. 
\item Given a space or spectrum $X$, we write $\c{X}{p}$ for its completion at a prime $p$, and $X_{(p)}$ for its localization at $p$, i.e., where all primes other than $p$ are inverted.
\item In Sections~\ref{sec:exist} and \ref{sec:untwisting}, all spaces and spectra are implicitly localized at the prime $3$.
\item Given an $E_\infty$-ring spectrum $R$ and two $R$-modules $X,Y$, we write $$\operatorname{Maps}_R(X,Y)$$ for the space of $R$-module maps from $X$ to $Y$.

\end{itemize}

\subsecl{Declaration of generative AI use}{}

During the latest revision to this work, the author interacted with ChatGPT Sol 5.6 to review and simplify arithmetic in corrections to Subsection~\ref{subsec:schwarzenberger}, proof of Lemma~\ref{lem:explicit_S5}, and for general copy editing within the updated Subsection~\ref{subsec:schwarzenberger}. ChatGPT and generative AI were not used elsewhere in the article. After using this tool/service, the author reviewed and edited the content as needed and takes full responsibility for the content of the article.

\section{A count of rank $3$ bundles on $\CP^5$}\label{sec:count}
The primary goal of this section is to prove Theorem~\ref{classification} by computing the set $[\CP^5,BU{\hspace{-.6pt}}(3)]$. For this we need the homotopy of $BU{\hspace{-.6pt}}(3)$ through degree $10$, which can be computed via the fiber sequence
$$BSU{\hspace{-.6pt}}(3) \to BU{\hspace{-.6pt}}(3) \to BU{\hspace{-.6pt}}(1)$$
and its associated homotopy long exact sequence. The homotopy of $U(1) \simeq S^1$ is known and enough of the homotopy of $SU(3)$ is computed in \cite{MT}. We give the result in Figure~\ref{fig:homotopy_BU3}.
\begin{figure}[h]
\begin{tabular}{| M{1.2cm} | M{1cm} | M{1cm} | M{1cm} | M{1cm} |M{1cm} | M{1cm} | M{1cm} | M{1cm} | M{1cm} | M{1cm} | N}
\hline
& \textbf{$\pi_2$} & \textbf{$\pi_3$} & \textbf{$\pi_4$} & \textbf{$\pi_5$} & \textbf{$\pi_6$} & \textbf{$\pi_7$} & \textbf{$\pi_8$} & \textbf{$\pi_9$} & \textbf{$\pi_{10}$} \\
\hline
& & & & & & & & &
\\[-8pt]
$BU{\hspace{-.6pt}}(3)$ & $\Z$ &0 & $\Z$ & 0 & $\Z$ & $\Z/6$ &0 & $\Z/12$ & $\Z/3$ \\[2pt]
\hline
\end {tabular}\caption{Homotopy of BU{\hspace{-.6pt}}(3)}\label{fig:homotopy_BU3}
\end{figure}

Note that the torsion in $\pi_nBU{\hspace{-.6pt}}(3)$ for $n\leq 10$ is either $2$- or $3$- primary. Thus we may break the computation into analyses at the primes 2 and 3.

The key tool that allows us to study the problem one prime at a time is the theory of rationalization and completion of spaces. Given a space $X$, let $\c{X}{p}$ denote its $p$-completion.
The Fracture Theorem for completion, as stated in \cite[Theorem 13.1.1]{MP}, implies the following:
\begin{lem}\label{cor:fracture} Let $X$ be a CW complex of real dimension at most $10$.  An element in $[X,BU{\hspace{-.6pt}}(3)]$ is the same data as pairs of maps
$$f_2\: X \to \c{BU{\hspace{-.6pt}}(3)}{2},\,\,\text{ }f_3\: X \to \c{BU{\hspace{-.6pt}}(3)}{3}$$
such that the $2$- and $3$-completed Chern classes are both in the image of the canonical map $H^*(X;\Z) \to \c{H^*(X;\Z)}{p}$ and agree under this identification.
\end{lem}

We take $X=\CP^5$ and apply the previous result. The in-depth analysis in the proof of Theorem~\ref{classification} is calculating $[\CP^5,\c{BU{\hspace{-.6pt}}(3)}{3}]$. This is carried out in Subsection~\ref{subsec:post_BU3}, with supporting technical results in Subsection~\ref{subsec:technical_claims_1}. In Subsection~\ref{subsec:2local}, we compute $[\CP^5,\c{BU{\hspace{-.6pt}}(3)}{2}]$. In Subsection~\ref{subsec:schwarzenberger}, we show that the Schwarzenberger condition $S_5$ is necessary and sufficient for three integers to be the Chern classes of a rank $3$ bundle on $\CP^5$. This completes the proof of Theorem~\ref{classification} and justifies Remark~\ref{rmk:S5_nec_suff}.

\subsecl{$3$-complete rank $3$ vector bundles on $\CP^5$}{subsec:post_BU3}

We give the first stages of a Postnikov-type tower for the $3$-completion of $BU{\hspace{-.6pt}}(3)$ and analyze maps from $\CP^5$ into this tower.

\begin{claim}\label{claim:princ_fib_BU3} There is a tower of principal fibrations given by the solid arrows below:
\begin{equation}\label{post1}\begin{tikzcd}[column sep=.5em, row sep= 2em]
K(\Z/3,10)\arrow[r] & P_{10} \arrow[d] \arrow[from=ddl,dashed, "\tau_{10}\,\,\,\,\,\,\," {near end, above}]\\
K(\Z/3,7) \times K(\Z/3,9)\arrow[r, crossing over] & P_9 \arrow[d] \arrow[r, "U"] & K(\Z/3,11) \\
BU{\hspace{-.6pt}}(3)\arrow[r,dashed, "{(c_1,c_2,c_3)}" {below}]\arrow[ur, dashed,"\tau_9" {near end, below}]
&K(\Z,2) \times K(\Z,4)\times K(\Z,6) \arrow[rr,"{k_7\times k_9}"] & & K(\Z/3, 8) \times K(\Z/3,10)
\end{tikzcd}\end{equation}
where $(c_1,c_2,c_3)$ is the product of first three Chern classes, $\tau_9$ is a choice of lift of $(c_1,c_2,c_3)$ to the space $P_9$, and $\tau_{10}$ is a choice of lift of $\tau_{9}$ to $P_{10}$.
Moreover, $(c_1,c_2,c_3)$ induces a $3$-complete equivalence after $6$-truncation; $\tau_9$ induces a $3$-complete equivalence after $9$-truncation; and $\tau_{10}$ induces a $3$-complete equivalence after $10$-truncation.
\end{claim}
At present, the explicit forms of $k_7 \times k_9$ and $U$ are not needed.

Given Claim~\ref{claim:princ_fib_BU3}, we can calculate $[\CP^5,\tau_{10}\c{BU{\hspace{-.6pt}}(3)}{3}] \simeq [\CP^5,\c{BU{\hspace{-.6pt}}(3)}{3}]$ by working up the tower. We need the following standard lemma.
\begin{lem}\label{lem:torsor} Let $X$ be a connected space. For any other space $Y$ let $Y^X$ denote the mapping space.
Given a fiber sequence of connected spaces \begin{equation}\label{eq:fib1} \begin{tikzcd} F \ar[r]& E \ar[r]& B.\end{tikzcd}\end{equation} and a map $f\: X \to E$ so that the composite map to $B$ is nullhomotopic, the set of homotopy classes of choices of lifts of $f$ to $F$ is a torsor for $\operatorname{coker}\big(\pi_1(E^X,f) \to \pi_1(B^X,0)\big)$.
\end{lem}

We apply Lemma~\ref{lem:torsor} to the diagram in Claim~\ref{claim:princ_fib_BU3}. Candidate Chern data $(a_1,a_2,a_2) \in H^{2}(\mathbb CP^5)\times H^{4}(\mathbb CP^5) \times H^{6}(\mathbb CP^5)$ lifts to $P_9$ if and only if $(k_7\times k_9)\circ \left( a_1,a_2,a_3\right) \simeq 0$. This gives a $\mod 3$ condition on Chern classes, which we do not compute since we recover the condition via different methods in Subsection~\ref{subsec:schwarzenberger}.

By Lemma~\ref{lem:torsor}, the number of lifts to $P_9$ are a torsor for a quotient of $$\pi_1\Big (\big(\K(\Z/3,8) \times K(\Z/3,10)\big)^{\CP^5}\Big) \simeq \HF{3}{7}(\CP^5) \times \HF{3}{9}(\CP^5) =0,$$ so when a lift exists it is unique. 

Note that the argument to this point also proves the following:

\begin{cor}\label{cor:rk3p4}Consider the map $c\: \c{BU{\hspace{-.6pt}}(3)}{3} \to K(\c{\Z}{3},2)\times K(\c{\Z}{3},4)\times K(\c{\Z}{3},6)$ given by the product $c_1\times c_2\times c_3$ of three-completed Chern classes. The induced $3$-complete Chern class maps  \[[\CP^4,\c{BU{\hspace{-.6pt}}(3)}{3}] \to H^2(\CP^4,\c{\Z}{3})\times H^4(\CP^4,\c{\Z}{3}) \times H^6(\CP^4,\c{\Z}{3})\] is injective.
\end{cor}

Resuming the argument for $\CP^5$, are no obstructions to lifting from $P_9$ to $P_{10}$, since $\HF{3}{11}(\CP^3)=0$. Choices of lift are a torsor for
$$ \operatorname{coker}\big( \pi_1({P_9}^{\CP^5}) \xrightarrow{U\circ -} \pi_1(K(\Z/3,11)^{\CP^5}) \big)$$
Since $\pi_1(K(\Z/3,11)^{\CP^5}) \simeq \pi_0(\K(\Z/3,10)^{\CP^5}))\simeq \HF{3}{10}(\CP^5) \simeq \Z/3,$ there are two possibilities that will depend on $a_1,a_2,a_3$: \begin{itemize}
\item The map is surjective, the cokernel is trivial, and there is a unique lift; or
\item The map is zero, the cokernel is $\Z/3$ and there are three lifts. \end{itemize}
To compute $\op{Im}(U\circ -),$ we consider a related problem. From the principal fibration \begin{equation}\label{eq:princ_fib_nice}K(\Z/3,7)\times K(\Z/3,9) \to P_9 \to K(\Z,2) \times K(\Z,4) \times
K(\Z,6),\end{equation}
we get an action of the fiber $\big(K(\Z/3,7) \times K(\Z/3,9)\big) \times P_9 \to P_9$. This gives an action
\begin{equation}\label{eq:action_KZ379}\pi_1\Big(K(\Z/3,7)^{\CP^5}\times K(\Z/3,9)^{\CP^5}\Big)\times \pi_1( P_9^{\CP^5})\to \pi_1(P_9^{\CP^5}). \end{equation}
\begin{claim}\label{claim:transitive_KZ379} The action given in Equation~\eqref{eq:action_KZ379} is transitive.
\end{claim}
Assuming this claim too, fix $a_1,a_2,a_3 \in \Z$ with $(k_7\times k_9)\circ\left( a_1,a_2,a_3\right)=0$. Consider the
diagram:
\begin{equation}\label{eq:image_BU3}
\begin{tikzcd}[column sep=.8in,row sep=.6in]
{}&P_9\arrow[r,"U"] & K(\Z/3,11)
\\
* \times \CP^5\arrow[ur,"{[a_1,a_2,a_3]}"]\arrow[r,"* \times 1"]
& S^1 \times \CP^5 \arrow[u,"a"]\arrow[r, "{(x\iota_1 t^3, y \iota_1 t^4,a)}" above]\arrow[ur, phantom,"\dagger" {near end, above}]\arrow[ul,phantom,"\star" near start]
& K(\Z/3,7)\times K(\Z/3,9) \times P_9, \arrow[ul,bend right=10,"m" above]\arrow[u,"m^*U" right]
\end{tikzcd}
\end{equation}
where only the triangles $\dagger$ and $\star$ commute. In the above:
\begin{enumerate}
\item $a\:S^1 \times \CP^3 \to \tau_5BU{\hspace{-.6pt}}(2)$ restricts to $[a_1,a_2,a_3]$ on $* \times \CP^3$.
\item $x$,$y \in \Z/3$ are arbitrary coefficients of the classes $\iota_1t^3$ and $\iota_1 t^4$, which are the natural generators of
\begin{align*} \HF{3}{7}(S^1 \times \CP^5) &\simeq \HF{3}{1}(S^1 \otimes \HF{3}{6}(\CP^5) \\
&\simeq \Z/3 \{ \iota_1\} \otimes \Z/3 \{ t^3\}\end{align*}
and
\begin{align*} \HF{3}{9}(S^1 \times \CP^5) &\simeq \HF{3}{1}(S^1) \otimes \HF{3}{8}(\CP^5) \\
&\simeq \Z/3 \{ \iota_1\} \otimes \Z/3 \{ t^4\}.\end{align*}
\end{enumerate}
Given Claim~\ref{claim:transitive_KZ379}, to compute $\op{Im}(U\circ -)$, it suffices to compute $m^*U\circ(x\iota_1t^3,y\iota_1t^4,a)$
as $(x,y)$ ranges over $\Z/3 \times \Z/3$. We will obtain formula for the difference \begin{equation}\label{eq:difference} m^*U\circ (x\iota_1t^3,y\iota_1t^5,a)-m^*U\circ
(0,0,a).\end{equation} Showing that $\op{Im}(U\circ -)$ is $\Z/3$ is equivalent to finding $x,y\in \Z/3$ so that the difference in Equation~\eqref{eq:difference} is nonzero.

\begin{claim}\label{claim:m_upper_U} The class $m^*U \in \HF{3}{11}\big(K(\Z/3,7) \times K(\Z/3,9)\times P_9\big)$ is given by
$$m^*U = U +P^1(\iota_7') -\iota_2\iota_9' +\iota_2^2\iota_7'-\iota_4\iota_7' \in \HF{3}{11}\big(K(\Z/3,7) \times K(\Z/3,9)\times P_9\big),$$ where $\iota_7'$ and $\iota_9'$ generate $\HF{3}{7}
\big(K(\Z/3,7)
\big)$ and $\HF{3}{9}\big(K(\Z/3,9)\big)$, respectively; and where $\iota_2, \iota_4$ are the images in $\HF{3}{*}(P_9)$ of generators in $\HF{3}{i}(K(\Z,2))$ and $\HF{3}{i}(K(\Z,4))$, respectively.
\end{claim}
Given Claim~\ref{claim:m_upper_U}, since $\iota_2$ pulls back to $c_1 \pmod 3$ in $\HF{3}{*}(\CP^5)$ and $\iota_4$ pulls back to $c_2 \pmod 3$, we see that
\begin{align*} m^*U\circ (x\iota_1t^3,y\iota_1t^4,a)-U^*m\circ (0,0,a) &= U+ xP^1(\iota_1t^3)-( a_1t)y\iota_1t^4\\ & \,\,\,\,\,\,\,\,\,\,\,\,\,\,\,\, +( a_1^2t^2)x\iota_1t^3-(a_2t^2)x\iota_1t^3 -U\\
&= -(ya_1-xa_1^2+xa_2) \iota_1t^5.
\end{align*}
The quantity $ya_1-xa_1^2+xa_2$ is zero $\mod 3$ for all choices of $x$ and $y$ if and only if $a_1$ and $a_2$ are both zero $\mod 3$.

Predicated on Claims~\ref{claim:princ_fib_BU3}, \ref{claim:transitive_KZ379}, and \ref{claim:m_upper_U}, we have shown:

\begin{prop}\label{prop:classification_rank2_CP3_post} Consider the map 
\[ U\circ -\:  \pi_1({P_9}^{\CP^5})\to \pi_1(K(\Z/3,11)^{\CP^5})\] given by precomposing $U$ from in Diagram~\eqref{post1}.
Given integers $(a_1,a_2,a_3)$, we associate map $a_i\: \CP^5 \to K(Z,2i)$ that is well-defined up to homotopy. Suppose that $(k_7\times k_9)\circ \left( a_1,a_2,a_3\right)\simeq 0$, where $k_7 \times k_9$ is again as in Diagram~\eqref{post1}. The number of rank $3$ vector bundles on $\CP^5$ with $i$-th Chern class equal to $a_i$ is a torsor for $\op{coker}(U\circ -)$. Moreover, the following two situations can occur:
\begin{itemize}
\item If either $a_1$ or $a_2$ are nonzero $\mod 3$, then the map $U\circ -$
 is surjective and, up to homotopy, there is a unique $3$-complete vector bundle with $i$-th Chern class $a_i$.
\item If $a_1 \equiv 0 \pmod 3$ and $a_2 \equiv 0 \pmod 3$, then the map $U\circ -$ is zero and there are three distinct homotopy classes of $3$-complete vector bundles with $i$-th Chern class $a_i$.
\end{itemize}
\end{prop}

\subsecl{Proof of technical claims}{subsec:technical_claims_1}

We now prove the claims ~\ref{claim:princ_fib_BU3}, \ref{claim:transitive_KZ379}, and \ref{claim:m_upper_U}, completing the proof of Proposition~\ref{prop:classification_rank2_CP3_post}.
\begin{proof}[Proof of Claim~\ref{claim:princ_fib_BU3}]
The pullback of the Chern class map $$c:=(c_1,c_2,c_3)\:BU{\hspace{-.6pt}}(3) \to K(\Z,2) \times K(\Z,4) \times K(\Z,6)$$ on $\mod 3$-cohomology is a $3$-complete equivalence through degree 6.
We correct the degree $8$ and degree $10$ cohomology terms simultaneously via a map
\begin{center}\begin{tikzcd}[column sep=3em]K(\Z,2) \times K(\Z,4)\times K(\Z,6) \arrow[r,"{k_7 \times k_9}"] & K(\Z/3, 8) \times K(\Z/3, 10),
\end{tikzcd}\end{center}
and a factorization
\begin{center}\begin{tikzcd}
& P_9\arrow[d,dashed] \\
BU{\hspace{-.6pt}}(3) \arrow[ru, dashed,"{\tau_9}"]\arrow[r, "c"] & K(\Z,2) \times K(\Z,4)\times K(\Z,6) \arrow[r,dashed,"{k_7 \times k_9}"] & K(\Z/3\Z, 8) \times K(\Z/3\Z,10),
\end{tikzcd}\end{center}
where $P_9 := \operatorname{hofib}(k_7 \times k_9)$, such that:
\begin{enumerate}
\item The map $(k_7\times k_9) \circ c$ is nullhomotopic; and
\item The lift $\tau_9\: BU{\hspace{-.6pt}}(3) \to P_9$ is $\mod 3$-cohomology isomorphism up to at least degree $10$ and therefore realizes the $9$-truncation of $BU{\hspace{-.6pt}}(3)$, up to $3$-completion.
\end{enumerate}
We complete (1) in Construction~\ref{const1} and (2) in Verification~\ref{verif1}.
\begin{const}\label{const1}\label{const2}
Let $P^i$ denote the $\mod 3$ Steenrod operation of degree $4i$. The first relation among Steenrod operations on Chern classes is $P^1$ on $c_2$: $P^1(c_2) = c_1^2 c_2 + c_2^2 - c_1c_3.$
Let $\iota_j$ denote a generator for $\HF{3}{j}(K(\Z,j))$ and take
$$k_7 := P^1 \iota_4 - \iota_2^2 \iota_4 - \iota_4^2 + \iota_2\iota_6 \in \HF{3}{8}\Big(K(\Z,2) \times K(\Z,4) \times K(\Z,6)\Z\Big).$$
We identify a candidate for $k_9$ by computing $P^1$ on $c_3$.
$P^1c_3= c_3(c_1^2+c_2)$ so let $$k_9 := P^1\iota_6-\iota_6\big(\iota_2^2+\iota_4\big)= P^1\iota_6-\iota_6\iota_2^2-\iota_6\iota_4 \in \HF{3}{10}\Big(K(\Z,2) \times K(\Z,4) \times K(\Z,6)\Z\Big).$$
\end{const}
\begin{verif}\label{verif1}
One computes the $\HZt$-cohomology of integral Eilenberg Mac Lane spaces, which can be computed directly from the path-loop fibration.
The main result we need is the following:
\begin{prop}\label{K2K4K6_cohomology} Let $\iota_j$ generate $\HF{3}{j}K(\Z,j)$ for $j=2,4,6.$ We can identify the multiplicative structure of $\HF{3}{*}\left(K(\Z,2) \times K(\Z,4) \times K(\Z,6)\right)$ through degree $11$ as follows:
$$\left(\HF{3}{*}\big(K(\Z,2) \times K(\Z,4) \times K(\Z,6)\big)\right)_{\leq 11} \simeq \left(\Z/3\Z[\iota_2, \iota_4,\iota_6, Y_8, W_{10}] \otimes \Lambda[N_9, S_{11}]\right)_{\leq 11}$$
where the subscript indicates the degree of the polynomial or exterior generator, the notation $(-)_{\leq 11}$ indicates that we quotient by all elements of degree at least $12$, and
\begin{align*}
Y_8&= P^1 \iota_4 & W_{10} &= P^1 \iota_6 \\
N_{9} &= \beta P^1 \iota_4 & S_{11} &= \beta P^1 \iota_6.
\end{align*}
\end{prop}
From the above, we can compute the Serre spectral sequence for the fibration $$K(\Z/3,7) \times K(\Z/3,9) \to P_9 \to K(\Z,2) \times K(\Z,4) \times K(\Z,6).$$
The $E_2$-page is given in Figure~\ref{ss1}. Moreover, if $\beta$ denotes the Bockstein power operation:
\begin{align*}
L_8&=\beta \iota_7
\\
R_{10}&= \beta \iota_9
\\ M_{11}&=P^1\iota_7.
\end{align*}
\begin{figure}[b]
\centering
\textbf{The $E_2$-page of a spectral sequence computing $\HF{3}{*}(P_9).$}\par\medskip
\begin{tikzpicture}
\matrix (m) [matrix of math nodes,
nodes in empty cells,nodes={minimum width=2.5ex,
minimum height=2.5ex,outer sep=-1pt},
column sep=.25ex,row sep=.25ex]{
11 & M_{11} & & & & & & & & & & && \\
10 & R_{10} & & & & & & & & & & && \\
9 & \iota_9 & & & & & & & & & & && \\
8 & L_8& & & & & & & & & & && \\
7 &\iota_7 & & & & & & & & & & && \\
6 & & & & & & & & & & & && \\
5 & & & & & & & & & & & && \\
4 & & & & & & & & & & & && \\
3 & & & & & & & & & & & && \\
2 & & & & & & & & & & & && \\
1 & & & & & & & & & & & && \\
0 & & & \iota_2 & & \iota_4 & &\iota_6 & & Y_8 & N_9 & W_{10} & S_{11} & \\
\quad\strut & 0 & 1 & 2 & 3 & 4 & 5 & 6 & 7 & 8 & 9 & 10 &11& \strut \\};
\draw[dashed] (m-1-2.north west) -- (m-13-14.south east);
\draw[thick] (m-1-1.east) -- (m-13-1.east) ;
\draw[thick] (m-13-1.north) -- (m-13-14.north) ;\end{tikzpicture}\caption{Only multiplicative generators for the $E_2$-page are indicated.}\label{ss1}\end{figure}
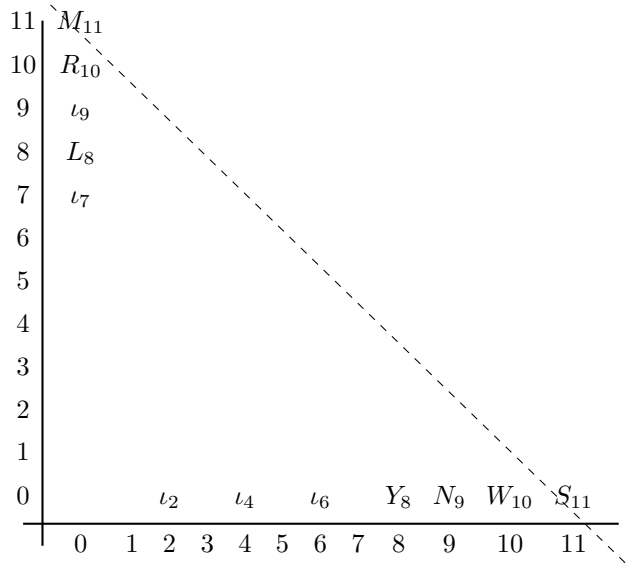
To obtain the associated graded of $\HF{3}{*}(P_9)$, we compute all relevant differentials using a combination of the following two facts (Kudo's transgression theorem, see \cite{Kudo56} or \cite[Ch. 6]{McCleary}):
\begin{itemize}
\item Given a principal fibration $F \to E \to K(\Z/p\Z,n)$, the fundamental class $\iota_{n+1}$ is transgressive in the mod $p$ Serre spectral sequence for $\Omega K(\Z/ p \Z,n) \to F \to E.$
\item A power operation applied to a transgressive class is transgressive; transgressions commute with power operations.
\end{itemize}
From the above items and the fact that $L_8 = \beta(\iota_7)$, we deduce that
\begin{align*} d_7 (\iota_7) &= Y_8- \iota_2^2\iota_4 - \iota_4^2+\iota_6\iota_2,\,\,\text{and}\\
d_8(L_8) &= \beta(d_7(\iota_7)) = \beta( Y_8- \iota_2^2\iota_4 - \iota_4^2+\iota_6\iota_2) = \beta(Y_8) = N_{9}.\end{align*}
Similarly, since $\beta(\iota_9)=R_{10}$, we get that
\begin{align*}d_{9}(\iota_9)&=W_{10}-\iota_6(\iota_2^2 -2\iota_4)\,\,\text{ and} \\
d_{10}(R_{10})&=\beta\big(W_{10}-\iota_6(\iota_2^2 +\iota_4)\big)= S_{11}.\end{align*}
All other terms strictly below the dotted line in Figure~\ref{ss1} are computed using the Leibniz rule. Thus the images of $\iota_2$, $\iota_4$, and $\iota_6$ are polynomial generators for $\HF{3}{*}(P_9)$ up to degree $10$; since
$c^*\: \HF{3}{2j} (K(\Z,2) \times K(\Z,4) \times K(\Z,6)) \to \HF{3}{*}(BU{\hspace{-.6pt}}(3)\Z)$ satisfies $\iota_{2j} \mapsto c_j,$
this shows that a lift of $(c_1,c_2,c_3)$ induces an equivalence through degree $9$, completing Verification~\ref{verif1}.
\end{verif}
Evidently, the next stage in the tower is given by a class $U\:P_9 \to K(\Z/3,11)$.
\end{proof}
\begin{proof}[Proof of Claim~\ref{claim:transitive_KZ379}] Consider the $\pi_1$ portion of the homotopy long exact sequence associated with the fibration \eqref{eq:princ_fib_nice}
\begin{equation}\label{keyLES2}
\begin{tikzcd}[column sep=.15in, row sep=.15in]
\pi_1\Big(\big(K(\Z/3, 7) \times K(\Z/3,9)\big)^{\CP^5}\Big) \arrow[d,"s"] \\
\pi_1(P_9^{\CP^5}) \arrow[d] \\
\pi_1\Big(\big(K(\Z,2) \times K(\Z,4) \times K(\Z,6)\big)^{\CP^5}\Big)
\end{tikzcd}
\end{equation}
where the basepoint for $\left(K(\Z,2) \times K(\Z,4) \times K(\Z,6)\right)^{\CP^5}$ is $(a_1, a_2,a_3)$.
The last term of \eqref{keyLES2} is zero, so $s$ is surjective. The action
\eqref{eq:action_KZ379} is given on $$(x,a)\in \pi_1\left(\big(K(\Z/3, 7) \times K(\Z/3,9)\big)^{\CP^5}\right)
\times \pi_1(P_9^{\CP^5})$$ by
\begin{equation}\label{eq:principal_action} (x,a)\,\, \mapsto \,\, s(x)a \in \pi_1(P_9^{\CP^5}).\end{equation}
Thus, the surjectivity of $s$ implies the action is transitive.
\end{proof}

\begin{proof}[Proof of Claim~\ref{claim:m_upper_U}]
To understand $U$ more explicitly, we study the spectral sequence in Figure~\ref{ss1} up to and including the dotted line. Computing differentials, we see that the class $M_{11} = P^1(\iota_7)$
detects a nonzero class in $\HF{3}{11}(P_9)$.

The action that \eqref{eq:principal_action} gives rise to a map of spectral sequences, from the Serre spectral sequence for
$$ K(\Z/3,7) \times K(\Z/3,9)\to P_9 \to K(\Z,2) \times K(\Z,4) \times K(\Z,6) $$
to the Serre spectral sequence for
\begin{equation}\label{eq:Serre1}\Big(K(\Z/3,7) \times K(\Z/3,9)\Big)^{\times 2}\to K(\Z/3,7) \times K(\Z/3,9) \times P_9 \to \prod_{i=1}^3 K(\Z,2i).\end{equation}
We compute this map of spectral sequences using the fiber-by-fiber action.

For $i=7$ and $i=9$, let $\iota_i$ and $\iota_i'$ generate the two copies of $\HF{3}{i}(K(\Z/3,i))$ in the fiber of Equation~\ref{eq:Serre1}. The comultiplication on the fiber implies that the coaction on the $E_2$-page is:
\begin{align*}\iota_7 &\mapsto \iota_7+\iota_7',& \iota_9 &\mapsto \iota_9 + \iota_9'.\end{align*}
We claim that, in the double complex of the source sequence, $U$ should be represented by
\begin{equation}\label{eq:U_really}M_{11} -\iota_4\iota_7+\iota_2^2\iota_7 -\iota_2\iota_9.\end{equation}
To see this, note that $M_{11}$ is transgressive and
\begin{align*}d_{11}(M_{11})&= d_{11}(P^1\iota_7)\\
&=P^1(d_7(\iota_7))\\
&=P^1(P^1\iota_4-\iota_2^2\iota_4 - \iota_4^2 + \iota_2\iota_6) \\
&= -P^2\iota_4+\iota_2^4\iota_4-\iota_2^2P^1\iota_4 +
\iota_4P^1\iota_4 + \iota_2^3\iota_6 +\iota_2P^1\iota_6 \\
&= -\iota_4^3+\iota_2^4\iota_4-\iota_2^2P^1\iota_4 +
\iota_4P^1\iota_4 + \iota_2^3\iota_6 +\iota_2P^1\iota_6 \\
&= \left(\iota_4(d_7\iota_7) +\iota_2^2\iota_4^2-\iota_2\iota_4\iota_6\right)+\iota_2^4\iota_4 -\iota_2^2P^1\iota_4+\iota_2^3\iota_6+\iota_2P^1\iota_6\\
&=\left( \iota_4(d_7\iota_7)-\iota_2^2d_7(\iota_7) +\iota_2^3\iota_6 \right)-\iota_2\iota_4\iota_6+\iota_2^3\iota_6+\iota_2P^1\iota_6\\
&= \iota_4(d_7\iota_7)-\iota_2^2d_7(\iota_7)-\iota_2\iota_4\iota_6-\iota_2^3\iota_6+\iota_2P^1\iota_6 \\
&= \iota_4(d_7\iota_7) -\iota_2^2(d_7\iota_7)+\iota_2(d_9\iota_9).\end{align*}
This indicates that a cocycle representative for $U$ in the double complex computing $H^*(P_9;\Z/2)$ is Equation~\eqref{eq:U_really} on the $E_2$-page.

Therefore, on the $E_2$-page, we see that the action on the class representing $U$ is detected by
\begin{align*}(M_{11} - \iota_4\iota_7+ \iota_2^2\iota_7 - \iota_2\iota_9 ) \xmapsto{m^*} (M_{11}- \iota_4\iota_7+\iota_2^2\iota_7 - \iota_2\iota_9 +P^1 \iota_7'-\iota_4\iota_7'+\iota_2^2\iota_7' -\iota_2\iota_9').\end{align*}
Passing to the $E_\infty$-page of the Serre spectral sequence for Equation~\eqref{eq:Serre1} we see that $m^*U$ is detected by \begin{align*}M_{11} + P^1 \iota_7'-\iota_4\iota_7' +\iota_2^2\iota_7'
-\iota_2\iota_9'\in \HF{3}{*}\big(K(\Z/3,7) \times K(\Z/3,9)\times P_9\big),\end{align*}
and therefore $m^*U=U + P^1 \iota_7'
-\iota_4\iota_7'+\iota_2^2\iota_7'-\iota_2\iota_9',$
completing the proof of the claim.
\end{proof}

\begin{rmk}\label{rmk:top_cell} Instead of analyzing the tower of Diagram~\ref{post1}, we could transpose over the skeleton-truncation adjunction in the homotopy, and instead build a map by working up the skeleton of $\CP^5$. The action of $\op{coker}(U\circ -)$ corresponds to the action of a quotient of $\pi_{10}\c{BU{\hspace{-.6pt}}(3)}{3}$ on lifts of a given map $\sk^9\CP^5 \to \c{BU{\hspace{-.6pt}}(3)}{3}$ to the $10$-skeleton. This shows that the action from Construction~\ref{const:action_Z3} is the relevant one.
\end{rmk}

\subsecl{$2$-complete rank $3$ vector bundles on $\CP^5$}{subsec:2local}

In this section, we show that any $2$-complete bundles on $\CP^5$ with the same $2$-complete Chern classes are isomorphic. More precisely:
\begin{prop}\label{prop:2local}Consider the map $c\: \c{BU{\hspace{-.6pt}}(3)}{2} \to K(\c{\Z}{2},2)\times K(\c{\Z}{2},4)\times K(\c{\Z}{2},6)$ given by the product $c_1\times c_2\times c_3$ of two-completed Chern classes. The induced $2$-complete Chern class maps  \[[\CP^5,\c{BU{\hspace{-.6pt}}(3)}{2}] \to H^2(\CP^5,\c{\Z}{2})\times H^4(\CP^5,\c{\Z}{2}) \times H^6(\CP^5,\c{\Z}{2})\] and

 \[[\CP^4,\c{BU{\hspace{-.6pt}}(3)}{2}]\to H^2(\CP^4,\c{\Z}{2})\times H^4(\CP^4,\c{\Z}{2}) \times H^6(\CP^4,\c{\Z}{2})\]
are injective.
\end{prop}
\begin{proof}First consider the statement for $\CP^5$. To understand $[\CP^5,\c{BU{\hspace{-.6pt}}(3)}{2}]$,
we build a map from $\CP^5$ into $\c{BU{\hspace{-.6pt}}(3)}{2}$ cell-by-cell. First,
recall the $2$-complete homotopy of $BU{\hspace{-.6pt}}(3)$, as in Figure~\ref{fig:htpy_BU3_2}, computed from Figure~\ref{fig:homotopy_BU3}.
\begin{figure}[h]
\begin{tabular}{| M{1.2cm} | M{1cm} | M{1cm} | M{1cm} | M{1cm} |M{1cm} | M{1cm} | M{1cm} | M{1cm} | M{1cm} | M{1cm} | N}
\hline
& \textbf{$\pi_2$} & \textbf{$\pi_3$} & \textbf{$\pi_4$} & \textbf{$\pi_5$} & \textbf{$\pi_6$} & \textbf{$\pi_7$} & \textbf{$\pi_8$} & \textbf{$\pi_9$} & \textbf{$\pi_{10}$} \\
\hline
& & & & & & & & &
\\[-8pt]
$\c{BU{\hspace{-.6pt}}(3)}{2}$ & $\c{\Z}{2}$ &0 & $\c{\Z}{2}$ & 0 & $\c{\Z}{2}$ & $\Z/2$ &0 & $\Z/4$ & $0$ \\[2pt]
\hline
\end {tabular}\caption{$2$-complete homotopy of BU{\hspace{-.6pt}}(3)}\label{fig:htpy_BU3_2}
\end{figure}
Consider the dotted arrows $(i)$ to $(v)$ in Diagram~\eqref{eq:skeletal} below, where $\sk^{2i}\CP^5\simeq \CP^i$ denotes the $i$-skeleton of $\CP^5$ for the standard cellular construction of $\CP^n$. Specifically, we recursively build $\CP^{n+1}$ as the cofiber of the Hopf map $S^{2n+1}\to \CP^n$.
\begin{equation}\label{eq:skeletal}
\begin{tikzcd}[row sep=.4cm]
\CP^5\arrow[from=d]\arrow[ddrr,bend left=32,dashed,"(v)" above]
& \\
\sk^8\CP^5\arrow[from=d]\arrow[drr,bend left=16,dashed,"(iv)" above]
& \\
\sk^6\CP^5\arrow[from=d]\arrow[rr,dashed,"(iii)"]
& &\c{BU{\hspace{-.6pt}}(3)}{2} \\
\sk^4\CP^5 \arrow[from=d]\arrow[urr,dashed,bend right=16,"(ii)" above]
&\\
\sk^2\CP^5\arrow[from=d]\arrow[uurr,bend right=32,dashed,"(i)" above]
&\\
*
\end{tikzcd}
\end{equation}
An arrow $(i)$ corresponds to a 2-complete first Chern class $\CP^2 \to K(\c{\Z}{2},2)$. The obstruction to lifting further is in $\pi_3(\c{BU{\hspace{-.6pt}}(3)}{2})=0$.
The choices of lifts to an arrow $(ii)$ are acted on transitively by
$\pi_3(\c{BU{\hspace{-.6pt}}(3)}{2}) \simeq \c{{\Z}}{2}$ and correspond to
$c_2$. The obstructions to lifting to an arrow $(iii)$ lie in $\pi_5(\c{BU{\hspace{-.6pt}}(3)}{2})=0$, and the choices of lift to $(iii)$ correspond to $c_3$.

The obstruction to a lift to a map $(iv)$ lies in $\pi_7(\c{BU{\hspace{-.6pt}}(3)}{2}) \simeq \Z/2.$
The choices of lifts are acted on transitively by $\pi_8(\c{BU{\hspace{-.6pt}}(3)}{2})\simeq 0.$ The obstruction to lifting from $(iv)$ to $(v)$ are in $\pi_9(\c{BU{\hspace{-.6pt}}(3)}{2})\simeq \Z/4$. The choices of lift are acted upon transitively by $\pi_{10}(\c{BU{\hspace{-.6pt}}(3)}{2})=0.$

Since $\sk^8 \CP^5 =\CP^4$, the above argument also shows that $2$-local rank $3$ vector bundles on $\CP^4$ are determined by their Chern classes.
\end{proof}
\begin{rmk} We have shown that there are $\mod 2$ and $\mod 4$ conditions on the Chern classes of a rank $3$ vector bundle on $\CP^5$, but no new $2$-primary invariants. However, for a general $10$-skeletal space (one that is not even), there may be additional 2-complete bundles not determined by Chern classes.
\end{rmk}
Combining Lemma~\ref{cor:fracture}, Corollary~\ref{cor:rk3p4}, and Proposition~\ref{prop:2local} we obtain the following:
\begin{cor}\label{cor:rk3p4_final} A rank complex $3$ bundle on $\CP^4$ is determined by its Chern classes.
\end{cor}

\subsecl{The Schwarzenberger conditions}{subsec:schwarzenberger}

We discuss necessary conditions for a collection of integers $a_1,a_2, a_3\in \Z$ to be the Chern classes of a topological vector bundle of rank $3$ on $\CP^5$.
Following \cite[Theorem 1]{Switzer}, let integers $c_1,\ldots, c_n$ be integers. Then there exist complex numbers $\delta_1,\ldots, \delta_n \in \mathbb C$ such that 

\begin{equation}\label{chern_roots}y^n+c_1y^{n-1}+\cdots + c_{n-1}y + c_n = \prod_{j=1}^n(y+\delta_j).\end{equation}
The Schwarzenberger condition is the requirement
\begin{equation}\label{schwarz} 
S_n\, \, : \forall r \in \Z \text{ such that }2 \leq r \leq n, \,\, \sum_{j=1}^n{\delta_j \choose r}  \in \mathbb Z \,.
\end{equation}
\begin{thm}[{\cite[Theorem A]{Thomas}, \cite[Theorem 1]{Switzer}}]\label{lem:Thomas_thmA} Integers $c_1,\ldots , c_k \in \Z$ are the Chern classes of a rank $k$ vector bundle on $\CP^k$ if and only if $c_1,\ldots, c_k$ satisfy the condition $S_k$.
\end{thm}
From this, we deduce necessary and sufficient conditions for three integers to be the Chern classes of a rank $3$ bundle on $\CP^5$.
\begin{lem}\label{lem:schwarz_enough}Let $a_1,a_2,a_3 \in \Z$. Then there exists a complex rank $3$ topological vector bundle $V$ on $\CP^5$ with $c_i(V)=a_i$ if and only if the $5$-tuple $(a_1,a_2,a_3,0,0)$ satisfies the Schwarzenberger condition $S_5.$

\end{lem}

\begin{proof}By Theorem~\ref{lem:Thomas_thmA} above, 
the condition is necessary.
To show the condition $S_5$ is sufficient, we prove that a rank $5$ vector bundle $V'$ on $\CP^5$ with $c_4$ and $c_5$ equal to zero is
isomorphic to a bundle $V \oplus \underline {\mathbb C}^2$, i.e. its stable class has a rank $3$ representative. Consider \cite[Proposition 5.7.5]{Zab}, which
implies that any (complex) rank $7$ vector bundle on $\CP^5$ with top four Chern classes zero is a sum of a rank $3$ bundle and four trivial bundles. We apply this to $V'\oplus \underline {\mathbb C}^2$ to conclude that $V'\oplus \underline{\mathbb C}^2\cong V \oplus \underline{\mathbb C}^4$. Since we are in the stable range, cancellation applies and we conclude $V' \cong V \oplus \underline{\mathbb C}^2$.
\end{proof}

We compute $S_5$ explicitly.

\begin{lem}\label{lem:explicit_S5}
The condition $S_5$ on $(a_1,a_2,a_3,0,0)$ is equivalent to the following system of congruences:
\begin{align*}a_1a_2+a_3 &\equiv0\pmod{2},\\ -a_2-a_1^2a_2+a_1a_3-a_2^2
&\equiv0\pmod{3},\\
a_3(a_1^2+a_2) &\equiv0\pmod{3},\\
a_2^2+a_2+a_1a_2-a_3
&\equiv0\pmod{4},\\
(a_3-a_1a_2)(3+a_2-a_1^2) &\equiv0\pmod{8}. \end{align*} \end{lem} 
\begin{proof}
We first expand the general formula in \eqref{schwarz}. 
Let $c_1,\ldots, c_n$ be given and let $\delta_j$ be as 
in \eqref{chern_roots}. Let $n \geq 1$ be a given integer and let 
$r \in \mathbb Z$ be such that $2 \leq r \leq n$.  

\begin{align*}
\sum_{j=1}^n {\delta_j \choose r} & =  \sum_{j=1}^n \frac{1}{r!} \prod_{k=0}^{r-1} (\delta_j-k).
\end{align*}
Let $[r-1]=\{0,\ldots, r-1\}$. Given a finite set $I$, 
let $\#I$ denote the cardinality of $I$. 
With this notation, we expand the right-hand side of the previous equation to obtain:
\begin{align*}
\sum_{j=1}^n {\delta_j \choose r}
&=  \sum_{j=1}^n \frac{1}{r!}\sum_{t=1}^r\sum_{\substack{ I\subset [r-1]\\ \# I=r-t} } \left( \prod_{k\in I}(-k)\right) \delta_j^t\\
\end{align*}
For $t\geq 0$, let $p_t=\sum_{j=1}^n \delta_j^t$. We may rewrite the previous equation as:
\begin{align}\label{eq:pt1}
\sum_{j=1}^n {\delta_j \choose r}
&=  \frac{1}{r!}\sum_{t=1}^r\sum_{\substack{ I\subset [r-1]\\ \# I=r-t} } \left( \prod_{k\in I}(-k)\right)p_t
\end{align}
 Note that, by definition of $\delta_j$, the elementary symmetric functions on $\delta_j$ are precisely the Chern classes $c_1,\ldots, c_n$, and thus the power sums $p_1,\ldots, p_r$ are polynomials in $c_1,\ldots, c_n$. 

Now we specialize to the case $n=5$, $c_i=a_i$ for $1\leq i \leq 3$, $c_4=c_5=0$. Using Newton's identities to express the power sums in terms of elementary symmetric functions, we find that:

\begin{equation}\label{sympoly}
\begin{matrix}
p_1&=& a_1\\
p_2&=&a_1^2-2a_2\\
p_3&=&a_1^3-3a_1a_2+3a_3\\
p_4&=& a_1^4-4a_1^2a_2+4a_1a_3+2a_2^2\\
p_5&=&a_1^5-5a_1^3a_2+5a_1a_2^2+5a_1^2a_3-5a_2a_3.
\end{matrix}
\end{equation}
\begin{itemize}
\item {\bf Case 1: $r=2.$} Substituting the results of Equations~\eqref{sympoly} in Equation~\eqref{eq:pt1} gives
\begin{align*}
\sum_{j=1}^5 {\delta_j \choose 2}
&=  \frac{1}{2!}\sum_{t=1}^2\sum_{\substack{ I\subset [1]\\ \# I=2-t} } \left( \prod_{k\in I}(-k)\right)p_t\\
&= \frac{1}{2}( -p_1+p_2)\\
&=\frac{1}{2}(-a_1+a_1^2-2a_2).
\end{align*}
Since $a_1 \equiv a_1^2\pmod 2$, this is always an integer.

\item {\bf Case 2: $r=3.$} Equation~\eqref{eq:pt1} becomes
\begin{align*}
\sum_{j=1}^5 {\delta_j \choose 3}
&=  \frac{1}{3!}\sum_{t=1}^3\sum_{\substack{ I\subset [2]\\ \# I=3-t} } \left( \prod_{k\in I}(-k)\right)p_t\\
&= \frac{1}{6}\left((-1)(-2)p_1+(-1-2)p_2+p_3 \right)\\
&= \frac{1}{6}(2p_1-3p_2+p_3).
\end{align*}
Substituting Equations~\eqref{sympoly} in the above, we obtain the condition
\begin{equation*}
2a_1-3(a_1^2-2a_2)+a_1^3-3a_1a_2+3a_3 \equiv 0 \pmod{6},
\end{equation*}
or equivalently:
\begin{equation*}
\begin{matrix}
a_1^2+a_1a_2+a_1^3+a_3 &\equiv 0 \pmod{2}\\
-a_1+a_1^3 & \equiv 0 \pmod{3}
\end{matrix}
\end{equation*}
The second equation gives no condition. The first equation simplifies to:
\begin{equation}\label{eq:case2}
\begin{matrix}
a_1a_2+a_3 &\equiv 0 \pmod{2}.
\end{matrix}
\end{equation}

\item {\bf Case 3: $r=4.$} Equation~\eqref{eq:pt1} becomes
\begin{align*}
\sum_{j=1}^5 {\delta_j \choose 4}
&=  \frac{1}{4!}\sum_{t=1}^4\sum_{\substack{ I\subset [3]\\ \# I=4-t} } \left( \prod_{k\in I}(-k)\right)p_t\\
&= \frac{1}{4!}\left( -6p_1+(2+3+6)p_2-(1+2+3)p_3+p_4\right) \\
&=\frac{1}{4!}\left(-6p_1+11p_2-6p_3+p_4\right).
\end{align*}
We substitute Equations~\eqref{sympoly} to obtain:
\begin{align*}
\sum_{j=1}^5 {\delta_j \choose 4}
&=\frac{1}{4!}(-6a_1+11(a_1^2-2a_2)\\ & \,\,\,\,\,\,\,-6(a_1^3-3a_1a_2+3a_3)+a_1^4-4a_1^2a_2+4a_1a_3+2a_2^2)\\
&= \frac{1}{4!}(-6a_1+11a_1^2-22a_2-6a_1^3+18a_1a_2\\ &\,\,\,\,\,\,\,-18a_3 +a_1^4-4a_1^2a_2+4a_1a_3+2a_2^2)\\
\end{align*}
This quantity is an integer if and only if
\begin{equation}\label{quantity34}-6a_1+11a_1^2-22a_2-6a_1^3+18a_1a_2-18a_3 +a_1^4-4a_1^2a_2+4a_1a_3+2a_2^2\end{equation} is zero modulo $8$ and $3$. 

First we reduce \eqref{quantity34} modulo $8$. Note that we already require $a_1a_2\equiv a_3\pmod{2}$, so if the previous Schwarzenberger conditions are satisfied:
\[4(-a_1^2a_2+a_1a_3)\equiv 0 \pmod{8}.\]
The condition reduces to:
\begin{equation}\label{eq:222}Q:=-6a_1+11a_1^2-22a_2-6a_1^3+18a_1a_2-18a_3 +a_1^4+2a_2^2\equiv 0 \pmod{8} .\end{equation}

Using the condition $a_1a_2\equiv a_3\pmod{2}$, this further reduces to:
\[
2(a_2^2+a_2+a_1a_2-a_3)\pmod{8}.
\]
To see this, note that
\begin{align*}
Q
-2(a_2^2+a_2+a_1a_2-a_3) &=-6a_1+11a_1^2-6a_1^3+a_1^4-24a_2+16a_1a_2+16a_3\\
&=a_1(a_1-1)(a_1-2)(a_1-3)-24a_2+16a_1a_2-16a_3.
\end{align*}
Every term on the right is divisible by $8$. Therefore,
the $\pmod{8}$ condition in this case in combination with the previous conditions is equivalent to: 
\begin{equation}\label{case3mod8}
a_2^2+a_2+a_1a_2-a_3\equiv0\pmod{4}.
\end{equation}

Reducing \eqref{quantity34} modulo 3, we obtain:
\begin{align*}
-a_1^2-a_2+a_1^4-a_1^2a_2+a_1a_3-a_2^2\equiv 0 \pmod{3},
\end{align*}
which simplifies to:
\begin{align}\label{eq:case3}
-a_2-a_1^2a_2+a_1a_3-a_2^2\equiv 0 \pmod{3}.
\end{align}

\item {\bf Case 4: $r=5.$} Equation~\eqref{eq:pt1} becomes
\begin{align*}
\sum_{j=1}^5 {\delta_j \choose 5}
&=  \frac{1}{5!}\sum_{t=1}^5\sum_{\substack{ I\subset [4]\\ \# I=5-t} } \left( \prod_{k\in I}(-k)\right)p_t\\
&= \frac{1}{5!}\left( 24 p_1-50p_2+35p_3-10p_4+p_5 \right).
\end{align*}
Thus, $ \sum_{j=1}^5 {\delta_j \choose 5}$ is an integer if and only if
\begin{equation*}24 p_1-50p_2+35p_3-10p_4+p_5 \equiv 0 \pmod{5!},\end{equation*}
which is equivalent to
\begin{equation}\label{eq:case4_1}24 p_1-50p_2+35p_3-10p_4+p_5 \equiv 0 \pmod{m} \,\,\, \forall m \in \{3,8,5\}.\end{equation}
Reducing Equation~\eqref{eq:case4_1} modulo $5$ and substituting Equations~\eqref{sympoly},
we obtain
\[ -a_1+a_1^5-5a_1^3a_2+5a_1a_2^2+5a_1^2a_3-5a_2a_3 \equiv 0 \pmod 5,\] which is always satisfied.
 
Reducing Equation~\eqref{eq:case4_1} modulo $3$
we obtain
\[  p_2-p_3-p_4+p_5 \equiv 0 \pmod 3.\]
Substituting Equations~\eqref{sympoly}, we obtain:
\begin{align*}p_2-p_3-p_4+p_5 &\equiv a_1^2-2a_2-(a_1^3-3a_1a_2+3a_3)-(a_1^4-4a_1^2a_2+4a_1a_3+2a_2^2)\\ & \,\,\,\,\,\,\,\,+a_1^5-5a_1^3a_2+5a_1a_2^2+5a_1^2a_3-5a_2a_3 \\& \equiv 0\pmod{3}, \end{align*}
which we rewrite as
\begin{align*} a_1^2+a_2-a_1^3-a_1^4+a_1^2a_2-a_1a_3+a_2^2\\ 
+a_1^5+a_1^3a_2-a_1a_2^2-a_1^2a_3+a_2a_3&\equiv 0 \pmod{3}.\end{align*}
Cancelling appropriate powers of $a_1$, this simplifies to:
\begin{equation}\label{bleh33}a_2+a_1^2a_2-a_1a_3+a_2^2
+a_1a_2-a_1a_2^2-a_1^2a_3+a_2a_3 \equiv 0 \pmod{3}.\end{equation}
We now combine the above with the $r=4$ $\pmod{3}$ requirement from Equation~\eqref{eq:case3}. Note that multiplying
Equation~\eqref{eq:case3} by $-1$ gives
\[
a_2+a_1^2a_2-a_1a_3+a_2^2\equiv0\pmod{3}.
\]
Consequently, Equation~\eqref{bleh33} reduces to
\[
a_1a_2-a_1a_2^2-a_1^2a_3+a_2a_3\equiv0\pmod{3}.
\]
On the other hand, multiplying Equation~\eqref{eq:case3} by $a_1$ and
using $a_1^3\equiv a_1\pmod{3}$ gives
\[
a_1a_2-a_1a_2^2+a_1^2a_3\equiv0\pmod{3}.
\]
Subtracting this congruence from the preceding one, we obtain
\begin{equation}\label{case4mod3}
a_3(a_1^2+a_2)\equiv0\pmod{3}.
\end{equation}

 Now we consider Equation~\eqref{eq:case4_1} modulo $8$. Substitute Equations~\eqref{sympoly} to
obtain:
\[
6p_2+3p_3+6p_4+p_5  \equiv 0 \pmod{8}.\]
We can multiply through by $3\equiv 3^{-1}\pmod{8}$ to obtain: 
\[
2p_2+p_3+2p_4+3p_5  \equiv 0 \pmod{8}.\]
Substituting Equations~\eqref{sympoly}, this condition becomes
\begin{align*}
0
&\equiv 2(a_1^2-2a_2)
 +(a_1^3-3a_1a_2+3a_3)\\
&\qquad
 +2(a_1^4-4a_1^2a_2+4a_1a_3+2a_2^2)\\
&\qquad
 +3(a_1^5-5a_1^3a_2+5a_1a_2^2
       +5a_1^2a_3-5a_2a_3).\end{align*}
Distributing and regrouping, we get the equation modulo $8$:
\begin{equation}\label{eq:111-0000}
2a_1^2+4a_2+a_1^3-3a_1a_2+3a_3
+2a_1^4+4a_2^2
+3a_1^5+a_1^3a_2-a_1a_2^2-a_1^2a_3+a_2a_3
\equiv 0.\end{equation}
Note that $a_1^2\equiv a_1^4\pmod{8}$ unless
$a_1\equiv2,6\pmod{8}$. Therefore, for every value of $a_1\pmod{8}$,
\[
2a_1^2+2a_1^4\equiv4a_1^4\pmod{8},
\]
since when $a_1\equiv2,6\pmod{8}$, both summands on the left-hand
side are zero modulo $8$, as is the right-hand side.
Similarly, $a_1^3\equiv a_1^5\pmod{8}$ for every integer $a_1$.
Therefore,
\[
a_1^3+3a_1^5\equiv4a_1^5\pmod{8}.
\]
Combining these two congruences, we obtain
\[
2a_1^2+a_1^3+2a_1^4+3a_1^5
\equiv4a_1^4+4a_1^5
=4a_1^4(1+a_1)
\equiv0\pmod{8}.
\]
We substitute in Equation~\eqref{eq:111-0000} to obtain:
\begin{equation}\label{eq:111-0101}
4a_2-3a_1a_2+3a_3
+4a_2^2
+a_1^3a_2-a_1a_2^2-a_1^2a_3+a_2a_3\equiv 0\pmod{8} \end{equation} 
We simplify Equation~\eqref{eq:111-0101} one step further. Note that:
\[
4a_2+4a_2^2=4a_2(1+a_2)\equiv0\pmod{8}.
\]
Therefore Equation~\eqref{eq:111-0101} reduces to
\begin{align*}
0
&\equiv
-3a_1a_2+3a_3+a_1^3a_2-a_1a_2^2-a_1^2a_3+a_2a_3\\
&=
3(a_3-a_1a_2)+(a_1^2-a_2)(a_1a_2-a_3)\\
&=
(a_3-a_1a_2)(3+a_2-a_1^2)
\pmod{8}.
\end{align*}
Thus, the $r=5$ condition modulo $8$ is
\begin{equation}\label{case4mod8}
(a_3-a_1a_2)(3+a_2-a_1^2)\equiv0\pmod{8}.
\end{equation}
\end{itemize}
Combining Equations~\eqref{eq:case2}, \eqref{case3mod8}, \eqref{eq:case3},
\eqref{case4mod3}, and \eqref{case4mod8}
gives precisely the system of congruences stated in the lemma.
\end{proof}

To illustrate the form of the equations above, we record the following
weaker necessary condition, omitting the modulo $4$ and modulo $8$
requirements.
\begin{cor}
Let $(a_1,a_2,a_3)\in\mathbb Z^3$. Let
$(a_1,a_2,a_3)\pmod n$ denote entry-wise reduction modulo $n$.
If there exists a rank $3$ bundle on $\CP^5$ with $i$-th Chern class
equal to $a_i$, then the residue of $(a_1,a_2,a_3)$ satisfies the
following two conditions:
\begin{enumerate}
\item
\[
(a_1,a_2,a_3)\pmod{2}\in
\begin{Bmatrix}
(0,\star,0), & (\star,0,0), & (1,1,1)
\end{Bmatrix}.
\]

\item
\[
(a_1,a_2,a_3)\pmod{3}\in
\begin{Bmatrix}
(0,0,\star), & (0,-1,0), & (1,0,0), & (-1,0,0)\\
(1,1,0), & (-1,1,0), & (1,-1,-1), & (-1,-1,1)
\end{Bmatrix}.
\]
\end{enumerate}
In the above, $\star$ means that there is no constraint on the
indicated entry.
\end{cor}

\section{Defining a twisted $\tmf$-valued invariant}\label{sec:exist}

By Theorem~\ref{classification}, rank $3$ bundles on $\CP^5$ that are not determined by their Chern data arise from the action of $\pi_{10}BU{\hspace{-.6pt}}(3) \simeq \Z/3$ given in Construction~\ref{const:action_Z3}. To go from an action to an invariant, we must study the unstable homotopy of $BU{\hspace{-.6pt}}(3)$ in greater detail. We outline the key insights in Lemmas~\ref{lem:htpy_BU2_BU3_spheres}, \ref{lem:htpy_BU3_sphere2}, and \ref{lem:negative_gamma_works} below. These are motivational but not logically necessary for what follows, so we omit most proofs.
\begin{conv} Throughout this section all spaces and spectra are implicitly localized at $3$.
\end{conv}

For each $n\geq 1$, there is a fiber sequence $U(n+1)/U(n) \to BU{\hspace{-.6pt}}(n) \to BU{\hspace{-.6pt}}(n+1).$ Note that $U(n+1)/U(n)\simeq S^{2n+1}$, so in fact we have a fiber sequence

\begin{equation}\label{delta}S^{2n+1} \xrightarrow{\delta_{n+1}} BU{\hspace{-.6pt}}(n) \to BU{\hspace{-.6pt}}(n+1).\end{equation} 
Considering the long exact sequence on homotopy associated to \eqref{delta}, and using the fact that $\pi_{2n+1}BU(n+1)=0,$ we see that $\delta_{n+1}$ generates $\pi_{2n+1}BU(n) \simeq \Z/n!$ \cite{Mimura}. The homotopy class of $\delta_{n+1}$ is linked to the existence of non-isomorphic vector bundles with the same Chern data in both the case of rank $2$ bundles on $\CP^3$ and that of rank $3$ bundles on $\CP^5$, as shown by the next result.

\begin{lem}\label{lem:htpy_BU2_BU3_spheres} A generator for $\pi_6BU{\hspace{-.6pt}}(2)$ is given by the composite $$S^6 \xrightarrow{\eta^u} S^5 \xrightarrow{\delta_3} BU{\hspace{-.6pt}}(2),$$
where $\delta_3$ is as in \eqref{delta} and $\eta^u$ is an unstable representative for the lowest-degree order $2$ element in $\pi_*\sphere$.

A generator for $\pi_{10}BU{\hspace{-.6pt}}(3)$ is given by the composite $$S^{10} \xrightarrow{\alpha_1^u} S^{7} \xrightarrow{\delta_4} BU{\hspace{-.6pt}}(3),$$
where $\delta_4$ is as in \eqref{delta} and $\alpha_1^u$ is an unstable representative for the lowest-degree order $3$ element in $\pi_*\sphere$.
\end{lem}
Moreover, we can describe the generator for $\pi_{10}(BU{\hspace{-.6pt}}(3))$ as a further composite.
\begin{conv}\label{convalpha} Following standard naming conventions, we will use \[\eta\:\sphere^1 \to \sphere^0\] and \[\alpha_1\:\sphere^3 \to \sphere^0\] to refer to the the first $2$- and $3$-torsion elements in $\pi_*\sphere$, respectively. \end{conv}

\begin{lem}\label{lem:htpy_BU3_sphere2} Let $x\: S^4 \to BU{\hspace{-.6pt}}(3)$ generate $\pi_4BU{\hspace{-.6pt}}(3)$. Then there is a map $\epsilon\: S^7 \to S^4$ such that:

\begin{enumerate}
\item $\epsilon \circ x$ generates the three-torsion in $\pi_7BU{\hspace{-.6pt}}(3)$; and
\item $\susp \epsilon = \Sigma^4\alpha_1$.
\end{enumerate}
\end{lem}
\begin{rmk} Lemma~\ref{lem:htpy_BU2_BU3_spheres} for $BU{\hspace{-.6pt}}(3)$ and Lemma~\ref{lem:htpy_BU3_sphere2} will follow from the the proof of Theorem~\ref{thm:tmf_class_exists}. \end{rmk}
Combining the previous two lemmas, $\pi_{10}BU{\hspace{-.6pt}}(3)$ is generated by
\begin{equation}\label{eq:alpha12}\sigma\:S^{10}\xrightarrow{\alpha_1^u} S^7 \xrightarrow{\epsilon} S^4 \xrightarrow{x} BU{\hspace{-.6pt}}(3).\end{equation}

This shows that a generator $\sigma$ for $\pi_{10}BU{\hspace{-.6pt}}(3)$ is stably trivial in multiple ways: first, the bundle on $S^{10}$ represented by $\sigma$ is stably trivial, meaning is null after composition with the natural map $BU(3) \to BU$. This follows from the fact that $\sigma$ is torsion in $\pi_*BU(3)$ while $\pi_*BU$ is torsion-free. Second, $\susp \sigma =\Sigma^4 \alpha_1^2x=0$ since $\alpha_1^2=0$ in $\pi_*\sphere$. In fact the composition in \eqref{eq:alpha12} is null after just one suspension.

Instead of applying the suspension spectrum functor to Diagram~\eqref{eq:alpha12}, we can apply a Thom spectrum functor, with respect to a multiple of the universal bundle on $BU(3)$. Let $V$ be a bundle on $BU{\hspace{-.6pt}}(3)$ (for example, the universal bundle $\gamma_3$, its determinant, or $-\gamma_3$). Equation~\eqref{eq:alpha12} gives a sequence of spectra

\begin{equation}\label{eq:thomified_htpy0}
\begin{tikzcd}[row sep=1.2em]
\Th{S^{10}}{V|_{S^{10}}}\ar[r,"\operatorname{Th}(\alpha_1^u)"] & \Th{S^7}{V|_{S^7}}\ar[r,"\operatorname{Th}(\epsilon)"] & \Th{S^4}{V|_{S^4}} \ar[r,"\op{Th}(x)"] & \Th{BU{\hspace{-.6pt}}(3)}{V}.
\end{tikzcd}
\end{equation}

Since $\alpha_1^u \circ \epsilon \circ x$ and $\epsilon \circ x$ represent torsion classes in $\pi_*BU(3)$, thus are trivial in $\pi_*BU$, the bundles $V|_{S^{10}}$ and $V|_{S^7}$ are stably trivial in Diagram~\eqref{eq:thomified_htpy0} above. If we fix a spherical orientation fo $V|_{S^7}$, we obtain a diagram:
\begin{equation}\label{eq:thomified_htpy}
\begin{tikzcd}[row sep=1.5em, column sep=3em]
\sphere^{10}\ar[d]\ar[drrr,bend left=20, dashed, "\tilde v" above] \ar[r,"\Sigma^7\alpha_1" below] &\sphere^7\ar[d] \\
\suspp S^{10}\ar[r,"\operatorname{Th}(\alpha_1^u)"] & \suspp S^7\ar[r,"\operatorname{Th}(\epsilon)"] & \Th{S^4}{V|_{S^4}} \ar[r,"\op{Th}(x)"] & \Th{BU{\hspace{-.6pt}}(3)}{V}.
\end{tikzcd}
\end{equation}
For various choices of $V$, we can ask whether $\tilde v$ is null.

The Thom spectrum $\Th{S^4}{V|_{S^4}}$ has two cells, one in degree four and one in degree zero. Thomifying can either keep the cells split or introduce an $\alpha_1$-attaching map, in which case the Thom spectrum is $C(\alpha_1)$, the cofiber of the map $\alpha_1\: \sphere^3 \to \sphere^0$ defined in Convention~\ref{convalpha}. 
Which option occurs depends on $V$. For the dotted composite $\tilde v$ to be nonzero, the latter must occur.
Moreover, given any spectrum $X$ together with a map $C(\alpha_1)\to X$, if

\begin{equation}\label{eq:composite}\big(\sphere^{10} \xrightarrow{\Sigma^7\alpha_1} \sphere^{7} \to C(\alpha_1)\to X \big) \not\simeq 0.\end{equation} Then the image in $X$ of the $4$-cell of $C(\alpha_1)$ cannot support a $P^1$ in $X$. 
Indeed, let $x$ denote the cohomology class on $X$ corresponding to the image of the $4$-cell of $C(\alpha_1)$. If $P^1(x)=y$ in the cohomology of $X$, then $y$ is an $8$-cell attached via an $\alpha_1$ attaching map, and $\sphere^7 \to C(\alpha_1) \to X$ is null, since this composite corresponds to $\alpha_1$ on $x$.

From experimentation, it seems that taking $X=\Th{BU{\hspace{-.6pt}}(3)}{V}$ with $V$ any multiple, exterior power, or tensor power of $\gamma_3$ fails one condition or the other. Explicitly, either $\Th{S^4}{V|_{S^4}}\simeq \suspp S^4$ or there is a nonzero $P^1$ on the relevant $4$-cell of $\Th{BU{\hspace{-.6pt}}(3)}{V}$.

To resolve this issue, we modify our classifying space.
\begin{defn}\label{def:BUn_coz} Let $BU{\hspace{-.6pt}}(n)_{\coz}:=\op{hofib}\big(c_1\pmod 3\: BU{\hspace{-.6pt}}(n)\to K(\Z/3,2)\big).$\end{defn}
\begin{rmk} We note a few facts about $BU{\hspace{-.6pt}}(n)_{\coz}$. 
\begin{itemize}
\item The space $BU{\hspace{-.6pt}}(n)_{\coz}$ has even cohomology, so
any rank $n$ bundle on $\CP^k$ with $c_1\equiv 0 \pmod 3$ lifts uniquely, up to homotopy, along the natural map $BU{\hspace{-.6pt}}(n)_{\coz}\to BU{\hspace{-.6pt}}(n)$.
\item The natural map $BU{\hspace{-.6pt}}(n)_{\coz} \to BU{\hspace{-.6pt}}(n)$ is a induces an isomorphism on $i$-th homotopy groups for $i>2$.
\item The space $BU{\hspace{-.6pt}}(n)_{\coz}$ carries a universal bundle which we denote $\gamma_n$.
\end{itemize}
\end{rmk}
Thus our previous analyses of rank $3$ bundles on $\CP^5$ can be repeated after adding the constraint $c_1\equiv 0 \pmod 3$ and substituting $\BUc$ in place of $BU{\hspace{-.6pt}}(3)$. In particular, we get a modification of Diagram~\eqref{eq:thomified_htpy}:
\begin{equation}\label{eq:thomified_htpy2}
\begin{tikzcd}[row sep=1.5em, column sep=3em]
\sphere^{10}\ar[d]\ar[drrr,bend left=20, dashed, "\tilde v" above] \ar[r,"\Sigma^7\alpha_1" below] &\sphere^7\ar[d] \\
\suspp S^{10}\ar[r,"\operatorname{Th}(\alpha_1^u)"] & \suspp S^7\ar[r,"\operatorname{Th}(\epsilon)"] & \Th{S^4}{\smallminus\gamma_3} \ar[r,"\op{Th}(x)"] & \BUct.
\end{tikzcd}
\end{equation}
\begin{lem}\label{lem:negative_gamma_works} In the diagram above, $\Th{S^{4}}{\smallminus\gamma_3}\simeq C(\alpha_1)$ and the element $\tilde v$ is nontrivial in $\pi_{10}\big(\BUct\big).$
\end{lem}

It will be useful to have better terminology for the Thomified homotopy classes such as $\tilde v$.

\begin{defn}\label{def:thom_sigma} Given a pointed map $y\: S^n \to BU{\hspace{-.6pt}}(r)_{\coz}$ representing a stably trivial bundle on $S^n$, and a nullhomotopy $u$ of the composite map to $BGL_1\sphere$, let $\op{Th}_u(y)\: \sphere^n \to \Th{BU{\hspace{-.6pt}}(r)}{\smallminus\gamma_r}$ denote the composite
$$ \sphere^{n} \xrightarrow{i_2}\suspp S^n \xrightarrow{u_{\simeq}} \Th{S^n}{-y} \to \Th{BU{\hspace{-.6pt}}(r)}{\smallminus\gamma_r},$$
where the arrow $i_2$ is the inclusion of the top cell determined by a base point and the map $u_{\simeq}$ is the spherical Thom isomorphism determined by $u$. (We recall the basics of orientations and Thom isomorphisms in Subsection~\ref{subsec:background_orient}.)
\end{defn}

The main goal of this section is to prove the following:
\begin{thm}\label{thm:tmf_class_exists} Given $\sigma\: S^7 \to BU{\hspace{-.6pt}}(3)$ generating $\pi_7BU{\hspace{-.6pt}}(3)$ and a Thom class $u_0$ for $\sigma$, there is a unique class $$\tilde \rho \in \tmf^{-3}( \sk^{26}\BUct)$$ such that $$\operatorname{Th}_{u_0}(\sigma)^*\tilde \rho=\alpha_1\beta_1 \in \pi_{13}\tmf,$$ where $\sk^{26}\BUct$ is the Thomification of a $26$-skeleton of $\BUc.$
\end{thm}
\begin{rmk} Throughout this section, we will use certain structural properties of the spectrum of topological modular forms, localized at the prime $3$. The spectrum $\op{tmf}$ of topological modular forms is a celebrated object in homotopy theory; its existence and geometric properties are likely to be of general interest to the reader. We refer to \cite{DFHH} for an exposition and further sources on $\op{tmf}$. For our purposes, we will only use that $\tmf$ is an ($E$-infinity) ring spectrum and that its $3$-local homotopy ring has a certain form. A summary of the homotopy of $\tmf$ is contained in \cite[Part 1, Chapter 13, Section 1]{DFHH}.
We will frequently use the following fact: the classes $\alpha_1 \in \pi_3\sphere,$ 
$\beta_1\in \pi_{10}\sphere$, and $\alpha_1\beta_1 \in \pi_{13}\sphere$ all have nonzero images under the Hurewicz map $\pi_{*}\sphere \to \pi_{*}\tmf$ induced by the unit map $\sphere \to \tmf$.
\end{rmk}
\begin{rmk} The reader may wonder why we look for a class in $\tmf$ cohomology. The spectrum $X=\Sigma^{-3}\tmf$ carries a natural map $C(\alpha_1)\to \Sigma^{-3}\tmf$ such that Equation~\eqref{eq:composite} holds. Moreover, $\tmf$ is one of the simplest ring spectra with this property. This makes $\tmf$ a natural candidate to detect $\pi_{10}\BUc$.
\end{rmk}

In Subsection~\ref{subsec:start_existence_proof}, we outline the proof strategy for Theorem~\ref{thm:tmf_class_exists}. We prove $\tilde \rho$ exists in Subsection~\ref{subsec:proof_exists}, predicated on cohomology calculations which are recorded in Subsection~\ref{subsec:cohomology_BUct}. The proof of uniqueness of $\tilde \rho$, given in Subsection~\ref{subsec:proof_unique}, also uses these calculations.

\subsecl{Proof outline: existence and uniqueness of a twisted $\tmf$ invariant}{subsec:start_existence_proof}

In this subsection, we will state a sequence of propositions and explain how they imply Theorem~\ref{thm:tmf_class_exists}. To begin, let $u$ be the canonical Thom class in the $\HZt$-cohomology of $\BUct$. As a module over the $\mod 3$ Steenrod algebra,

$$\HF{3}{*}\BUct \simeq \big(\HF{3}{*}\BUc\big)\cdot u.$$ We find that $P^1(u)=-c_2\cdot u$. Since $P^1$ detects $\alpha_1$ attaching maps, this implies that $\alpha_1$ on the Thom class in $\BUct$ is zero. Therefore the obstruction to extending the zero cell $\sphere^0 \to \BUct$ over the cofiber of $\alpha_1\: \sphere^3 \to \sphere^0$ vanishes, and we may choose and extension 
$$k\: C(\alpha_1)\to \BUct$$
that takes the $0$-cell in $C(\alpha_1)$ to the $0$-cell in $\BUct$ and the $4$-cell in $C(\alpha_1)$ to the cell in $\BUct$ representing the dual to the cohomology class $-c_2\cdot u$.

\begin{prop}\label{prop:splitting} The map $k\: C(\alpha_1) \to \sk^{26}\BUct$ splits after tensoring with $\tmf$. More precisely, there is a map of $\tmf$-modules $$r\: (\sk^{26}\BUct)\otimes \tmf \to C(\alpha_1) \otimes \tmf$$ so that $r\circ (k\otimes \tmf)$ is homotopic to the identity.
\end{prop}
Fix $u_0$ a spherical orientation for $\sigma\: S^{10}\to \BUc$ generating $\pi_{10}\BUc.$ Assuming the previous proposition, we immediately obtain:
\begin{cor}\label{cor:exists} There is a map $\tilde \rho\: \sk^{26}\BUct\to\Sigma^{-3}\tmf$ such that $$\left(\sphere^{10} \xrightarrow{\Thom_{u_0}(\sigma)} \sk^{26} \BUct \xrightarrow{\tilde\rho} \Sigma^{-3}\tmf\right)=\alpha_1\beta_1$$ in $\pi_{13}\tmf$.
\end{cor}
\begin{proof}[Proof of Corollary~\ref{cor:exists} assuming Proposition~\ref{prop:splitting}] Let $r$ and $k$ be as above. The map $\alpha_1\: \sphere^0\to\Sigma^{-3}\tmf$ extends over $C(\alpha_1)$, since $\alpha_1^2=0$.
Let $o: C(\alpha_1) \to \Sigma^{-3}\tmf$ denote an extension and let $\bar o=o\otimes \tmf$ be the associated map of $\tmf$-modules. Let $1\: \sphere \to \tmf$ be the unit map. We then define $\tilde \rho$ to be the composite
\[
\begin{tikzcd}[column sep=12]
\sk^{26}\BUct\otimes \sphere \ar[r,"1 \otimes \iota"]\ar[drr, "\tilde\rho" below]& \sk^{26}(\BUct)\otimes\tmf \ar[r,"r"] & C(\alpha_1)\otimes\tmf \ar[d,"\bar o" left]\\
& &\Sigma^{-3}\tmf.
\end{tikzcd}
\]
To show this map has the desired property, consider the homotopy commutative diagram:
\[
\begin{tikzcd}
\sk^{26}(\BUct)\otimes\tmf \ar[r,"r"] & C(\alpha_1)\otimes\tmf \ar[r,"\bar o"]\ar[l,bend right,"k\otimes \tmf" above] &\Sigma^{-3}\tmf \\
\sk^{26}(\BUct)\otimes\sphere \ar[u,"1\otimes \iota"] & C(\alpha_1)\otimes\sphere \ar[u,"1\otimes \iota"]\ar[ur,"o" below]\ar[l,bend right,"k" below] \\
\sphere^{10}\ar[u,"\Thom(\sigma)"]\ar[ur,"\beta_1\lbrack 0\rbrack" below]
\end{tikzcd}
\]
where $\beta_1\lbrack 0\rbrack$ is the image of $\beta_1\in\pi_{10}\sphere^0$ under $\sphere^0\rightarrow C(\alpha_1).$
That the lower triangle commutes is a consequence of the fact that composite 
\[\sphere^{10}\xrightarrow{\alpha_1}\sphere^7\xrightarrow{\alpha_1\lbrack 4\rbrack} C(\alpha_1)\] agrees with the Toda brack \[\langle \alpha_1,\alpha_1,\alpha_1\rangle =\beta_1.\] The map $o \circ \beta_1\lbrack 0\rbrack \in \pi_{13}\tmf$ is precisely $\alpha_1\beta_1$.
\end{proof}
The splitting is not canonical. However, given an identification $\Th{S^{10}}{-\sigma} \simeq \sphere^0\oplus \sphere^{10}$,
the class $\tilde \rho$ is uniquely determined by requiring it to restrict to $\alpha_1\beta_1$ on $\sphere^{10}$.
\begin{prop}\label{prop:uniqueness} Let $\epsilon\: S^7 \to BU{\hspace{-.6pt}}(3)$ generate $\pi_{10}BU{\hspace{-.6pt}}(3)$ and let $u_0$ be a spherical orientation for $\epsilon$. Let $\Thom_{u_0}$ be as in Definition~\ref{def:thom_sigma}. Suppose that $$\tilde \rho, \,\tilde \rho'\in\tmf^{-3}(\sk^{26}\BUct)$$ are such that 
\[\Thom_{u_0}(\epsilon)^*(\tilde \rho)=\Thom_{u_0}(\epsilon)^*(\tilde \rho')=\alpha_1\beta_1\in\pi_{13}\tmf.\] Then $\tilde \rho =\tilde \rho'$.
\end{prop}
We will prove this in Subsection~\ref{subsec:proof_unique}. This immediately implies:
\begin{cor}\label{cor:uniqueness} Let $\sigma = \alpha_1\circ \epsilon \in \pi_{10}BU{\hspace{-.6pt}}(3).$
Then:
\begin{itemize}
\item $\sigma= \alpha_1\circ \epsilon$ generates $\pi_{10}BU{\hspace{-.6pt}}(3)$.
\item If $\tilde \rho, \, \tilde \rho' \in \tmf^{-3}(\sk^{26}\BUct)$ satisfy $$\Thom_{u_0}(\sigma)^*(\tilde \rho) =\Thom_{u_0}(\sigma)^*(\tilde \rho')=\alpha_1\beta_1\in \pi_{13}\tmf,$$

then $\tilde \rho = \tilde \rho'.$
\end{itemize}
\end{cor}
\begin{proof} Note that the orientation $u_0$ for $\epsilon$ gives an orientation $u_0$ for $\sigma$ such that $$\alpha_1\cdot \Thom_{u_0}(\epsilon)=\Thom_{u_0}(\sigma).$$

For the second item, suppose that $\rho',\rho$ both satisfy
\begin{align*}\Thom(\sigma)^*(\tilde \rho) &= \alpha_1\beta_1=\Thom(\sigma)^*(\tilde \rho')\end{align*}
in $\pi_*\tmf$. This implies:
\begin{align*} \Thom(\epsilon)^*(\tilde \rho) &=\beta_1 = \Thom(\epsilon)^*(\tilde \rho')
, \end{align*}
so, by Proposition~\ref{prop:uniqueness},
$\tilde \rho =\tilde \rho' \in \tmf^{-3}\left(\sk^{26}\BUct\right)$.

For the first item, note that $\Thom(\sigma)^*(\tilde \rho) \neq 0$, which implies $\sigma \neq 0$. Since $\pi_{10}BU{\hspace{-.6pt}}(3)$ is cyclic, this implies $\sigma$ generates.
\end{proof}
Together, Corollary~\ref{cor:exists} and Corollary~\ref{cor:uniqueness} imply Theorem~\ref{thm:tmf_class_exists}. It remains to prove Propositions~\ref{prop:splitting} and \ref{prop:uniqueness}. These are the projects of the next subsections.

\subsecl{Proof of Proposition~\ref{prop:splitting}}{subsec:proof_exists}

By a skeleton of $\BUct$ we mean a term in a filtration of $\BUct$ obtained by Thomifying a skeletal filtration of $\BUc$. Since $\BUc$ is finite type and has even cohomology, $\BUc$ has a cell structure filtering the space as in Diagram~\eqref{eq:filt_BUc} below.
\begin{equation}\label{eq:filt_BUc}
\begin{tikzcd}[row sep = 7]
& \vdots\\
\vee_{j \in H_{i+2}} S^{i+3}\ar[r,"\vee \bar c_{i+2,j}"] & \sk^{i+2}\BUc \ar[u] \\
\vee_{j \in H_{i}}S^{i+1} \ar[r,"\vee \bar c_{ij}"] & \sk^{i}\BUc \ar[u]\\
&\vdots \ar[u] &,
\end{tikzcd}
\end{equation}
where each $H_{i}$ above is a finite indexing set and $\sk^{i+2}\BUc$ is obtained by taking the cofiber of a map $\vee \bar c_{ij}\:\vee_{j \in H_{i}}S^{i+1} \to \sk^{i}\BUc.$

Each $\sk^i\BUc$ carries a bundle pulled back from $\gamma_3$ on $\BUc$.
We can Thomify all diagrams involved to obtain a filtration for $\BUct$: each stage in Diagram~\eqref{eq:filt_BUc} gives pushout in spectra
\begin{equation}\label{eq:skeleta_Thom_BSU3}
\begin{tikzcd}
* \ar[r] & \sk^{i+2}\BUct \\
\oplus_{H_{i}}\sphere^{i+1}\ar[u]\ar[r,"\oplus_{H_i}c_{ij}"]& \sk^{i}\BUct\,. \ar[u]
\end{tikzcd}\end{equation}
In Diagram~\ref{eq:skeleta_Thom_BSU3}, we define $\sk^i\BUct := \Th{\sk^i\BUc}{\smallminus\gamma_3}$ and $c_{ij}$ is the Thomification of $ \bar c_{ij}$ restricted to $\sphere^{i+1}$.

Consider the cofiber $$C:=\operatorname{cof}\left( C(\alpha_1)\xrightarrow{k} \sk^{26}\BUct\right).$$
\begin{lem}\label{lem:zero_maps}With notation as above, $\pi_0\operatorname{Maps}_{\tmf}(\Sigma^{-1}C\otimes \tmf, C(\alpha_1)\otimes \tmf) =0.$
\end{lem}
Given this lemma, the going around map $\Sigma^{-1}C\otimes \tmf \to C(\alpha_1)\otimes \tmf$ is null and there is an extension making Diagram~\ref{eq:section_exists_2} homotopy commutative, which is the desired section.
\begin{equation}\label{eq:section_exists_2}
\begin{tikzcd}
\Sigma^{-1}C\otimes \tmf \ar[r,"0"] & C(\alpha_1)\otimes \tmf \ar[r,"k"] \ar[d,"="]&\sk^{26}\BSUt\otimes \tmf \ar[dl,dashed,"\exists r"] \\
& C(\alpha_1)\otimes \tmf
\end{tikzcd}
\end{equation}

\begin{proof}[Proof of Lemma~\ref{lem:zero_maps}]

Using the free-forgetful adjunction for $\tmf$-modules:
$$\pi_0\operatorname{Maps}_{\tmf}(\Sigma^{-1}C \otimes \tmf, C(\alpha_1)\otimes \tmf) \simeq \pi_0\operatorname{Maps}_\sphere(\Sigma^{-1}C, C(\alpha_1)\otimes \tmf).$$
To establish the result, we compute $\pi_0\operatorname{Maps}_\sphere(\Sigma^{-1}C, C(\alpha_1)\otimes \tmf)$ via an Atiyah--Hirzebruch spectral sequence
\begin{equation}\label{eq:AHSS_maps}E_{2}^{q,p}=H^p\left(\Sigma^{-1}C; \left(\tmf\right)_{-q}C(\alpha_1)\right)\implies \pi_{p+q}\operatorname{Maps}_{\sphere}\left(\Sigma^{-1}C, C(\alpha_1)\otimes \tmf\right).\end{equation}
We use grading conventions indicated in Figure~\ref{fig:AHSS_schematic} and we depict the spectral sequence in Figure~\ref{fig:tmf_hom_maps}. Explicitly, the term in the $q$-th position on the horizontal axis and the $p$-th position on the vertical axis is the term $E_2^{q,p}$. The $r$-th differential in position $(q,p)$ is a map $d_r^{(q,p)}\: E_2^{q,p} \to E_r^{q-r+1,p+r}$.
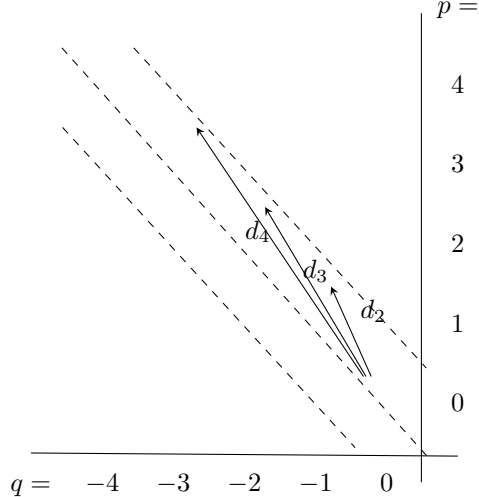
\begin{figure}[h]
\centering
\textbf{Schematic grading convention for the Atiyah--Hirzebruch spectral sequence.}\par\medskip
\begin{tikzpicture}
\matrix (m) [matrix of math nodes,
nodes in empty cells, nodes={minimum width=6ex,
minimum height=6ex},
column sep=.25ex,row sep=.25ex]{
& & & & & & p= &\\
& & \node[](d){};& & & & 4 &\\
& & & \node[] (c){};& & & 3 &\\
& & & &\node[] (b){}; & & 2 & \\
& & & & & & 1 \\
& & & & & \node[] (a){} ; & 0 & \\
q= & -4 & -3 & -2 & -1 & 0 & \\};
\draw[dashed] (m-1-1) -- (m-7-7);
\draw[dashed] (m-1-2) -- (m-6-7);
\draw[dashed] (m-2-1) -- (m-7-6);
\draw[] (m-1-6.east) -- (m-7-6.east);
\draw[] (m-7-1.north) -- (m-7-7.north) ;
\draw [arrow] (a) -- node[anchor=south west]{$d_2$} (b);
\draw [arrow] (a) -- node[anchor=south]{$d_3$} (c);
\draw [arrow] (a) -- node[anchor=south east]{$d_4$} (d);
\end{tikzpicture}\caption{Dotted lines indicate diagonals $p+q=i$ which converge to an associated graded of the $i$-th homotopy; the direction of the first few differential are
indicated. }\label{fig:AHSS_schematic}\end{figure}

We need to understand some aspects of $\HF{3}{*}(\Sigma^{-1}C)$. From Proposition~\ref{prop:cohomology_BUct}, whose proof we defer to the next subsection, we have that
$$\HF{3}{*}(\BUct)\simeq \Z/3[t,c_2,c_3]\cdot u$$ as a module over the Steenrod algebra, where $|t|=2,$ $|c_2|=4$, $|c_3|=6$. We have drawn the $P^1$-action on those classes which are not multiples of $t$ in Diagram~\ref{fig:P1_mod_BUct}. Note that $k^*$
is surjective, so $\HF{3}{*}(C)$ can be identified with a submodule of $\HF{3}{*}(\BUct)$ consisting of all elements except $\Z/3$-multiples of $u$ and $c_2\cdot u$.
\begin{figure}[h]
\centering
\textbf{$P^1$-module structure of $\BUct$ through degree $18$, excluding generators divisible by $t$.}\par\medskip
\begin{tikzpicture}
\matrix (m) [matrix of math nodes,
nodes in empty cells, nodes={minimum width=5ex,
minimum height=5ex},
column sep=.25ex,row sep=.25ex]{
22 && & & & & & & c_2c_3^3\cdot u \\
20&&
& & & &c_2^2c_3^2 \cdot u \\
18&& & & c_2^3c_3\cdot u & & & & c_3^3\cdot u \\
16&&c_2^4\cdot u & & & & c_2c_3^2 \cdot u & & \\
14&& & & c_2^2c_3\cdot u & & & & \\
12&&c_2^3\cdot u & & & & c_3^2 \cdot u & & \\
10&& & & c_2 c_3\cdot u & & & & \\
8&& c_2^2\cdot u& & & & & & \\
6&& & & c_3\cdot u & & & & \\
4&&c_2\cdot u & & & & & & \\
2& & \\
0 & & u \\ };
\draw[dashed] (m-4-3) -- (m-6-3);
\draw[dashed] (m-6-3) -- (m-8-3);
\draw[dashed] (m-10-3) -- (m-12-3);
\draw[dashed] (m-3-5) -- (m-5-5);
\draw[dashed] (m-5-5) -- (m-7-5);
\draw[dashed] (m-2-7) -- (m-4-7);
\draw[dashed] (m-4-7) -- (m-6-7);
\draw[dashed] (m-1-9) -- (m-3-9);
\end{tikzpicture}\caption{The left column is the degree of the cohomology class. Dotted lines indicate $\pm \alpha_1$ attaching maps detected by a $P^1$ in $\HF{3}{*}\BUct$. We omit classes above degree 18 that do not attach to cells at or below 18. }\label{fig:P1_mod_BUct}\end{figure}
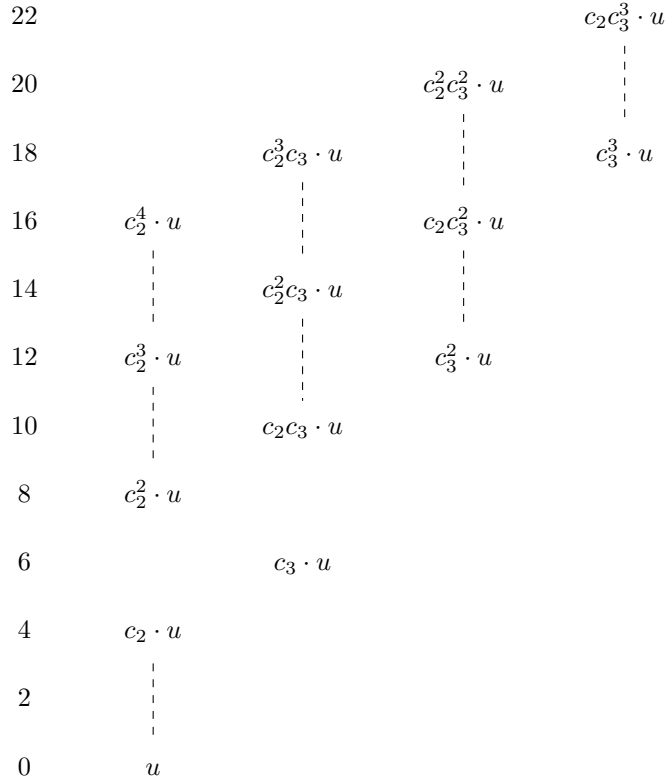 Given a class $t^lc_2^ic_3^j\cdot u \in \HF{3}{*}(C)$, we refer the reader to Proposition~\ref{cor:P1_u_classes} to deduce that
\begin{equation}\label{eq:P1_generators} P^1(t^lc_2^ic_3^j\cdot u) =l(t^{l+2}c_2^ic_3^j\cdot u) + (i+j-1)t^lc_2^{i+1}c_3^j \cdot u. \end{equation}

Since $\HF{3}{*}(\Sigma^{-1}C)$ has odd cohomology, terms on the diagonal $p+q=0$ on the $E_2$-page of the Atiyah--Hirzebruch spectral sequence of Equation~\ref{eq:AHSS_maps} arise from odd elements in $\pi_*C(\alpha_1)\otimes \tmf$. In the relevant range, we can describe
$\pi_*(C(\alpha_1)\otimes \tmf)$ as follows:

$$\pi_{* \leq 26}(C(\alpha_1)\otimes \tmf) \simeq \Z/3\cdot \{1, \alpha_1[4], \alpha_1\beta_1[4],\beta_1[0], \beta^2[0]\}\oplus H,\,\, \text{ where}$$
\begin{itemize}
\item $H$ has only even terms (arising from the classes $c_2$ and $c_4$ in $\pi_*\tmf$) and these terms support no $\alpha_1$ or $\beta_1$ multiplications, so can be disregarded for our calculation.
\item The $\pi_*\sphere$-module structure is given by:
\begin{align*}
\alpha_1\cdot(\alpha_1[4])&=\beta_1[0] &\beta_1\cdot 1& =\beta_1[0] & \\
\alpha_1\cdot (\alpha_1\beta_1[4])&=\beta_1^2[0] &\beta_1\cdot(\beta_1[0])&=\beta_1^2[0] \\ \beta_1\cdot(\alpha_1[4])&=\alpha_1\beta_1[4],
\end{align*}
and all other multiplications are zero.
\item The degree of an indicated class is the degree of the corresponding class in $\pi_*\tmf$ plus the number in the bracket.\footnote{These elements are named by the classes that detect them on the $E_2$-page of the Atiyah--Hirzebruch spectral sequence computing $\pi_*(C(\alpha_1)\otimes \tmf))$.} So, for example, $\alpha_1[4]$ has degree $7$.
\end{itemize}
The only odd classes in $\pi_{*\leq 26}\left(C(\alpha_1)\otimes \tmf\right)$ are $\alpha_1\beta_1[4]$ and $\alpha_1[4]$, so contributions
on the $E_2$-page are in bidegrees $(-17,17)$ and $(-7,7)$.
First consider the classes which are not multiples of $t$:
$$\alpha_1[4]\otimes (c_2^2u), \,\,\, \alpha_1\beta_1[4]\otimes (c_2^3c_3\cdot u), \,\,\, \alpha_1\beta_1[4]\otimes (c_3^3\cdot u).$$
By inspection of Figure~\ref{fig:P1_mod_BUct} and the fact that $\langle \alpha_1,\alpha_1,\alpha_1\rangle =\beta_1$ we see that
\begin{align*}
d_4(\alpha_1[4]\otimes c_2^2\cdot u)&=\beta_1[0]\otimes c_2^3\cdot u\\
d_4(\alpha_1\beta_1[4]\otimes c_3^3\cdot u) &=-\beta_1^2[0]\otimes c_2c_3^3\cdot u\\
\alpha_1\beta_1[4]\otimes c_2^3c_3\cdot u&=d_8(\beta_1[0]\otimes c_2c_3\cdot u)
\end{align*}

In degree $7$ we have
$$t^4\cdot u \,\,\,\,\,\,  \,\,\,\,\,\, c_2^2\cdot u.$$
By Equation~\eqref{eq:P1_generators} we have:
\begin{align*}P^1(t^4 \cdot u)&=t^6\cdot u -t^4c_2\cdot u\\
P^1(c_2^2\cdot u) &= c_2^3\cdot u\end{align*}
Therefore:
\begin{align*}
d_4(\alpha_1[4]\otimes c_2^2\cdot u)&=\beta_1[0]\otimes c_2^3\cdot u\\
d_4(\alpha_1[4]\otimes t^4\cdot u)&=\beta_1[0]\otimes(t^6\cdot u -t^4c_2\cdot u).\end{align*}
Thus the cell $(-7,7)$ does not contribute to the $E_\infty$-page.

In degree $17$ we have generators:
\begin{align*}
t^9 \cdot u & & t^7 c_2\cdot u & & t^6 c_3\cdot u & & t^5 c_2^2 \cdot u\\
t^4 c_2c_3\cdot u & & t^3 c_2^3 \cdot u && t^3 c_3^2 \cdot u & & t^2 c_2^2c_3 \cdot u\\
tc_2^4\cdot u & & tc_2c_3^2\cdot u & & c_2^3c_3\cdot u & & c_3^3\cdot u \\
\end{align*}
Using Equation~\ref{eq:P1_generators}, we see that:

The image of $P^1$ on the generators for $H^{17}(\Sigma^{-1}C)$ listed above is:
\begin{align*}
-t^9c_2 \cdot u & & t^9 c_2\cdot u & & 0 & & -t^7 c_2^2 \cdot u+t^7c_2^3\cdot u\\
t^6 c_2c_3\cdot u+t^4c_2^2c_3\cdot u & & -t^3 c_2^4 \cdot u && t^3 c_2c_3^2 \cdot u & & -t^4 c_2^2c_3 \cdot u -t^2c_2^3c_3\cdot u\\
t^3c_2^4\cdot u & & t^3c_2c_3^2\cdot u-tc_2^2c_3^2\cdot u & & 0 & & -c_2c_3^3\cdot u \\
\end{align*}

Since $d_4(\alpha_1\beta_1[4]\otimes x)=\beta_1^2\otimes P^1(x),$ we see that $\ker(d_4)\: E_4^{(-17,17)} \to E_4^{(-20,21)}$ is generated by $\alpha_1\beta_1$ tensored with:

\begin{align*}
t^9\cdot u+t^7c_2\cdot u & & t^6c_3\cdot u \\ tc_2^4\cdot u + t^3c_2^3\cdot u && c_2^3c_3\cdot u
\end{align*}
The next differential involved is a $d_8$ with source (-10,9), which is computed as $$d_8(\beta_1[0]\otimes y)=\alpha_1\beta_1[4]\otimes P^1P^1(y).$$ Using Equation~\eqref{eq:P1_generators}, we get:
\begin{align*}
P^1\left(P^1(t^5\cdot u)\right)&=P^1\left(P^1(t^5)\right)\cdot u -P^1(t^5)P^1(u) + t^5P^1\left(P^1(u)\right)\\
&= (7*5)t^9\cdot u -(5t^7)(-c_2\cdot u) + t^5\cdot 0=-(t^9\cdot u +t^7c_2\cdot u),\\
P^1\left(P^1(tc_2^2 u)\right)&=P^1\left(P^1(t)\right)\cdot u -P^1(t)P^1(c_2^2u) + tP^1\left(P^1(c_2^2u)\right)\\
&= -(t^3c_2^3\cdot u +tc_2^4\cdot u),\\
P^1\left(P^1(t^2c_3\cdot u)\right)&=P^1\left(P^1(t^2)\right)\cdot u -P^1(t^2)P^1(c_3\cdot u) + t^2P^1\left(P^1(c_3\cdot u)\right)= -t^6c_3\cdot u, \\
\end{align*} since $P^1(c_2\cdot u)=0$. And similarly:
\begin{align*}
P^1\left(P^1(c_2c_3\cdot u)\right)&=P^1\left(P^1(c_2)\right)\cdot u -P^1(c_2)P^1(c_3\cdot u) + c_2P^1\left(P^1(c_3\cdot u)\right)= -c_2^3c_3\cdot u,
\end{align*}
This shows that $\operatorname{gr}\big(\pi_0(\operatorname{Maps}_\sphere(\Sigma^{-1}C, C(\alpha_1)\otimes \tmf)\big)=0,$ which implies the group itself is zero.
\begin{figure}[h]
\centering
\textbf{The relevant part of the $E_2$-page of the Atiyah--Hirzebruch spectral sequence computing $\pi_*\operatorname{Maps}_{\sphere}(\Sigma^{-1}C,C(\alpha_1)\otimes \tmf).$}\par\medskip
\begin{tikzpicture}
\matrix (m) [matrix of math nodes,
nodes in empty cells,nodes={ minimum width=4ex,
minimum height=3ex},
column sep=.1ex,row sep=.1ex]{
& & & & &&&&&&&&&& &\node[](f){}; \\
& \node[circle,draw,yscale=.2,xscale=.2] {}; \node[circle,draw,yscale=.3,xscale=.3] (b){};&&&\node[circle,draw,yscale=.2,xscale=.2] {}; &&&&&&&\node[circle,draw,yscale=.2,xscale=.2] {}; &&&\node[circle,draw,yscale=.2,xscale=.2] {}; &&21
\\ 
\\ 
& \node[circle,draw,yscale=.2,xscale=.2] {}; &&&\node[circle,draw,yscale=.2,xscale=.2] {}; &&&&&&&\node[circle,draw,yscale=.2,xscale=.2] {}; &&&\node[circle,draw,yscale=.2,xscale=.2] {}; &&19
\\ 
& \\ 
& \node[circle,draw,yscale=.2,xscale=.2] {}; &&&\node[circle,draw,yscale=.2,xscale=.2] {}; \node[rectangle,draw,yscale=.5,xscale=.8](a){};&&&&&&&\node[circle,draw,yscale=.2,xscale=.2] {}; &&&\node[circle,draw,yscale=.2,xscale=.2] {}; &&17
\\ 
\\ 
& \node[circle,draw,yscale=.2,xscale=.2] {}; &&&\node[circle,draw,yscale=.2,xscale=.2] {}; &&&&&&&\node[circle,draw,yscale=.2,xscale=.2] {}; &&&\node[circle,draw,yscale=.2,xscale=.2] {}; &&15
\\ 
& & \\
& \node[circle,draw,yscale=.2,xscale=.2] {}; &&&\node[circle,draw,yscale=.2,xscale=.2] {}; &&&&&&&\node[circle,draw,yscale=.2,xscale=.2] {}; &&& \node[circle,draw,yscale=.2,xscale=.2] {};& &13 %
\\ 
& & \\
& \node[circle,draw,yscale=.2,xscale=.2] {}; &&&\node[circle,draw,yscale=.2,xscale=.2] {}; &&&&&&&\node[circle,draw,yscale=.2,xscale=.2](e) {}; &&&\node[circle,draw,yscale=.2,xscale=.2] {}; &&11
\\ 
& \\
& \node[circle,draw,yscale=.2,xscale=.2] {}; &&&\node[circle,draw,yscale=.2,xscale=.2] {}; &&&&&&&\node[circle,draw,yscale=.2,xscale=.2](c) {}; &&& \node[circle,draw,yscale=.2,xscale=.2] {};&&9
\\ 
& & \\ 
& \node[circle,draw,yscale=.2,xscale=.2] {}; &&&\node[circle,draw,yscale=.2,xscale=.2] {}; &&&&&&&\node[circle,draw,yscale=.2,xscale=.2] {}; &&& \node[circle,draw,yscale=.2,xscale=.2](d) {}; \node[rectangle,draw,yscale=.5,xscale=.8]{};&&7
\\ 
\node[](h){}; & & & & & & & & & & & & & & & & \node[](i){};\\
\quad\strut -q & 20& 19&18 &17 & 16& 15& 14& 13 & 12 & 11 & 10 & 9& 8 & 7 &\node[](g) {}; \strut \\
\quad\strut &\smallbetasq &&& {\smallalphabeta} & && && & &\smallbeta & &&\smallalpha & \strut \\};
\draw [arrow,dashed] (a) -- node[anchor=west]{$d_4$} (b);
\draw[arrow,dashed](c) --node[anchor=west]{$d_8$}(a);
\draw[arrow,dashed](d) --node[anchor=west]{$d_4$}(e);
\draw[](f) --(g);
\draw[](h)--(i);
\end{tikzpicture}\caption{The $x$-axis is graded by $q$ with the generator of $\pi_{-q}(C(\alpha_1)\otimes \tmf)$. The $y$-axis is graded by the degree of generators for $\HF{3}{p}(\Sigma^{-1}C)$. A circle in degree $(p,-q)$ indicates a non-zero three-torsion group. Rectangles mark the bi-degrees that can contribute to the $p+q=0$ line. We indicate the differentials that eliminate the terms on $p+q=0$.}\label{fig:tmf_hom_maps}\end{figure}
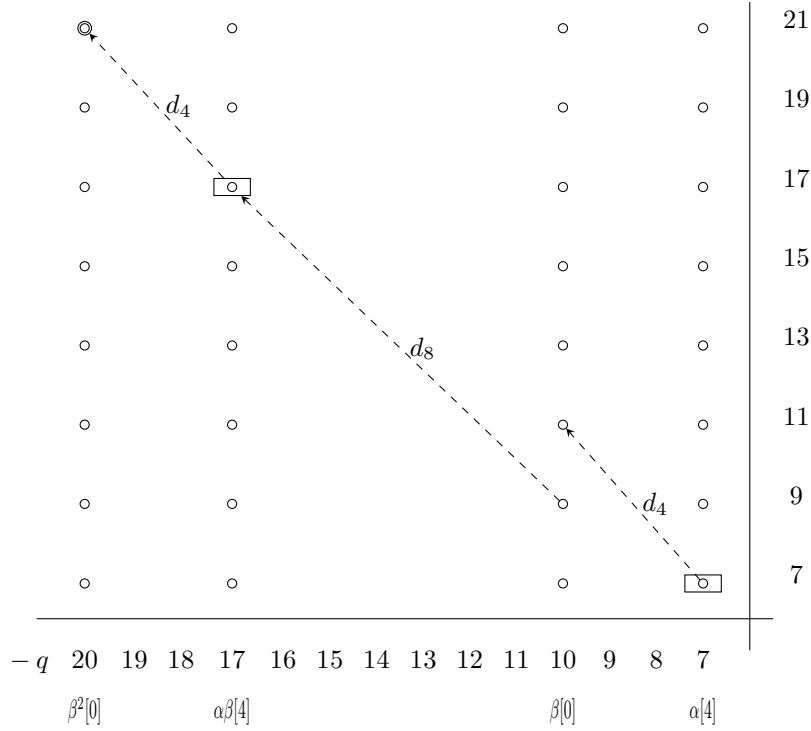
\end{proof}
\begin{rmk} One might hope that $C(\alpha_1)\otimes \tmf$ split from $\sk^i\BUct \otimes \tmf$ for $i>26$. In addition to being a cleaner result, such an extension might allow us to extend the $\rho$ invariant to vector bundles on higher-dimensional spaces. Unfortunately, our approach to splitting $C(\alpha_1)\otimes \tmf$ is computationally intensive and it is not clear how far up the skeleton the splitting extends. However, from discussions with M.J. Hopkins, we believe $\tilde \rho$ can be extended to $\BUct$ by a different method. We hope to work this out in the future, although it is unnecessary for our results.
\end{rmk}

\subsecl{The cohomology of $\BUct$ and related spectra}{subsec:cohomology_BUct}

This subsection includes the main calculations needed for the proof of Theorem~\ref{thm:tmf_class_exists}. Some of these results have already been used in the previous subsection, and some are necessary for the next subsection. We begin with the cohomology of the relevant classifying space and the Thom spectrum of interest.

\begin{prop}\label{prop:cohomology_BUct}
The $P^1$-module structure on the $\pmod 3$-cohomology of $\BUc$ and $\BUct$ is given as follows:
\begin{enumerate}
\item $\HF{3}{*}(\BUc) \simeq \Z/3[t,c_2,c_3]$, where $|t|=2$, $|c_2|=4$, and $|c_3|=6$; and
\begin{align*}
P^1c_2&=c_2^2 & (P^1P^1)c_2&=2c_2^3\\
P^1t&=t^3 & (P^1 P^1)t&=0\\
P^1c_3&=c_2c_3 & (P^1P^1)c_3&=-c_2^2c_3.\\
\end{align*}
\item $\HF{3}{*}(\BUct) \simeq \HF{3}{*}(\BUc)\cdot u,$ with
\begin{align*}
P^1u&=-c_2\cdot u & (P^1 P^1)u &=0\\
\end{align*}
where we view $\HF{3}{*}(\BUc)$ as a $P^1$-algebra and therefore the $P^1$-module structure is determined by the action on $u$.
\end{enumerate}
\end{prop}

\begin{proof}
Part (1) follows from Steenrod operations on Chern classes in $\HF{3}{*}(BU{\hspace{-.6pt}}(3))$, together with the fact the natural map $\BUc \to BU{\hspace{-.6pt}}(3)$ induces $c_1\mapsto 0$, $c_2\mapsto c_2$, and $c_3 \mapsto c_3$.

Part (2) follows from the $\HZt$-Thom isomorphism together with the universal formula for Steenrod operations on Thom classes, which we now explain. However, we will later need formulas for Steenrod operations on the Thom class for the bundle $-\gamma_4$ on $BU{\hspace{-.6pt}}(4)_{\coz}$, so we give these and then deduce those for $-\gamma_3$ on $\BUc$.

We first compute $P^1$ on the Thom class $u_{\gamma_4}$ for $\gamma_4$ on $BU{\hspace{-.6pt}}(4)$ via the universal example of $MU(1)^{\times 4}$.
\begin{align*} P^1(wxyz) &=w^3xyz+wx^3yz+wxy^3z+wxyz^3\\
& =wxyz(w^2+ x^2+y^2+z^2)\\
&=wxyz((w+x+y+z)^2-2(wx+wy+wz+xy+xz+yz)), \end{align*}
implying that
\begin{equation}\label{eq:P1c4}
P^1(u_{\gamma_4})=(c_1^2+c_2)u_{\gamma_4}.
\end{equation}
The Chern classes for $-\gamma_4$ are the coefficients of the power series inverse to
$$c({\gamma_4})=1+c_1t+c_2t^2+c_3t^3+c_4t^4.$$
\begin{align*}
\frac{1}{c(\gamma_4)}&=1-(c_1t+c_2t^2+c_3t^3+c_4t^4)+(c_1t+c_2t^2+c_3t^3+c_4t^4)^2 \\
&\,\,\,\,\,\,\,\,\,\,\,\, -(c_1t+c_2t^2+c_3t^3+c_4t^4)^3 +(c_1t+c_2t^2+c_3t^3+c_4t^4)^4+\dots \\
&= 1-c_1t+(c_1^2-c_2)t^2+(-c_1^3+2c_1c_2-c_3)t^3+(c_1^4-3c_1^2c_2+c_2^2+2c_3c_1-c_4)t^4+,\dots
\end{align*}
so that
\begin{align*}
c_1(-\gamma_4) &= -c_1
& c_2(-\gamma_4) & =c_1^2-c_2
\\c_3(-\gamma_4) & = -c_1^3+2c_1c_2-c_3
&c_4(-\gamma_4)&=c_1^4+c_2^2-c_3c_1-c_4.
\end{align*}
The above implies
\begin{align}\label{eq:P1u}P^1(u_{\smallminus\gamma_4})&=\big( (-c_1)^2 + c_1^2-c_2\big) \cdot u_{\smallminus\gamma_4}= -(c_1^2+c_2)\cdot u_{\smallminus\gamma_4}.\end{align}

We next calculate $P^1c_1^2$ and $P^1c_2$ in $H^*(BU{\hspace{-.6pt}}(4))$. For $c_1$ we have:
\begin{align*}P^1((w+x+y+z)^2)&=2(w+x+y+z)P^1(w+x+y+z)\\
&=2(w+x+y+z)(w+x+y+z)^3=2(w+x+y+z)^4, \,\, \text{ and}
\end{align*}
\begin{align}\label{P1c12}P^1(c_1^2)=-c_1^4.\end{align}
Let $s_n=s_n(w,x,y,z)$ denote the $n$-th elementary symmetric polynomial in $x,y,z,w$. The computation for $c_2$ goes as follows:
\begin{align*}
P^1(wx+wy+wz+xy+xz+yz)&=w^3x+wx^3+w^3y+wy^3+w^3z+wz^3\\ &\,\,\,\,\,\,\,\, +x^3y+xy^3+x^3z+xz^3+y^3z+yz^3\\
&=s_2(w^2+x^2+y^2+z^2)-s_3s_1+xyzw
\\
&=s_2(s_1^2+s_2) - s_3s_1 +s_4
\end{align*}
Therefore:
\begin{align}\label{eq:P1c2}
P^1(c_2)&=c_1^2c_2+c_2^2-c_1c_3+c_4
\end{align}
Combining Equations~\eqref{eq:P1u},\eqref{P1c12}, and \eqref{eq:P1c2} we get:
\begin{align*}P^1 P^1(u_{\smallminus\gamma_4})&=-\big( P^1(c_1^2+c_2)\cdot u_{\smallminus\gamma_4}+ (c_1^2+c_2)P^1(u_{\smallminus\gamma_4})\big)\\
&= -\big(-c_1^4+c_1^2c_2+c_2^2-c_1c_3+c_4 -(c_1^2+c_2)^2) \big)\cdot u_{\smallminus\gamma_4}\\
&=-(c_1^4-c_1^2c_2-c_1c_3+c_4)\cdot u_{\smallminus\gamma_4}.
\end{align*}

Now let $\tilde u$ denote the Thom class of the negative of the universal bundle on $BU{\hspace{-.6pt}}(4)_{\coz}$ and as before let $u$ denote the Thom class of $-\gamma_3$ on $\BUc$. Since the universal bundle has zero first Chern class, we get:
\begin{align}\label{eq:P1_tildeu}
P^1(\tilde u)&=-c_2\cdot \tilde u &P^1 P^1(\tilde u)&=-c_4\cdot \tilde u
\end{align}
\begin{align}\label{eq:P1_real_u}
P^1( u)&=-c_2\cdot u &P^1P^1( u)&=0.
\end{align}
This completes the proof. \end{proof}
From the above we can use the Leibniz rule for $P^1$ to derive:
\begin{cor}\label{cor:P1_u_classes} In $\HF{3}{*}(\BUct)$,
\begin{align*}
P^1(t^i\cdot u)&=it^{i+2}\cdot u -t^ic_3\cdot u\\
P^1(c_2^ic_3^j\cdot u)& = (i+j-1) c_2^{i+1}c_3^j\cdot u
\end{align*}
\end{cor}

To prove that $\tilde \rho$ is unique in Subsection~\ref{subsec:proof_unique}, we will need to analyze a certain cofiber. To that end, the following calculation will be useful.
\begin{prop}\label{prop:P1_module_cofiber} Let $\epsilon\: S^7 \to \BUc$ generate $\pi_7\BUc$, let $u_0$ denote the Thom class for $\epsilon$ and let 
$\Thom_{u_0}(\epsilon)\:\sphere^7 \to \BUct$ be as defined in 
Definition~\ref{def:thom_sigma}.  Let $C(\Thom_{u_0}\!\epsilon)$ denote the cofiber of $\Thom_{u_0}(\epsilon).$ 
As a module over the subalgebra of the mod $3$ Steenrod algebra generated by $P^1$,
\[\HF{3}{*}\left(C(\Thom_{u_0}\!\epsilon)\right) \simeq \HF{3}{*}\BUct \oplus \Z/3\cdot \{y\},\]
where $|y|=8$ and \begin{align*}
P^1(-c_2\cdot u)&=y\\
P^1(y)&=0.
\end{align*}
For integral cohomology, we have an isomorphism of graded abelian groups
\[H^*\left(C(\Thom_{u_0}\!\epsilon);\Z\right) \simeq H^*(\BUct;\Z) \oplus \Z \cdot \{y\},\] where $|y|=8$.
\end{prop}
\begin{proof}
Let $\iota\: BU{\hspace{-.6pt}}(3) \to BU{\hspace{-.6pt}}(4)$ denote the natural map. Because the composite $\iota \circ \epsilon\: S^7 \to BU{\hspace{-.6pt}}(4)$ is null, we get a homotopy commutative diagram
\begin{equation}\label{eq:BU3BU4}
\begin{tikzcd}[row sep=.5in, column sep=.5in]
S^7 \arrow[r,"{}"]\arrow[dr, "0"] & \BUc \arrow[d,"{\gamma_3 \oplus \underline{\C}}" near start]\arrow[r] & (\BUc)/S^7 \arrow[dl, dashed, "\delta"] \\
& BU{\hspace{-.6pt}}(4)_{\coz},&
\end{tikzcd}
\end{equation}
where $\delta$ is any extension of the bundle $\gamma_3 \oplus \underline{\C}$ to the cofiber. We get a homotopy pushout by taking Thom spectra:
\begin{equation}\label{eq:pushout_spectra}
\begin{tikzcd}
\sphere^0 \oplus \sphere^{7} \arrow[r,"p_{1}"]\arrow[d,"\Thom(\epsilon)" left]
& \sphere^0\arrow[d]\\
\BUct\arrow[r]
&\Th{(\BUc)/S^7}{-\delta},
\end{tikzcd}
\end{equation}
where $p_1$ is the projection onto the first factor. Therefore
$$C(\Thom_{u_0}\!\epsilon)\simeq\Th{(\BUc)/S^7}{-\delta}.$$
Let $T:=\BUct$ and $T_4:=\Th{BU{\hspace{-.6pt}}(4)_{\coz}}{\smallminus\gamma_4}$ below. From Equation~\eqref{eq:BU3BU4} we get a commuting diagram
\begin{equation}\label{eq:coh_G}
\begin{tikzcd}[column sep = 10]
H^{i-1}(T)\arrow[r] & H^{i-1}(\sphere^7)\arrow[r] & H^i\left(C(\Thom_{u_0}\!\epsilon)\right)\arrow[r] & H^i(T)\arrow[r] & H^i(S^7)\\
& & H^i(T_4)\arrow[ull]\arrow[ul,"0" above]\arrow[u,"a"]\arrow[ur]\arrow[urr,"0" above]
\end{tikzcd}
\end{equation}
where $H^i$ denotes either $\Z$ or $\Z/3$ coefficients and the top row is exact. The map $a$ induces a ring isomorphism
\begin{equation}\label{eq:cof_is_kinda_BU4} H^*\left(C(\Thom_{u_0}\!\epsilon)\right)/H^{>8}\simeq H^*(\Th{BU{\hspace{-.6pt}}(4)_{\coz}}{\smallminus\gamma_4})/H^{>8},\end{equation}
where $H^{>8}$ is the ideal of elements of degree greater than $8$. Equation~\eqref{eq:cof_is_kinda_BU4} identifies the class $c_4\cdot u$ in mod $3$  cohomology with a class $y \in H^8(C(\Thom_{u_0}\!\epsilon))$. Moreover, for cohomology with mod $3$ coefficients, $y \in \HF{3}{8}\left(C(\Thom_{u_0}\! \epsilon)\right)$ satisfies
\begin{align*}
P^1(-c_2\cdot u)&=y\\
P^1(y)&=0,
\end{align*}
as was to be shown.

To complete the proof, note that Diagram~\eqref{eq:coh_G} implies
\[H^{*}(\BUct) \simeq H^*\left (C(\Thom_{u_0}\!\epsilon)\right )/\langle y\rangle\]
and 
\[\HF{3}{*}(\BUct) \simeq \HF{3}{*}(C(\Thom_{u_0}\!\epsilon))/\langle y\rangle\] as modules over the Steenrod algebra.
\end{proof}

\subsecl{Proof of Proposition~\ref{prop:uniqueness}}{subsec:proof_unique}

Consider the cofibers \[C:=\operatorname{Cof}(\epsilon\:S^7 \to \sk^{26}\BUc)\]
and
\begin{align*}C'&:=\operatorname{Cof}(\sphere^7 \to \sk^{26}\BUct) \simeq \Th{C}{-\delta},
\end{align*}
where $\delta\: C \to BU{\hspace{-.6pt}}(4)$ is an extension of $\gamma_3\oplus \underline{\mathbb C}$ over $C$.
We get a diagram
\begin{equation}\label{eq:0_vee_beta}
\begin{tikzcd}[column sep=1em, row sep=1em]
& \sphere^{10} \arrow[d,"\alpha_1" left] \arrow[dr,"\tilde v"]\\
\Sigma^{-1}C'\arrow[r,"\phi"]& \sphere^{7} \arrow[r]\arrow[d,"\beta_1" left] &\sk^{26}\BUct\arrow[dl,dashed, above] \\
& \Sigma^{-3}\tmf,
\end{tikzcd}
\end{equation} where $\tilde v$ is as in Diagram~\eqref{eq:thomified_htpy2}.
A dotted arrow is an element $\tilde \rho \in \tmf^{-3}\BUct$ satisfying Corollary~\ref{cor:exists}. Choices of the dotted arrow in Diagram~\eqref{eq:0_vee_beta} up to homotopy are a torsor for $\pi_0\operatorname{Maps}_{\sphere}(C',\Sigma^{-4}\tmf).$ We show that this group is zero via an Atiyah--Hirzebruch spectral sequence
$$E_2^{p,q}=H^p(C'; \pi_{-q+3}\tmf)\implies \tmf^{p+q-3}(C'),$$
depicted in Figure~\ref{fig:tmf_choice}.

We computed the cohomology of $C(\Thom_{u_0}\!\epsilon)$ in Proposition~\ref{prop:P1_module_cofiber} and
\[H^{*\leq 26}\left(C(\Thom_{u_0}\!\epsilon)\right)\simeq H^{*\leq 26}(C').\] 
The homotopy of $\tmf$ is known \cite{Bauer,Mathew}. The terms along the line $p+q=0$ on the $E_2$-page are $u \otimes \alpha_1$ and $\alpha_1\beta_1$ tensored with elements in $\HF{3}{10}(C')\simeq \HF{3}{10}(\BUct/S^7)\cdot u.$ Since $\pi_*\tmf$ has no other odd homotopy groups until degree $27$, there are no further possible contributions to consider.

First, consider $u\otimes \alpha_1$. Since $P^1P^1(u)=y$ and $\langle \alpha_1,\alpha_1,\alpha_1\rangle =\beta_1$, $d_8(u \otimes \alpha_1) = \beta_1 \otimes y \neq 0$ and the class $\alpha_1\otimes u$ does not survive the spectral sequence.

On the other hand, there are many bidegree $(-10,10)$ classes to check:

$$E_2^{(-10,10)} \simeq\alpha_1\beta_1\otimes \langle t^5\cdot u, \, t^3c_2\cdot u,\, t^2c_3\cdot u,\,tc_2^2\cdot u,\,c_2c_3\cdot u \rangle.$$
We claim that all classes above support a differential or are the target of a nonzero differential. Figure~\ref{fig:tmf_choice} shows the first interesting differentials in and out of this cell on the $E_2$-page.

\begin{figure}[h]
\centering
\textbf{A larger portion of the $E_2$-page of the Atiyah--Hirzebruch spectral sequence computing $\tmf^*\left(C'\right)$. }\par\medskip
\begin{tikzpicture}
\matrix (m) [matrix of math nodes,
nodes in empty cells,nodes={minimum width=3ex,
minimum height=3ex},
column sep=.1ex,row sep=.1ex]{
& & & & & & & & & & & & p= \\
& \node[circle,draw,yscale=.1,xscale=.1](c) {}; \node[rectangle,draw,yscale=.1,xscale=.1]{};& & & & \node[rectangle,draw,yscale=.1,xscale=.1]{};& & & \node[circle,draw,yscale=.2,xscale=.2] {}; & \node[rectangle,draw,yscale=.1,xscale=.1]{}; & &\node[circle,draw,yscale=.2,xscale=.2] {}; & 18 \\
& & & & & & & & & & & & 17 \\
&\node[circle,draw,yscale=.2,xscale=.2] {}; \node[rectangle,draw,yscale=.1,xscale=.1]{}; & & & & \node[rectangle,draw,yscale=.1,xscale=.1]{};& & & \node[circle,draw,yscale=.2,xscale=.2] {}; & \node[rectangle,draw,yscale=.1,xscale=.1]{};& &\node[circle,draw,yscale=.2,xscale=.2] {}; & 16 \\
& & & & & & & & & & & & 15 \\
&\node[circle,draw,yscale=.2,xscale=.2] {}; \node[rectangle,draw,yscale=.1,xscale=.1](f){}; & & & & \node[rectangle,draw,yscale=.1,xscale=.1]{};& & & \node[circle,draw,yscale=.2,xscale=.2] {}; & \node[rectangle,draw,yscale=.1,xscale=.1]{};& &\node[circle,draw,yscale=.2,xscale=.2] {}; & 14 \\
& & & & & & & & & & & & 13 \\
& \node[circle,draw,yscale=.2,xscale=.2] {}; \node[rectangle,draw,yscale=.1,xscale=.1]{}; & & & & \node[rectangle,draw,yscale=.1,xscale=.1]{};& & \node[](b'){};& \node[circle,draw,yscale=.2,xscale=.2] {}; & \node[rectangle,draw,yscale=.1,xscale=.1]{};& &\node[circle,draw,yscale=.2,xscale=.2] {}; & 12 \\
& & & & & & & & & & & & 11 \\
&\node[circle,draw,yscale=.2,xscale=.2] {}; \node[rectangle,draw,yscale=.1,xscale=.1]{}; & & & & \node[rectangle,draw,yscale=.1,xscale=.1]{};& & & \node[rectangle,draw,yscale=.3,xscale=.3,rotate=45] (b) {}; & \node[rectangle,draw,yscale=.1,xscale=.1]{};& &\node[circle,draw,yscale=.2,xscale=.2] {}; & 10 \\
& & & & & & & & & & & & 9 \\
&\node[circle,draw,yscale=.2,xscale=.2] {}; \node[rectangle,draw,yscale=.1,xscale=.1]{}; & & & & \node[rectangle,draw,yscale=.1,xscale=.1]{};& & & \node[circle,draw,yscale=.2,xscale=.2] {};& \node[rectangle,draw,yscale=.1,xscale=.1]{};& &\node[circle,draw,yscale=.2,xscale=.2] {}; & 8 \\
& & & & & & & & & & && 7 \\
&\node[circle,draw,yscale=.2,xscale=.2] {}; \node[rectangle,draw,yscale=.1,xscale=.1]{}; & & & & \node[rectangle,draw,yscale=.1,xscale=.1]{};& & & \node[circle,draw,yscale=.2,xscale=.2] {}; & \node[rectangle,draw,yscale=.1,xscale=.1]{};& & \node[circle,draw,yscale=.2,xscale=.2] (a){};& 6\\
\quad\strut q=& -17 & -16 & -15 & -14& -13&-12 &-11 &-10 &-9 &-8 &-7\strut \\
\pi_{-q}
&\beta_1^2, & & & & c_4^2 & & & \alpha \beta & c_6 & & \beta \\
& c_4c_6 & & & & & & & & & & \\
};
\draw[dashed] (m-3-2.north) -- (m-13-12.north);
\draw [arrow] (a) -- node[anchor=west]{$d_4$} (b);
\draw [arrow] (b) -- node[anchor=south]{$d_8$} (c);
\end{tikzpicture}\caption{The dashed line indicates the $p+q=0$ line converging to $\pi_0\operatorname{Maps}_{\sphere}(C',\Sigma^{-3}\tmf).$ A dot indicates a non-zero three-torsion group; a square a non-zero torsion-free group. The diamond in bidegree $(-10,10)$ indicates a nonzero $E_2$-terms which may contribute to the computation.}\label{fig:tmf_choice}\end{figure}
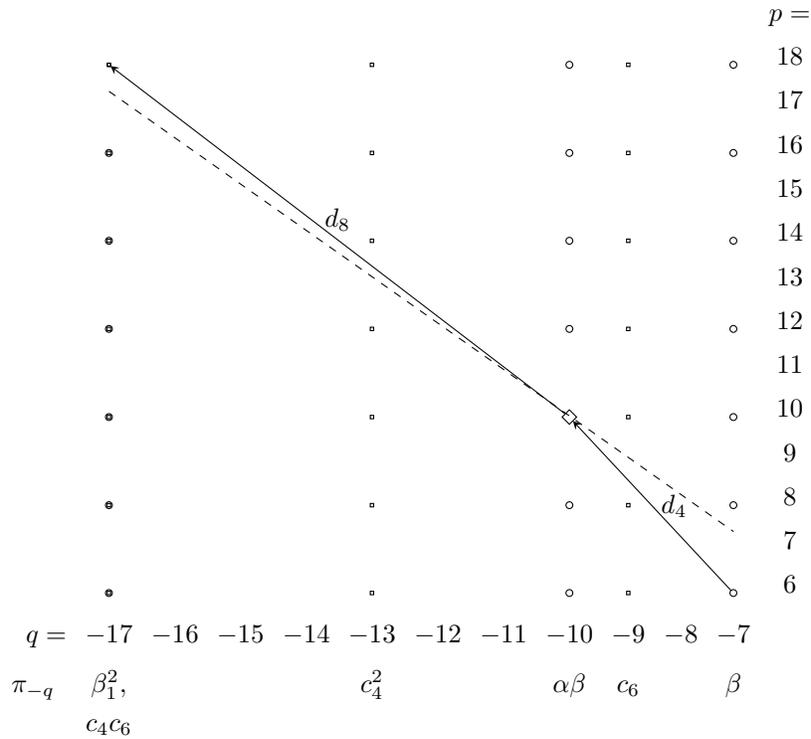

First, we check which of the above is the target of a differential in (A) below. Then we check that the remaining classes support a nonzero differential in (B).

\begin{enumerate}
\item[(A)] The only possible differential is a $d_4$ originating in bidegree $(-7,6)$ is computed as follows: classes in this cell are $\beta_1$ times classes in \begin{equation}\label{eq:h6}\HF{3}{6}(\BUc)\cdot u \simeq \langle t^3 \cdot u, \, tc_2 \cdot u, \, c_3 \cdot u \rangle\end{equation} Proposition~\ref{prop:P1_module_cofiber} implies:
\begin{align*}
P^1(t^3u)& =t^3P^1(u)=-c_2t^3u \\
P^1(tc_2\cdot u)&=t^3c_2\cdot u+tc_2^2\cdot u -tc_2^2\cdot u=t^3c_2\cdot u
\\
P^1(c_3 \cdot u)&=c_2c_3\cdot u -c_2c_3 \cdot u =0
\end{align*}
Therefore: $d_4(\beta_1 \otimes t^3\cdot u)= -\alpha_1\beta_1 \otimes c_2t^3\cdot u$ and the target dies on the $E^5$-page.

\item[(B)] A class $ \beta_1 \alpha_1 \otimes (z \cdot u)$ which survives to the $E^8$-page supports a nonzero $d_8$ if $P^2z$ is nonzero.
Using Proposition~\ref{prop:P1_module_cofiber}, facts about Steenrod operations we get:
\begin{align*} -(P^1 P^1)(t^5\cdot u)&=-(P^1 P^1)(t^5)\cdot u +P^1(t^5)P^1(u) \\&= (t^9+t^7c_2)\cdot u\\ \\
-(P^1 P^1)(t^2c_3\cdot u) & = -(P^1P^1)(t^2)c_3\cdot u+P^1(t^2)P^1(c_3\cdot u)-t^2(P^1P^1)(c_3\cdot u)\\
&=t^6c_3\cdot u\\
-(P^1 P^1)(tc_2^2 \cdot u) &= -(P^1P^1)(tc_2\cdot u) c_2+P^1(tc_2\cdot u)P^1(c_2)-P^1P^1(c_2)(tc_2\cdot u)\\
&= -P^1(t^3c_2 \cdot u)c_2 +t^3c_2^3\cdot u+tc_2^4 \cdot u\\
&= -t^3P^1(c_2\cdot u)+t^3c_2^2\cdot u+tc_2^4u\\
& = t^3c_2^2\cdot u+tc_2^4\cdot u\\ \\
-(P^1 P^1)(c_2c_3\cdot u) & = -(P^1P^1)(c_3 \cdot u)c_2+P^1(c_3\cdot u)P^1(c_2)-(P^1P^1)(c_2)(c_3\cdot u)\\
&= -(P^1P^1)(c_2)c_3\cdot u \\& =c_2^3c_3\cdot u.
\end{align*}
This shows that the classes $t^5 \cdot u, \, t^2c_3\cdot u, \, tc_2^2 \cdot u, \, c_2c_2\cdot u$ support nonzero differentials whose joint span is $4$-dimensional.
\end{enumerate}
No terms on the $p+q=0$ line survive the spectral sequence, so the group in question is zero.

\section{Untwisting the invariant}\label{sec:untwisting}

Our work in Section~\ref{sec:exist} provides an association

$$V \mapsto \Thom(V)^*\tilde \rho \in \tmf^{-3}(\Th{\CP^5}{\smallminus V}).$$
Vector bundles on $\CP^5$ with $c_1\equiv 0 \pmod 3$ and $c_2\equiv 0 \pmod 3$ are $\tmf$-orientable, as we now explain: Any complex bundle $V$ carries an $H\Z$-Thom class. Since $\tau_0\tmf = H\Z$, $V$ is $\tmf$-orientable if and only if this class lifts up the Postnikov tower for $\tmf$. Since $\CP^5$ is finite-dimensional, this is a finite lifting problem. The Postnikov tower for $\tmf$ through degree $10$ has only one stage that obstructs lifting the $\HZ$-Thom class. This gives the condition that $c_1^2-2c_2\equiv 0 \pmod 3$.

Thus, for $V$ over $\CP^5$ with $c_1\equiv 0 \pmod 3$ and $c_2\equiv 0 \pmod 3$, there exist isomorphisms $\tmf^*(\Th{\CP^5}{\smallminus V})\simeq \tmf^*(\suspp \CP^5)$.
However, there are many choices of such an isomorphism and our choice cannot be dependent on $V$.

The ideal way to resolve this would be via a universal example: a space $B$ together with a bundle $V_B$ which carries a (canonical) $\tmf$-orientation, such that all bundles of interest are canonically pullbacks of $V_B$ and therefore inherit a $\tmf$-orientation. Classical examples of this phenomenon are numerous: $BU{\hspace{-.6pt}}(n)$ is canonically $H\Z$ oriented, giving the classical $H\Z$-Thom isomorphism for complex bundles; $BSpin$ is canonically $KO$-oriented via the Atiyah--Bott--Shapiro orientation, giving a canonical $KO$-Thom isomorphism for spin bundles \cite{ABS}; $BString$ is canonically $tmf$-oriented, giving a canonical $tmf$-orientation for string bundles \cite{AHR}.

Our bundles are not string, nor is there an obvious candidate for a universal example. Other, more hands-on, approaches to resolving orientability problems can be found in the literature, for example in \cite{Chat,BhatChat}\footnote{Since $L_{K(2)}TMF=EO_2$ at the prime $3$, the orientability studied in \cite{Chat,BhatChat} is closely related to ours.}. The approach we take here is informed by discussions with H. Chatham.

In Theorem~\ref{thm:well_defined}, we show that for $V\: \CP^5 \to BU{\hspace{-.6pt}}(3)$ with $c_1\equiv 0 \pmod 3$ and $c_2\equiv 0 \pmod 3$, there is a set of Thom isomorphisms subject to some concrete restrictions, with the following property: there is a map $i\:\sphere^{10}\to \suspp \CP^5$ such that, for any Thom isomorphism $v$ in this distinguished set, the image of $\Thom(V)^*(\tilde \rho)$ under the restriction
\begin{equation}\label{eq:indep}\tmf^*(\Th{ \CP^5}{\smallminus V}) \xrightarrow{v}\tmf^*(\suspp \CP^5)\xrightarrow{i^*} \tmf^{-3}(\sphere^{10})\end{equation} does not depend on the choice of $v$. Thus, applying the composite in Equation~\eqref{eq:indep} to the class $\Thom(V)^*(\tilde \rho)$ defines the desired invariant $\rho(V)$.

In Subsection~\ref{subsec:background_orient}, we briefly recapitulate the theory of orientations and establish notation. In Subsection~\ref{subsec:choiceOrientation} we study orientation theory of bundles on $\CP^5$ in detail, define a set of orientations that make Equation~\eqref{eq:indep} independent of choice within this set. We prove independence in Theorem~\ref{thm:well_defined}. We can then define the invariant $\rho$ for $V$ rank $3$ on $\CP^5$ with $c_1\equiv 0 \pmod 3$ and $c_2\equiv 0 \pmod 3$ via the definition indicated in the previous paragraph.

The next Subsection~\ref{subsec:rho_works} features the verification that $\rho$ distinguishes bundles with the same Chern classes, completing the proof of Theorem~\ref{thm:combined}.

The final two subsections include some computations (Subsection~\ref{subsec:sum_line_bundles}) and the observation that $\rho$ can also be defined for rank $2$ bundles (Subsection~\ref{subsec:rank_2}). In these subsections, we also discuss future research questions that we hope to address.

\begin{conv} For all remaining sections except Subsection~\ref{subsec:rank_2}, all spaces and spectra are implicitly localized at the prime $3$. \end{conv}

\subsecl{Background on orientability and orientations}{subsec:background_orient}

We present the relevant background on orientability, orientations, Thom classes, and Thom isomorphisms. The classical version is as follows. Let $V$ be a real rank $n$ vector bundle over a space $X$, with disc bundle $dV$ and sphere bundle $sV$. Let $S^n \to dV$ be the inclusion of a fiber inducing $ i\: S^n \to dV/sV$. Recall that, classically, a Thom class for a vector bundle $V$ over $X$ in generalized cohomology theory $E$ is a class $v \in E^{n}(dV/sV)$ such that $i^*v$ is a unit in $E^0(\sphere^0)\simeq E^n(\sphere^n)$. Orientability refers to the existence of such a class; an orientation is a choice of such a class. Given an such orientation, pairing with the Thom class under the Thom diagonal gives a Thom isomorphism
$$(-)\cdot u\: E^*(\suspp X) \simeq E^*(\Th{X}{V}).$$
\begin{conv} We will use the terms orientation, Thom class, and Thom isomorphism synonymously. \end{conv}
Alternatively, a vector bundle $V\: X \to BU{\hspace{-.6pt}}(n)$ gives rise to a (stable) spherical bundle
\begin{equation}\label{eq:assoc_spherical} V_\sphere:= \left(X \to BU{\hspace{-.6pt}}(n)\to BU \xrightarrow{J}BGL_1\sphere\right),\end{equation}
where $J$ the complex $J$-homomorphism.\footnote{The complex $J$ homomorphism is commonly defined as a map $J\:U \to GL_1\sphere$. Our $J$ is the delooping of the former.} The unit map $1\: \sphere \to E$ induces a map $BGL_11\: BGL_1\sphere \to BGL_1E$ and obtain \begin{equation}\label{eq:assoc_E_bdl} V_E:= \left(X \xrightarrow{V_\sphere}BGL_1\sphere \xrightarrow{BGL_11}BGL_1E\right).\end{equation}
$E$-orientability is the condition that $V_E$ in Equation~\eqref{eq:assoc_E_bdl} is nullhomotopic; an $E$-orientation is a choice of nullhomotopy of $ V_E$ (see \cite{ABGHR2}, building on \cite{May77}).

\begin{rmk}\label{rm:module_iso} An orientation gives not just an isomorphism on cohomology, but an isomorphism of $E$-modules:
$$\Th{X}{V}\otimes E\simeq \suspp X \otimes E.$$
\end{rmk}

\begin{notation}\label{not:stabilized_bdl} In this section we will need to refer to many different maps associated to a given $V\: X \to BU{\hspace{-.6pt}}(r)$. Our notation is slightly abusive but clear from context.
\begin{itemize}
\item $V$ will refer to any of the following associated maps:
\begin{itemize}
\item The definitional map $X \to BU{\hspace{-.6pt}}(r)$,
\item The composite $X \to BU{\hspace{-.6pt}}(r) \to BU$, and
\item The transpose of the previous item under the loops-suspension adjunction, which is a map $\suspp X \to bu$.
\end{itemize}
\item For $E$ a commutative ring spectrum, $V_E$ will refer to both:
\begin{itemize}
\item The composite $X \xrightarrow{V} BU \xrightarrow{J} BGL_1\sphere \xrightarrow{BGL_11} BGL_1E$, and
\item The transpose $ \suspp X \xrightarrow{V} bu \xrightarrow{j} bgl_1\sphere \xrightarrow{bgl_11} bgl_1E$.
\end{itemize}
\end{itemize}
\end{notation}

\subsecl{Selecting orientations and the definition of $\rho$}{subsec:choiceOrientation}

The $3$-local spectrum $\Sigma^\infty_+ \CP^5$ has a splitting arising from the Adams splitting of $\suspp\CP^\infty$: \begin{equation}\label{eq:sum}\suspp \CP^5 \simeq X_0 \oplus X_1,\end{equation}
where
\begin{align*}X_1 &\simeq \Sigma^2C(\alpha_1) \oplus \sphere^{10}, &
X_0&\simeq \suspp H\mathbb P^2 \simeq \sphere^0 \oplus \Sigma^4 C(\alpha_1).\end{align*}
The summand $\suspp\mathbb HP^2$ is split via $p=\Sigma^\infty_+p'$ where $p'\: \mathbb CP^5 \to \mathbb HP^2$ is the map of spaces given by taking the $10$-skeleton of the quotient map $\CP^\infty \to \mathbb H P^\infty$.
\begin{rmk} We fix isomorphisms of $X_0$ and $X_1$ with the sums above. We will write $i$ for all inclusions of summands and $p$ for all projections onto summands. These choices can be made once independent of any future bundles involved.
\end{rmk}

Recall the terminology from Notation~\ref{not:stabilized_bdl}. To study $\tmf$-orientations is to study nullhomotopies of $$V_{\tmf}:\suspp \CP^5 \to bgl_1\tmf.$$ Our strategy is to restrict $V_{\tmf}$ to each summand in the decomposition of Equation~\eqref{eq:sum} and separately study nullhomotopies on each summand. On the summand $X_1$, we show that the bundles of interest possess a certain canonical orientation arising from the image of $j$ spectrum, while the choice on the summand $X_0$ will turn out not to matter.

At a prime $p$, the image of $j$ spectrum $bj$ can be defined as the cofiber of the map $\psi^q-1\: bu \to bu$, where $\psi^q$ is the $q$-th Adams operation and $q$ is a topological generator for $\hat{\Z}_{p}^\times$. By stable Adams conjecture (proved\footnote{Notable previous attempts at the stable Adams conjecture are \cite{Fried,FriedSey}. The Adams conjecture for spaces is proved in \cite{Sull,Quill}.} in \cite{BhatKit}) there is a factorization of the $j$ homomorphism $$j=\big( bu \to bj \xrightarrow{j'} bgl_1\sphere\big),$$ where the map $bu\to bj$ is the natural map to the cofiber of $\psi^q-1$.

\begin{lem}\label{lem:j_homtopy} The map $V_j:=\big(X_1 \xrightarrow{V|_{X_1}} bu\to bj\big)$ is nullhomotopic. Up to homotopy, there is a unique such nullhomotopy.
\end{lem}

\begin{proof} The Atiyah--Hirzebruch spectral sequence for computing $\pi_*\operatorname{Maps}_\sphere(X_1,bj)$ shows that both $\pi_0(\operatorname{Maps}_{\sphere}{X_1},bj)$ and $\pi_0\operatorname{Maps}_{\sphere}({X_1}, \Sigma^{-1}bj)$ are trivial, since $\pi_{\leq 10}bj$ is concentrated in degrees $0, 4,$ and $8$ \cite[Sec. 4]{MahRav} and $\HF{p}{*}X_1$ is concentrated in degrees $2,6,$ and $10$.
\end{proof}
\begin{defn}\label{def:jorient} Let $X$ be a space and let $W\: \suspp X\to bu$ be a given map. Assume that the composite $$W_j=\big(\suspp X \xrightarrow{W} bu \to bj\big)$$ is canonically null homotopic. Given any generalized cohomology theory $E$, the {\em $j$-orientation of $W_{E}$} will refer to the distinguished nullhomotopy
obtained by whiskering the canonical nullhomotopy of $W_j$ with the composite
$bj\xrightarrow{j'}bgl_1\sphere \xrightarrow{bgl_11} bgl_1E.$
\end{defn}
In particular, given a vector bundle $V$ of rank $3$ on $\CP^5$ with $c_1\equiv 0 \pmod 3$ and $c_2\equiv 0 \pmod 3$, $j$-orientation of $(V|_{X_1})_{\tmf}$ will refer to the nullhomotopy of the composite $$X_1 \xrightarrow{V|_{X_1}} bu \to bj \to bgl_1\tmf$$ obtained from the unique nullhomotopy of $(V_{X_1})_j$ given by Lemma~\ref{lem:j_homtopy}.

We are interested in orientations that extend $j$-orientations.
\begin{defn}\label{def:star} Let $Y$ be a space together with a splitting $\suspp Y=Z\oplus X$. Let $E$ be a ring spectrum with $\tau_0E=H\mathbb Z_{(3)}$.
Given a vector bundle $V$ on $Y$ is such that $X \xrightarrow{V|_{X}}bu \to bj$ is canonically null, we say that a $E$-orientation $v$ of $V$ satisfies condition $(*)$ if:
\begin{itemize}
\item[$(*1)$] $v$ restricts to the $j$-orientation of $(V|_{X})_E$; and
\item[$(*2)$] $v$ lifts the canonical $\HZ$-orientation under $0$-truncation $bgl_1E\to bgl_1\HZ_{(3)}$.
\end{itemize}
\end{defn}

The main goal of this section is to prove the following:

\begin{thm}\label{thm:well_defined}Let $V$ be a vector bundle over $\CP^5$ with $c_1(V)\equiv c_2(V)\equiv 0\pmod 3$. Let $v$ be a $\tmf$-orientation for $V$ satisfying condition $(*)$ with respect to the decomposition $\suspp \CP^\infty=X_0 \oplus X_1.$
Then the composite \begin{equation}\label{eq:rho_def}
\begin{tikzcd}[column sep=-18em, row sep=.5em]\rho(V) := \left(\sphere^{10} \xrightarrow{i\otimes 1} \Sigma^\infty_+ \CP^5\otimes \tmf \xrightarrow{v} \Th{\CP^5}{\smallminus V}\otimes \tmf \xrightarrow{\Thom(V)^*(\tilde\rho)}\Sigma^{-3}\tmf\right)
\end{tikzcd}
\end{equation} is independent of $v$, as an element in $\pi_0\left(\operatorname{Maps}_{\sphere}(\sphere^{10},\Sigma^{-3}\tmf) \right)\simeq \pi_{13}\tmf\simeq\Z/3.$
\end{thm}
\begin{rmk} Note that $v$ satisfying the hypothesis of the theorem exists.
Let $$V\: \suspp \mathbb HP^2 \oplus X_1 \to bgl_1\sphere$$ be $\tmf$-orientable. Since $\pi_0\operatorname{Maps}_{\sphere}(X_1,gl_1H\Z)=0$, there is a unique nullhomotopy of any null homotopic map $X_1 \to bgl_1H\Z$. Summing the $j$-orientation of $V|_{X_1}$ with any $\tmf$-orientation of $V|_{\suspp \mathbb HP^2}$ lifting the canonical $H\Z$ orientation gives an appropriate orientation of $V$.\end{rmk}
The rest of this section is devoted to the proof of Theorem~\ref{thm:well_defined}. To begin, suppose that $v$ and $w$ are two $\tmf$-orientations of a bundle $V$, both satisfying $(*)$. Consider the following diagram $\tmf$-modules:
\begin{equation}\label{diag:comms}
\begin{tikzcd}[row sep = .6em]
& \Sigma^\infty_+\CP^5 \otimes \tmf \\
\sphere^{10}\otimes \tmf \arrow[ur, "i"]\arrow[dr,"i" below] & \\
& \Sigma^\infty_+\CP^5 \otimes \tmf \arrow[uu, "w^{-1} v" right],
\end{tikzcd}
\end{equation}
where $w^{-1}v$ is the automorphism of $\CP^5 \otimes \tmf$ obtained from composing the Thom isomorphisms corresponding to $v$ with the inverse of the Thom isomorphism corresponding to $w$ (see Remark~\ref{rm:module_iso}).
Note that, if Diagram~\ref{diag:comms} were to commute up to homotopy, the theorem would be immediate. However, commutativity is stronger than necessary. More precisely, we only need the diagram to commute after applying a certain $\tmf$-cohomology. We will return to this after some preliminary calculations.

Let $\Auttmf$ denote automorphisms in the category of $\tmf$-modules.
We study which elements of $\pi_0\operatorname{Aut}_{\tmf}(\Sigma^\infty_+\CP^5\otimes \tmf)$ arise as ratios of orientations satisfying condition $(*)$. Since $\CP^\infty\simeq X_0 \oplus X_1$, an automorphism $a$ of $\suspp \CP^5\otimes \tmf$ is represented by

\begin{equation}\label{eq:matrix_aut} a=\begin{pmatrix} a_{00} & a_{01} \\
a_{10} & a_{11}
\end{pmatrix}\end{equation} where the $a_{ij} \in \op{Maps}_{\tmf}(X_i\otimes \tmf,X_j\otimes \tmf)$ for $i\neq j$ and $a_{ii} \in \Auttmf(X_i\otimes \tmf).$ To study the failure of Diagram~\ref{diag:comms} to commute, we examine the possibilities for $a_{10}$ and $a_{11}$.

\begin{prop}\label{prop:ratio_thom} Suppose that $v,w$ are two $\tmf$-Thom classes for a $\tmf$-orientable bundle of rank $3$ on $\CP^5$ and that both satisfy condition $(*)$. Let $a=w^{-1}v$ be the associated automorphism of $\suspp\CP^5\otimes \tmf.$ Then $a_{10}\: X_1\otimes \tmf \to \suspp\mathbb HP^2\otimes \tmf$ is null.\end{prop}
\begin{proof}
Consider the cofiber sequence of spectra $X_1 \xrightarrow{i} \suspp \CP^5 \xrightarrow{p} X_0.$ Recall that $p=\suspp p'$, where $p'$ is the $10$-skeleton of the natural map $\CP^\infty \to \mathbb HP^\infty$. We have a diagram
\begin{equation}\label{diag:extension}
\begin{tikzcd}
X_1 \ar[dd]\ar[ddrr,"\,\,\,\,\,\,\,\,\,\,\,\,\,\,\,\,\,\, \implies" {labl, above},"(\dagger)","{V}|_{X_1}" below]\ar[ddrr,bend left,"0"]\\
\\
\suspp \CP^5 \ar[d,"\suspp p' "] \ar[r,"V"] &bu \ar[r] & bj \ar[r,"{j'}"] & bgl_1\sphere\ar[r, "bgl_11"] &bgl_1\tmf. \\
\suspp\mathbb HP^2 \ar[urr,dashed, "q" {below}]
\end{tikzcd}
\end{equation}
In Diagram~\ref{diag:extension}, the nullhomotopy marked $(\dagger)$ is the unique one. The dashed arrow $q$ is determined by the nullhomotopy $(\dagger)$ of the fiber. Given this, a choice $v$ of nullhomotopy $V_{\tmf}$ extending the canonical $j$-orientation of $V|_{X_1}$ is equivalent to a choice $v'$ of nullhomotopy $bgl_11\circ j' \circ q$.

Thus, the map $p'$ of spaces gives rise to a map $\Thom(p')\:\Th{\CP^5}{\smallminus V} \to \Th{\mathbb HP^2}{-j'\circ q}$ participating in a homotopy commutative diagram

\begin{center}
\begin{tikzcd}[column sep=4em, row sep=1em]
\Th{\CP^5}{\smallminus V}\otimes \tmf \ar[d,"v"] \ar[rr,"\operatorname{Th}(p') \otimes \tmf"] && \Th{\mathbb HP^2}{-j'\circ p}\otimes \tmf \ar[d,"v'"] \\
\suspp \CP^5\otimes \tmf \ar[rr,"\suspp p'\otimes \tmf"]&& \suspp \mathbb HP^2\otimes \tmf .
\end{tikzcd}
\end{center}
Applying the same argument to obtain $w$ and $w'$, we get a homotopy commutative diagram

\begin{equation}\label{diag:restrict_to_HP}
\begin{tikzcd}
\suspp \CP^5\otimes \tmf \ar[d,"v"]\ar[rr,"\suspp p'"]&& \suspp \mathbb HP^2\otimes \tmf\ar[d,"v'"] \\
\operatorname{Th}(p')\:\Th{\CP^5}{\smallminus V}\otimes \tmf \ar[d,"w^{-1}"] \ar[rr,"\operatorname{Th}(p') \otimes \tmf"] && \Th{\mathbb HP^2}{-j'\circ q}\otimes \tmf \ar[d,"(w')^{-1}"] \\
\suspp \CP^5\otimes \tmf \ar[rr,"\suspp p'"]&& \suspp \mathbb HP^2\otimes \tmf .
\end{tikzcd}
\end{equation}
Diagram~\ref{diag:restrict_to_HP} implies that $a=w^{-1}v$ has $a_{10}\simeq 0$.\end{proof}

\begin{proof}[Proof of Theorem~\ref{thm:well_defined}]
Given $v$ and $w$ both $\tmf$-orientations for $V$ satisfying condition $(*)$, let $a=w^{-1}v$ and let $a_{11}\:X_1\to X_1$ be as in Equation~\eqref{eq:matrix_aut}. By Proposition~\ref{prop:ratio_thom}, all but the left-most triangle in the following diagram commute:
\begin{equation}\label{diag:comms2}
\begin{tikzcd}[column sep=10, row sep=6]
& X_1\arrow[r]&\Sigma^\infty_+\CP^5 \otimes \tmf\arrow[dr,"w"] \\
\sphere^{10}\otimes \tmf \arrow[ur, "i"]\arrow[dr,"i"] & & & (\CP^5)^{\smallminus V}\otimes \tmf \ar[r,"\operatorname{Th}(V)"] & BU{\hspace{-.6pt}}(3)^{\smallminus\gamma_3} \otimes \tmf.
\\
& X_1\arrow[uu,"a_{11}" right]\arrow[r] & \Sigma^\infty_+\CP^5 \otimes \tmf
\arrow[uu, "w^{-1}v"]\arrow[ur,"v"]
\end{tikzcd}
\end{equation}

To get that $\rho(V)$ as in Diagram~\eqref{eq:rho_def} is well-defined, it suffices to prove that Diagram~\eqref{diag:comms2} commutes after applying the functor $$\tmf^{-3}(-)=\pi_0\left(\op{Maps}_{\tmf}\left((-)\otimes \tmf,\Sigma^{-3}\tmf\right)\right).$$
Note that the map $a_{11} \circ i$ splits as a sum $b_0 \oplus b_1,$ where $b_0\in \Auttmf(\sphere^{10}\otimes \tmf)$ and $b_1\in \Maps_\sphere(\sphere^{10},\Sigma^2C(\alpha_1)\otimes \tmf).$ Since $\tmf^{-3}(\Sigma^2C(\alpha_1))=0$ by an Atiyah--Hirzebruch spectral sequence, it suffices to show that $b_0^*$ is the identity on $\tmf^{-3}(-)$.

We have reduced to studying the $\tmf$-module automorphisms of $\sphere^{10}\otimes \tmf$ which arise from automorphisms of $\suspp \CP^5\otimes \tmf$ of the form $w^{-1}v$ for $v,w$ satisfying conditions $(*)$.

Next, we partially compute the set of all nullhomotopies of $V_{\tmf}\:\suspp \CP^5\to bgl_1\tmf$. This set is a torsor for
$$G:=\pi_0\operatorname{Maps}_{\sphere}(\suspp\CP^5,\Sigma^{-1}blg_1\tmf)=\pi_0\operatorname{Maps}_{\sphere}(\suspp\CP^5,gl_1\tmf).$$
Fix one nullhomotopy $v$ of $V_{\tmf}$ which satisfies the hypotheses of Theorem~\ref{thm:well_defined}.
We can a group homomorphism $h\: G \to \pi_0\operatorname{Aut}_{\tmf}(\suspp \CP^5 \otimes \tmf)$ by
\begin{align*} g&\mapsto \left( \suspp\CP^5 \otimes \tmf \xrightarrow{gv} \Th{\CP^5}{-f}\otimes \tmf \xrightarrow{ v^{-1}} \suspp \CP^5\otimes \tmf\right)
\end{align*}
and a subgroup $G' := \{ h \in G \, | \, hv \text{ satisfies conditions $(*)$ }\} \subset G.$ This induces a group homomorphisms
$$h':=\Big(G'\hookrightarrow G \xrightarrow{h} \pi_0\Auttmf(\suspp\CP^5 \otimes \tmf) \xrightarrow{\tilde\pi} \pi_0\Auttmf(\sphere^{10} \otimes \tmf) \Big),$$ where $\tilde\pi$ takes an automorphism $a$ of $\suspp\CP^5 \otimes \tmf$ to the component $b_0\in \pi_0\Auttmf(\sphere^{10}\otimes \tmf).$
To show that
Diagram~\ref{diag:comms2} commutes after applying $\tmf^{-3}(-)$, it suffices to show that $h'$ is the zero map. We show that $G'$ is a subgroup of $\Z/3 \oplus \Z_{(3)}$. Since $\pi_0\Auttmf(\sphere^{10} \otimes \tmf) \simeq \Z_{(3)}^\times$ and $\Z/3 \oplus \Z_{(3)}$ admits no nontrivial maps to $\Z_{(3)}^\times,$ this suffices.

Given a basepoint for $\CP^5$, we split the zero cell of $\suspp\CP^5$. Requiring a given orientation to lift to the canonical $\HZ$-orientation determines the nullhomotopy on $\sphere^0$. Therefore we have that $G' \subset \pi_0 \operatorname{Maps}_{\operatorname{Sp}}(\susp\CP^5,gl_1\tmf).$ We compute this group via an Atiyah--Hirzebruch spectral sequence shown in Figure~\ref{fig:gl1_CP5} below.

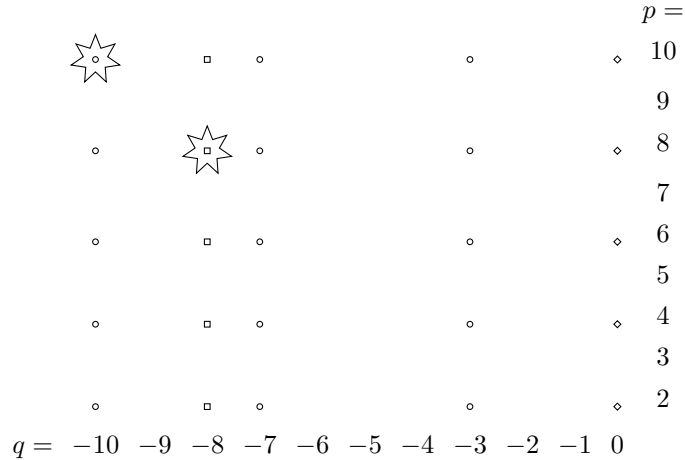
\begin{figure}[h]
\centering
\textbf{Nonzero terms on the $E_2$-page of the Atiyah--Hirzebruch spectral sequence $H^p(\susp \CP^5; \pi_{-q}gl_1\tmf) \implies \pi_{p+q}\operatorname{Maps}_{\sphere}(\susp\CP^5,gl_1\tmf)$. }\par\medskip
\begin{tikzpicture}
\matrix (m) [matrix of math nodes,
nodes in empty cells,nodes={minimum width=2.5ex,
minimum height=2.5ex},
column sep=.1ex,row sep=.1ex]{
& & & & & & & & & & & & p= \\
& \node[circle,draw,yscale=.2,xscale=.2] {}; \node[star,star points=7,star point ratio=2,minimum height=.1cm,minimum width=.1cm,draw] {}; & &\node[rectangle,draw,yscale=.2,xscale=.2]{}; & \node[circle,draw,yscale=.2,xscale=.2] {}; & & & & \node[circle,draw,yscale=.2,xscale=.2] {}; & & & \node[diamond,draw,yscale=.2,xscale=.2] {}; & 10 \\
& & & & & & & & & & & & 9 \\
& \node[circle,draw,yscale=.2,xscale=.2] {}; & &\node[rectangle,draw,yscale=.2,xscale=.2]{};\node[star,star points=7,star point ratio=2,minimum height=.1cm,minimum width=.1cm,draw] {}; & \node[circle,draw,yscale=.2,xscale=.2] {}; & & & & \node[circle,draw,yscale=.2,xscale=.2] {}; & & & \node[diamond,draw,yscale=.2,xscale=.2] {}; & 8 \\
& & & & & & & & & & & & 7 \\
& \node[circle,draw,yscale=.2,xscale=.2] {}; & &\node[rectangle,draw,yscale=.2,xscale=.2]{}; & \node[circle,draw,yscale=.2,xscale=.2] {}; & & & & \node[circle,draw,yscale=.2,xscale=.2] {}; & & & \node[diamond,draw,yscale=.2,xscale=.2] {}; & 6 \\
& & & & & & & & & & & & 5 \\
& \node[circle,draw,yscale=.2,xscale=.2] {}; & &\node[rectangle,draw,yscale=.2,xscale=.2]{}; & \node[circle,draw,yscale=.2,xscale=.2] {}; & & & & \node[circle,draw,yscale=.2,xscale=.2] {}; & & & \node[diamond,draw,yscale=.2,xscale=.2] {}; & 4 \\
& & & & & & & & & & & & 3 \\
& \node[circle,draw,yscale=.2,xscale=.2] {}; & &\node[rectangle,draw,yscale=.2,xscale=.2]{}; & \node[circle,draw,yscale=.2,xscale=.2] {}; & & & & \node[circle,draw,yscale=.2,xscale=.2] {}; & & & \node[diamond,draw,yscale=.2,xscale=.2] {}; & 2 \\
\quad\strut q=& -10 & -9 & -8& -7& -6&-5 &-4 &-3 &-2 &-1 &0\strut \\
};
\end{tikzpicture}\caption{A circle indicates a non-zero three-torsion group; a square a non-zero torsion-free group; and a diamond a group $\Z_{(3)}^\times$. The stars in bidegree $(-10,10)$ and $(-8,8)$ indicate terms along the $p+q=0$ line which contribute.}\label{fig:gl1_CP5}\end{figure}

Thus, there is an extension of $\Z_{(3)}$-modules
$$0 \to \Z/3 \to \pi_0 \operatorname{Maps}_{\operatorname{Sp}}(\susp\CP^5,gl_1\tmf) \to \Z_{(3)} \to 0.$$
Since $\operatorname{Ext}^1_{\Z_{(3)}}(\Z_{(3)},\Z/3)=0$, the group is $\Z_{(3)}\oplus \Z/3.$
\end{proof}

\subsecl{The invariant $\rho$ separates Chern classes for rank $3$ bundles on $\CP^5$}{subsec:rho_works}

Our next goal is to show that the invariant $\rho$ defined by Theorem~\ref{thm:well_defined} distinguishes vector bundles of rank $3$ on $\CP^5$ with the same Chern classes. Recall that $\pi_{10}\BUc\simeq \Z/3$ acts on $[\CP^5,\BUc]$ as in Construction~\ref{const:action_Z3}. Given $(f,\sigma) \in [\CP^5,\BUc]\times \pi_{10}\BUc,$
$$(f,\sigma) \, \mapsto \, \sigma V:= \Big( \CP^5 \xrightarrow{Q} \CP^5 \vee S^{10}\xrightarrow{f\vee \sigma} \BUc\Big).$$
Fix $a_1,a_2,$ and $a_3$ with $a_1\equiv 0\pmod 3$. By Theorem~\ref{classification}, this action restricts to a transitive one on each set
$$\V_{a_1,a_2,a_3}=\{ V\: \CP^5 \to \BUc \, | \, c_i(V)=a_i\}/ \sim,$$ which is free if only if $a_2\equiv 0 \pmod 3$.

By Theorem~\ref{thm:tmf_class_exists} we have a Thomification isomorphism
$$t\: \pi_{10}(\BUc) \to \pi_{10}(\Sigma^{-3}\tmf),$$
defined relative to an orientation for a generator of $\pi_{10}\BUct$. More precisely, let $\sigma \in \pi_{10}\BUct$ and an take an orientation $v_0$ for $\sigma$ satisfying condition $(*)$ from Definition~\ref{def:star} with respect to the splitting $\suspp S^{10}\simeq \sphere^0 \oplus \sphere^{10}$. Then
\begin{equation}\label{def:t}t(\sigma):=\tilde\rho\circ \Thom_{v_0}(\sigma),\end{equation}
with $\Thom_{v_0}$ as in Definition~\ref{def:thom_sigma}.
\begin{conv} We write $v_0$ for both the nullhomotopy of $\sigma\: \suspp S^{10} \to bgl_1\sphere$ and the associated nullhomotopy of $bgl_11\circ \sigma$.
\end{conv}
Our main goal in this subsection is to prove:
\begin{prop}\label{prop:rho_works} Let $V$ be a rank $3$ vector bundle on $\CP^5$ with $c_1\equiv 0 \pmod 3$ and $c_2 \equiv 0 \pmod 3$. The $\rho(\sigma V)= \rho(V)+t(\sigma).$
\end{prop}
This implies that the invariant $\rho$ separates Chern classes as follows.
\begin{cor}\label{cor:rho_works} If $V_1,$ $V_2$ are rank $3$ vector bundles on $\CP^5$ that have the same Chern classes, and such that $$c_1(V_1)=c_1(V_2)\equiv 0 \pmod 3, \,\,\,\,\,\,c_2(V_1)=c_2(V_2)\equiv 0 \pmod 3.$$
Then $\rho(V_1)=\rho(V_2)$ if and only if $V_1$ and $V_2$ are represented by homotopic maps to $BU{\hspace{-.6pt}}(3)$, i.e. if and only if $V_1$ and $V_2$ are topologically equivalent.
\end{cor}
\begin{proof}[Proof of Corollary~\ref{cor:rho_works} assuming Proposition~\ref{prop:rho_works}] Since $\pi_{10}(\BUc)$ acts transitively, there is some element $\sigma \in \pi_{10}(\BUc)$ such that $$V_1 = \sigma V_2 \text{ and }\rho(V_2)=\rho(V_1)+t(\sigma).$$
Since $t$ is an isomorphism, $\rho(V_1)=\rho(V_2)$ if and only if $\sigma=0$ if and only if $V_1\simeq V_2$.
\end{proof}

\begin{proof}[Proof of Proposition~\ref{prop:rho_works}]
Consider the diagram
\begin{equation}\label{eq:wedge_thom}
\begin{tikzcd}[row sep=2em]
\suspp \CP^5\ar[rr,"\sigma V" above] \ar[d,"\suspp Q" left]
& & bgl_1\tmf . \\
\suspp \CP^5 \oplus \sphere^{10}, \ar[urr,"\,\,\,\,\,V \oplus \sigma" below]
\end{tikzcd}
\end{equation}
where $Q\: \CP^5 \to \CP^5 \vee S^{10}$ is as in Construction~\ref{const:action_Z3}. By Equation~\ref{eq:sum}, $\suspp \CP^5 \simeq \suspp\HP^2 \oplus \Sigma^2C(\alpha_2)\oplus \sphere^{10}$. Under this identification, $$\suspp Q=1 \oplus 1 \oplus \Delta \: \suspp\HP^2 \oplus \Sigma^2C(\alpha_2)\oplus \sphere^{10} \to \suspp\HP^2 \oplus \Sigma^2C(\alpha_2)\oplus (\sphere^{10} \oplus \sphere^{10}).$$

Let $X_1':=\Sigma^2C(\alpha_2)\oplus (\sphere^{10}\oplus \sphere^{10})$. We obtain an orientation $v'$ of $V \oplus \sigma$ by summing the $j$-orientation of $(V\oplus \sigma|_{X_1'})_{\tmf}$ with any nullhomotopy of $(V|_{\suspp \HP^2})_{\tmf}$ that lifts the canonical $\HZ_{(3)}$-orientation. By construction, $v'$ satisfies $(*)$. This induces a nullhomotopy $v$ of $(\sigma V)_{\tmf}$ and a nullhomotopy $\bar v$ of $V$, both of which satisfy condition $(*)$. Thus we get a commuting diagram of $\tmf$-modules:

\begin{center}
\begin{tikzcd}
& & \Sigma^{-3}\tmf
\\
(\CP^5)^{-\sigma V}\otimes \tmf \ar[r,"\operatorname{Th}(Q)"]\arrow[urr, "\,\,(\sigma V)^*\tilde \rho"above, bend left=10] & (\CP^5 \oplus \sphere^{10})^{-(V\oplus \sigma)}\otimes \tmf \arrow[ur,"(V \oplus \sigma)^*\tilde\rho"]
\\
\suspp \CP^5 \otimes \tmf \ar[r,"Q"]\ar[u,"v^{-1}"] & (\suspp \CP^5 \oplus \sphere^{10})\otimes \tmf \ar[u,"(v')^{-1}"]
,\end{tikzcd}
\end{center}
where we suppress tensoring with $\tmf$ from the horizontal arrows.
Below, the pushout of spaces on the left induces the diagram of Thom spectra on the right:
\begin{equation*}
\begin{tikzcd}
* \ar[r] \ar[d]& S^{10} \ar[d]
& & \sphere^0 \ar[r]\ar[d] & (S^{10})^{-\sigma}\ar[d] \\
\CP^5\ar[r]& \CP^5 \vee S^{10}
& & (\CP^5)^{\smallminus V} \ar[r]& (\CP^5 \vee S^{10})^{\smallminus V\oplus \sigma}.
\end{tikzcd}
\end{equation*}
The $j$-orientation for $\sigma_\sphere$ gives an equivalence \begin{equation}\label{eq:equiv_sum}(\CP5 \vee S^{10})^{\smallminus V \oplus \sigma} \simeq (\CP^5)^{\smallminus V} \oplus \sphere^{10}.\end{equation} The nullhomotopy $v'|_{\sphere^{10}}$ of $\sigma_{\tmf}$ extends the $j$-orientation. So, using the identification \eqref{eq:equiv_sum}, we have that $\Thom(V \oplus \sigma)^*\tilde \rho=\Thom(V)^*\tilde \rho \oplus t(\sigma)$.
Thus we get the homotopy commutative Diagram~\eqref{diag:commutes7} of $\tmf$-modules below:
\begin{equation}\label{diag:commutes7}
\begin{tikzcd}[column sep = 3em]
(\CP^5)^{-\sigma V}\otimes \tmf\ar[rrr, "\Thom(\sigma V)^*\tilde \rho" above, bend left]
& \big((\CP^5)^{\smallminus V} \oplus \sphere^{10}\big)\otimes \tmf \ar[rr,"\Thom(V)^*\tilde \rho\oplus t(\sigma)"] & &\Sigma^{-3}\tmf\\
\suspp \CP^5\otimes \tmf \ar[r,"Q"]\ar[u,"v^{-1}"] & \big(\suspp \CP^5 \oplus \sphere^{10}\big)\otimes \tmf \ar[u,"{\bar v}^{-1}\oplus 1" right] \\
\sphere^{10}\otimes \tmf \ar[uurrr, bend right=45, "\,\,\,\,\,\,\,\,\,\,\,\,\,\,\,\,\,\rho(V)+t(\sigma)" below]\ar[u,"i"]\ar[r,"\Delta\otimes \tmf"] & \big(\sphere^{10} \oplus \sphere^{10}\big)\otimes \tmf\ar[u,"i\oplus 1" right] & & .
\end{tikzcd}
\end{equation}
Comparing the two outer paths from the lower left-hand corner to $\Sigma^{-3}\tmf$ gives the result.
\end{proof}

\subsecl{Computing $\rho$ on certain sums of line bundles}{subsec:sum_line_bundles}

Let $\rho$ be as defined by Theorem~\ref{thm:well_defined}. For $L$ a line bundle, we define
$\rho(L):=\rho(L\oplus \underline{\C}^2)$.
\begin{lem}\label{lem:sum_line_bundles}
Suppose that $\O(a_1), \mathcal O(a_2)$ and $\mathcal O(a_3)$ are line bundles on $\CP^5$ with $a_i\equiv 0 \pmod 3$. Then
$\rho\left(\mathcal O(a_1)\oplus \mathcal O(a_2) \oplus \mathcal O(a_3)\right)=\rho(\O(a_1))+\rho(\O(a_2))+\rho(\O(a_3)) \in \Z/3.$
\end{lem}
This immediately implies:
\begin{cor}\label{cor:three_times_bundle} Let $a$ be an integer divisible by $3$. Then $\rho(\O(a)^{\oplus 3}) = 0.$
\end{cor}
\begin{proof}[Proof of Lemma~\ref{lem:sum_line_bundles}] Let $V:=\oplus_{i=1}^3\O(a_i)$. Let $\tilde V$ be the bundle on $(\CP^5)^{\times 3}$ given by $$\tilde V:=\oplus_{i=1}^3p_i^*\O(a_i),$$ where $p_i$ is the $i$-th projection. We can factor $V\: \CP^5 \to \BUc$ as follows
$$V\: \CP^5 \xrightarrow{\Delta} {\CP^5}^{\times 3} \xrightarrow{\tilde V} \BUc$$
and naturally identify $\Th{(\CP^5)^{\times 3}}{-\tilde V} \simeq \otimes_i\Th{\CP^5}{-\O(a_i)}.$ Using $\tmf$-orientations satisfying $(*)$ \footnote{See Definition~\ref{def:star}.} for each $\O(a_i)$, we get a diagram of $\tmf$-modules:

\begin{center}
\begin{tikzcd}
\Th{\CP^5}{\smallminus V}
\ar[r,"\op{Th}(\Delta)"]
& \otimes_i\Th{\CP^5}{-\O(a_i)} \ar[r,"\op{Th}(\tilde V)"]
&\BUct \ar[r,"\tilde \rho"]
& \Sigma^{-3}\tmf \\
\suspp\CP^5 \ar[r,"\suspp \Delta"] \ar[u,"\simeq_{\tmf}"]
& (\suspp \CP^5)^{\otimes 3} \ar[u,"\simeq_{\tmf}"]\\
\sphere^{10}\ar[u] \ar[r,"\Sigma^\infty \Delta"] & (\sphere^{10})^{\otimes 3}\ar[u]\\
\sphere^{10}\ar[uuurrr,dashed,bend right=30, "\rho(V)" below]\ar[u]\ar[r,"1 \oplus 1 \oplus 1",dashed]& \sphere^{10}\oplus \sphere^{10}\oplus \sphere^{10}\ar[u]\ar[uuurr,dashed,"\oplus_i\rho(\O(a_i))\,\,\,\,\,\,\,\,\,\,\,\,\,\,\,\,\,\,\,\,\,\,\,\,\,\," above]
\end{tikzcd}
\end{center}
where all terms in the diagram are implicitly tensored with $\tmf$ and the maps marked $\simeq_{\tmf}$ are $\tmf$-Thom isomorphisms.
The diagram is homotopy commutative and comparing the dashed arrows proves the Lemma.\end{proof}

While this section provides some computations of $\rho$, we do not have a general formula. Indeed, it is unclear what a formula for $\rho$ should look like, since $\rho$ cannot be computed from Chern classes. Some inspiration can be drawn from \cite{AR}, where Atiyah and Rees show that the $\alpha$ invariant of a rank $2$ bundles on $\CP^3$ can be computed as a holomorphic semi-characteristic \cite[Theorem 4.2]{AR}, provided we choose a holomorphic representative for the topological class of the bundle. This leads to the following question:
\begin{q} For $V$ an algebraic vector bundle on $\CP^5$, is there some description of the invariant $\rho(V)$ in terms of sheaf cohomology of $V$?
\end{q}

\subsecl{A $3$-torsion $\tmf$-valued invariant for rank $2$ bundles}{subsec:rank_2}

The homotopy groups of $BU{\hspace{-.6pt}}(3)$ through degree $10$ are fairly sparse, so a complete analysis of $[\CP^5,BU{\hspace{-.6pt}}(3)]$ was possible. This allowed us to classify rank $3$ bundles on $\CP^5$ by first enumerating such bundles and second defining an invariant to distinguish them.
It turns out that the invariant $\rho$ is also interesting in the case of rank $2$ bundles on $\CP^5$.
\begin{prop}\label{prop:rank2} Let $BU{\hspace{-.6pt}}(2)_{\coz}$ denote the homotopy fiber of the map $c_1 \pmod{3}\: BU(2) \to K(\Z/3,2).$ Let $\gamma_2$ denote the pullback of the universal bundle on $BU(2)$ to $BU{\hspace{-.6pt}}(2)_{\coz}$.  Consider the composite
\[\sk^{26}\Th{BU{\hspace{-.6pt}}(2)_{\coz}}{\smallminus\gamma_2}\to \sk^{26}\BUct \xrightarrow{\tilde \rho} \Sigma^{-3}\tmf,\] where $\tilde \rho$ represents the class defined in Theorem~\ref{thm:tmf_class_exists}. The map has the following properties:
\begin{enumerate}
\item[a.] The map induced map $\pi_{10}BU{\hspace{-.6pt}}(2)_{\coz} \to \pi_{10}\Sigma^{-3}\tmf$ is a bijection.
\item[b.] The invariant $V \mapsto \rho(V\oplus \underline{\C})$ distinguishes $3$-local equivalence classes of rank $2$ vector bundles on $\CP^5$ with fixed $c_1,c_2$ where additionally $c_1\equiv 0\pmod 3$.
\end{enumerate}
\end{prop}
\begin{proof}
For (a), recall Diagram~\ref{eq:alpha12}. Note that the unstable generator for $\pi_4BU{\hspace{-.6pt}}(3)$ factors through $BU{\hspace{-.6pt}}(2)$, so Theorem~\ref{thm:tmf_class_exists} shows that the image of a generator for $\pi_{10}BU{\hspace{-.6pt}}(2)_{\coz}$ is nonzero. Therefore it suffices to check that $\pi_{10}BU{\hspace{-.6pt}}(2)\simeq \Z/3$. This is classical: since $BSU{\hspace{-.6pt}}(2)\simeq BS^3$, the homotopy in the relevant range is given in Figure~\ref{fig:homotopy_BU2}.

\begin{figure}[h]
\begin{tabular}{| M{1.2cm} | M{1cm} | M{1cm} | M{1cm} | M{1cm} |M{1cm} | M{1cm} | M{1cm} | M{1cm} | M{1cm} | M{1cm} | N}
\hline
& \textbf{$\pi_2$} & \textbf{$\pi_3$} & \textbf{$\pi_4$} & \textbf{$\pi_5$} & \textbf{$\pi_6$} & \textbf{$\pi_7$} & \textbf{$\pi_8$} & \textbf{$\pi_9$} & \textbf{$\pi_{10}$} \\
\hline
& & & & & & & & &
\\[-8pt]
$BU{\hspace{-.6pt}}(2)$ &$\Z$ & $0$ &$\Z$ &$\Z/2$ &$\Z/2$ &$\Z/12$ &$\Z/2$ &$\Z/2$ &$\Z/3$ \\[2pt]
\hline
\end {tabular}\caption{Homotopy of $BU{\hspace{-.6pt}}(2)$}\label{fig:homotopy_BU2}
\end{figure}

Since $\CP^5$ has even cohomology, only even $3$-local homotopy gives rise to $3$-local invariants (odd $3$-local homotopy contributes to constraints on possible Chern classes). Therefore a argument as in Remark~\ref{rmk:top_cell} shows that the $\pi_{10}BU{\hspace{-.6pt}}(3)$-action as in Construction~\ref{const:action_Z3} is the only source of $3$-local invariants beyond Chern classes; for rank $2$ bundles on $\CP^5$ with $c_1\equiv 0 \pmod 3$, these bundles detected by $\rho$.
\end{proof}
By work of Switzer, distinguishing rank $2$ vector bundles on $\CP^5$ with the same Chern classes is a $3$-local problem: there are either $0$, $1$, or $3$ vector bundles of rank $2$ on $\CP^5$ with any fixed integers as their Chern classes. All three options can occur, depending on various conditions on Chern classes. A necessary condition for there to be distinct bundles with the same Chern classes $c_1,c_2\in \Z$ is that $c_1^2\equiv c_2 \pmod{3}$ (see \cite[Theorem 4]{Switzer2} and \cite[Chapter 1, Section 6.1]{OSS}). In the case that $c_1\equiv c_2 \equiv 0 \pmod {3}$,  the invariant $\rho$ is defined and  distinguishes these bundles by Proposition~\ref{prop:rank2}.

\bibliographystyle{abbrv}
\bibliography{r3p5}

\begin{thebibliography}{10}

\bibitem{ABGHR}
M.~Ando, A.~J. Blumberg, D.~Gepner, M.~J. Hopkins, and C.~Rezk.
\newblock Units of ring spectra and thom spectra.
\newblock {\em Journal of Topology}, 7(4):1077--1117, 2014.

\bibitem{ABGHR2}
M.~Ando, A.~J. Blumberg, D.~Gepner, M.~J. Hopkins, and C.~Rezk.
\newblock Units of ring spectra, orientations and thom spectra via rigid
  infinite loop space theory.
\newblock {\em Journal of Topology}, 7(4):1077--1117, 2014.

\bibitem{AHR}
M.~Ando, M.~J. Hopkins, and C.~Rezk.
\newblock Multiplicative orientations of $\ko$-theory and the spectrum of
  topological modular forms.
\newblock Available at
  faculty.math.illinois.edu/$\sim$mando/papers/koandtmf.pdf, 2010.

\bibitem{AE}
B.~Antieau and E.~Elmanto.
\newblock A primer for unstable motivic homotopy theory.
\newblock {\em Surveys on recent developments in algebraic geometry},
  95:305--370, 2017.

\bibitem{ABS}
M.~Atiyah, R.~Bott, and A.~Shapiro.
\newblock Clifford modules.
\newblock {\em Topology}, 3(1):3--38, 1964.

\bibitem{AR}
M.~Atiyah and E.~Rees.
\newblock Vector bundles on projective 3-space.
\newblock {\em Inventiones Math.}, 35:131--153, 1976.

\bibitem{Bauer}
T.~Bauer.
\newblock Computation of the homotopy of the spectrum $tmf$.
\newblock {\em Geom. Topolo. Monogr.}, 13:11--40, 2008.

\bibitem{BhatChat}
P.~Bhattacharya and H.~Chatham.
\newblock On the {$EO$}-orientability of vector bundles.
\newblock Available at arXiv:2003.03795v3, 2020.

\bibitem{BhatKit}
P.~Bhattacharya and N.~Kitchloo.
\newblock The stable {A}dams conjecture and higher associative structures on
  {M}oore spectra.
\newblock {\em Ann. of Math.}, 195:375--420, 2022.

\bibitem{Chat}
H.~Chatham.
\newblock An orientation map for height $p-1$ real $e$ theory.
\newblock 2019.

\bibitem{DFHH}
C.~L. Douglas, J.~Francis, A.~G. Henriques, and M.~A. Hill.
\newblock {\em Topological modular forms}, volume 201 of {\em Mathematical
  Surveys and Monographs}.
\newblock Amer. Math. Soc., Providence, RI, 2010.

\bibitem{Fried}
E.~M. Friedlander.
\newblock The infinite loop adams conjecture via classification theorems for
  {$\mathcal F$}-spaces.
\newblock {\em Math. Proc. Cambridge Philos. Soc.}, 87(1):109--150, 1980.

\bibitem{FriedSey}
E.~M. Friedlander and R.~M. Seymour.
\newblock Two proofs of the stable adams conjecture.
\newblock {\em Bull. Amer. Math. Soc.}, 83(6):1300--1302, 1977.

\bibitem{Hu}
Y.~Hu.
\newblock Metastable complex vector bundles over complex projective spaces.
\newblock arXiv:2202.11800, 2022.

\bibitem{KS}
K.~Kodaira and D.~C. Spencer.
\newblock Divisor class groups on algebraic varieties.
\newblock {\em Proc. Nat. Acad. Sci. U.S.A.}, 39:872--877, 1953.

\bibitem{Kudo56}
T.~Kudo.
\newblock A transgression theorem.
\newblock {\em Mem. Fac. Sci. Kyusyu Univ.}, A(9):79--81, 1956.

\bibitem{MahRav}
M.~Mahowald and D.~Ravenel.
\newblock Towards a global understanding of the stable homotopy groups of
  spheres.
\newblock In S.~Gitler, editor, {\em The Leftschets Centennial Conference:
  Proceedings on Algebraic Topology II}, volume~58 of {\em Contemporary
  Mathematics}, pages 109--118. AMS, 1987.

\bibitem{Mathew}
A.~Mathew.
\newblock The homotopy groups of {$\TMF$}.
\newblock Available at math.uchicago.edu/{$\sim$}amathew/tmfhomotopy.pdf.

\bibitem{MP}
J.~P. May and K.~Ponto.
\newblock {\em More Concise Algebraic Topology: Localization, Completion, and
  Model Categories}.
\newblock Chicago Lectures in Mathematics. The Unviersity of Chicago Press,
  Chicago and London, 2012.

\bibitem{May77}
J.~P. May and F.~Quinn.
\newblock $e_\infty$ ring space and $e_\infty$ ring spectra.
\newblock {\em Lecture Notes in Mathematics}, 577, 1977.

\bibitem{McCleary}
J.~McCleary.
\newblock {\em A User's Guide to Spectral Sequences}.
\newblock Cambridge Studies in Advanced Mathematics. Cambridge University
  Press, 2 edition, 2000.

\bibitem{Mimura}
M.~Mimura.
\newblock Homotopy theory of {L}ie groups.
\newblock In I.~M. James, editor, {\em Handbook of Algebraic Topology}, page
  951–991. North-Holland, 1995.

\bibitem{MT}
M.~Mimura and H.~Toda.
\newblock Homotopy groups of {$SU(3)$}, {$SU(4)$}, and {$Sp(2)$}.
\newblock {\em J. Math. Kyoto Univ.}, 3(2):217--250, 1963.

\bibitem{OSS}
C.~Okenek, M.~Schneider, and H.~Spindler.
\newblock {\em Vector Bundles on Complex Projective Spaces}.
\newblock Birkhauser Verlag, 1980.

\bibitem{Quill}
D.~Quillen.
\newblock The {A}dams conjecture.
\newblock {\em Topology}, 10:67--80, 1971.

\bibitem{Sull}
D.~Sullivan.
\newblock Genetics of homotopy theory and the {A}dams conjecture.
\newblock {\em Ann. of Math.}, 100:1--79, 1974.

\bibitem{Switzer2}
R.~M. Switzer.
\newblock Complex $2$-plane bundles over complex projective space.
\newblock {\em Math. Z.}, 168:275--287, 1979.

\bibitem{Switzer}
R.~M. Switzer.
\newblock Rank $2$ bundles over {$P^n$} and the e-invariant.
\newblock {\em Indiana University Math. J.}, 28(6):961--974, 1979.

\bibitem{Thomas}
A.~Thomas.
\newblock Almost complex structures on complex projective spaces.
\newblock {\em Trans. Amer. Math. Soc.}, 193:123--132, 1974.

\bibitem{Zab}
A.~Zabrodsky.
\newblock {\em Hopf Spaces}.
\newblock North-Holland Publishing Company, Amsterdam, New York, Oxford, 1976.

\end{thebibliography}

\end{document}